\pdfoutput=1

\documentclass[12pt,english,reqno]{amsart}
\usepackage{amscd,amssymb,amsmath,amstext,amsthm,amsfonts}
\usepackage[dvips]{graphicx}
\usepackage{mathrsfs}
\usepackage{float}%Para as figuras ficarem fixada [H]

\usepackage{a4wide}
\newcommand{\co}{\mbox{$\mathcal{O}$}}

\newcommand{\BQ}{\begin{quote}}
\newcommand{\EQ}{\end{quote}}

\newcommand{\ITEM}{\begin{description}}
\newcommand{\offITEM}{\end{description}}

\newtheorem{definition}{Definition}

\usepackage[pdftex]{hyperref}%%%PARA FAZER LINK NO PDF, mas para MiKTeKs antigos devemos desmarcar para o DVI%%%
\usepackage[usenames]{color}
\usepackage[ansinew]{inputenc}

\newtheorem{Theorem}{Theorem}
\newtheorem*{theorem}{Theorem}

\newtheorem{maintheorem}{Theorem}

\newtheorem{maincorollary}[maintheorem]{Corollary}

\newtheorem{T}{Theorem}[section]
\newtheorem{Corollary}[T]{Corollary}
\newtheorem{Proposition}[T]{Proposition}
\newtheorem{Lemma}[T]{Lemma}

\newtheorem{Notation}[T]{Notation}
\newtheorem{Remark}[T]{Remark}

\newtheorem{Definition}[T]{Definition}
\newtheorem{Example}[T]{Example}
\newtheorem*{claim}{Claim}
\newtheorem{Claim}{Claim}

\def \AA {{\mathbb A}}

\def \RR {{\mathbb R}}
\def \ZZ {{\mathbb Z}}
\def \NN {{\mathbb N}}

\def \EE {{\mathbb E}}

\def \UU {{\mathbb U}}
\def \XX {{\mathbb X}}
\def \YY {{\mathbb Y}}
\def \QQ {{\mathbb Q}}

\def \cl {\mathcal{L}}
\def \cf {\mathcal{F}}

\def \cp {\mathcal{P}}
\def \cc {\mathcal{C}}

\def \cu {\mathcal{U}}
\def \co {\mathcal{O}}
\def \cn {\mathcal{N}}
\def \ck {\mathcal{K}}

\def \cw {\mathcal{W}}

\def \cR {\mathscr{R}}

\newcommand{\dem}{\begin{proof}}
\newcommand{\cqd}{\end{proof}}

\newcommand{\leb}{\operatorname{Leb}}

\newcommand{\period}{\operatorname{period}}
\newcommand{\interior}{\operatorname{interior}}

\newcommand{\diameter}{\operatorname{diameter}}

\begin{document}

\author{Paulo Brand\~ao}
\address{Impa, Estrada Dona Castorina, 110, Rio de Janeiro, Brazil.} \email{paulo@impa.br}

\date{\today}

\thanks{Work carried out at IMPA. Partially supported by CNPq, IMPA and PROCAD/CAPES}

%PARA SAIR DO MODO RASCUNHO, TEM Q cortar os capitulos de \section{Relation between $\rho$ and $\widetilde{J}$} (inclusive) ate o ponto exatamente antes das bifurcacoes.

\title{On The Structure of Lorenz Maps}

\maketitle

\begin{abstract}
%Our aim in this work is to give some contribution to Palis Conjecture on the the finitude, metric
%stability and continuity of the basin of attraction of dynamical systems in the setting of geometric Lorenz attractors of contractive type. 
%
%We look for distortion control on nice intervals to understand the structure of the hyperbolic sets, in a similar approach to the one that firstly obtained this result in the quadratic family setting, bringing also new lights to this former context.

%We show the finitude of ergodic components for any $C^3$ contracting Lorenz map $f$ negative Schwarzian derivative with periodic points accumulating in the critical point.
%As a consequence, if $f$ doesn't have a Cherry attractor, it has no
%wandering homterval, as conjectured by Martens and de Melo.

We study the non-wandering set of $C^3$ contracting Lorenz maps $f$ with negative Schwarzian derivative. We show that if $f$ doesn't have   attracting periodic orbit, then there is a unique topological attractor. Precisely, there is a transitive compact set $\Lambda$ such that $\omega_f(x)=\Lambda$ for a residual set of points $x \in [0,1]$. We also develop in the context of Lorenz maps the classical theory of spectral decomposition constructed for Axiom A maps by Smale.

\end{abstract}

\tableofcontents

%{lista de coisas ainda pra fazer}
%
%
%
%citar st pierre
%
%
%citar o wild attractor
%
%
%
%
%
%(MANUSCRITO DESENVOLVIDO PRA ATRATORES METRICOS PRA BYPASSAR ST PIERRE)
%
%
%A
%
%A' VIA EXPANSAO DE CROSS RATIO
%
%A" VIA INTERVALO DO MEIO
%
%
%COLOCAR TB A PROVA SEM SER POR ST PIERRE DE Q
%Se existe orbita periodica atratora entao omega(x) esta contido em uma orbita periodica atratora para qtp x.[ ], ou seja, orb per atratora eh ou uma ou duas.
%
%coisa da rotacao
%
%solenoide 4.5
%
%------
%
%
%========
%
%
%TRESSER

\newpage

\section{Introduction}

In the famous work of Lorenz \cite{Lor}, studying the solution of the following system of differential equations~(\ref{eqlorenz}) in $\mathbb{R}^3$ originated by truncating Navier-Stokes equations for modeling atmospheric conditions

\begin{eqnarray}
\label{eqlorenz}
\dot{x} & = & -10 x  + 10 y  \\
\dot{y} & = & 28 x  -y -xz \nonumber\\
\dot{z} & = & - \frac{8}{3} z  + x y  \nonumber 
\end{eqnarray}

\noindent
it appears for the first time what was thought to be an attractor with characteristics that became, a posteriori, the ones that defines a
``strange attractor'' (Nevertheless, for a long time, no one proved that this attractor in the original Lorenz work exhibits in fact these characteristics. It was later shown by Tucker \cite{Tu99} that it does).

Then, V.S. Afraimovich, V.V. Bykov, L.P. Shil'nikov, in \cite{ABS} and Guckenheimer and  Williams, in \cite{GW}, introduced Lorenz-like attractors that are models similar to Lorenz's one that exhibited, in fact, the same peculiar characteristics of the Lorenz attractor.

This model consists of considering a hyperbolic singularity with
one dimensional unstable manifold such that, in a linearizable neighborhood, the separatrices can be considered as one of the coordinate axes, say $x$, and that returns to this neighborhood cutting transversally the plane $z = constant$, with the eigenvalues $\lambda_2 < \lambda_3< 0 < \lambda_1 $ (see Figure~\ref{LorenzFluxo} ), and the expanding condition $ \lambda_3 + \lambda_1 > 0 $.

In \cite{GW} Guckenheimer and Williams show that if you fix a system with such an attractor then, in a neighborhood $U$ you have structural stability in $ codim \, 2 $, and in any representative family of types you have always one attractor attracting the neighborhood constructed in figure~\ref{LorenzFluxo},  so the finitude of attractors is guaranteed, as required by Palis conjecture.

%%%%%%%%%%%%%%%%%%%%%%%%%%%%%%%%%%%%%%%%%%%%%%
\begin{figure}
  \begin{center}\includegraphics[scale=.25]{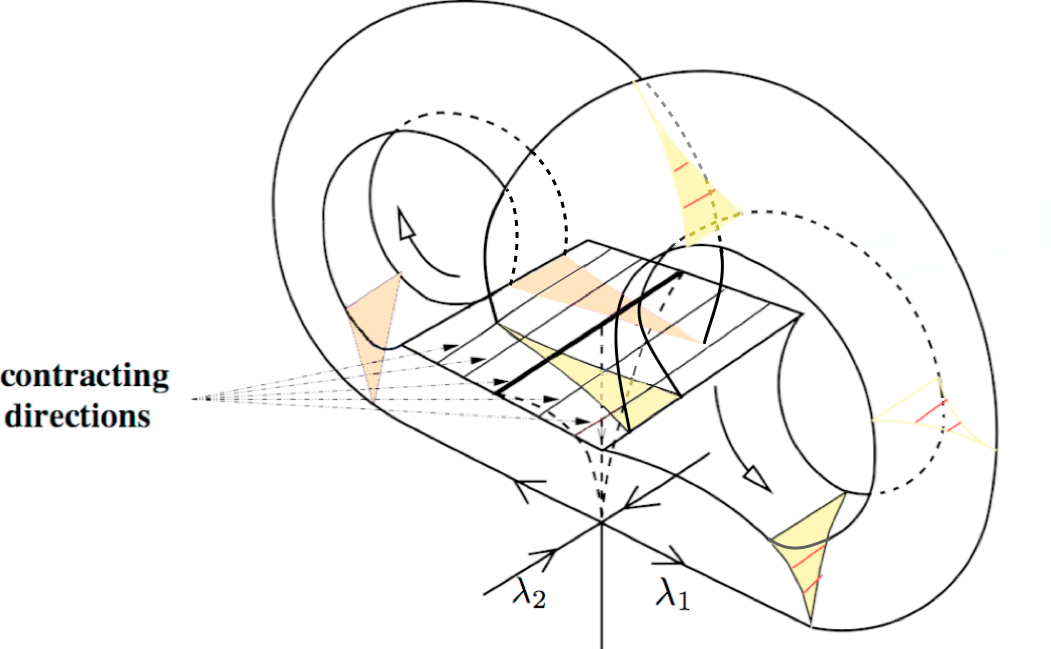}\\
  \caption{Dynamics of the Flow}\label{LorenzFluxo}
  \end{center}
\end{figure}
%%%%%%%%%%%%%%%%%%%%%%%%%%%%%%%%%%%%%%%%%%%%%%

In \cite{ACT}, Arneodo, Coullet and Tresser present an attractor obtained in the same way as the one obtained by Guckenheimer and Williams, just modifying the relation between the eigenvalues of the singularity, taking  $ \lambda_3 + \lambda_1 < 0 $ . Then we have a contracting Lorenz attractor.

We can consider the Poincar\'e map of the square

$Q=\left\{ \left| x\right| \leq cte;\left| y\right| \leq cte;z=cte\right\} $ into itself, having the returns as in figure~\ref{duasfig} .

%%%%%%%%%%%%%%%%%%%%%%%%%%%%%%%%%%%%%%%%%%%%%%
\begin{figure}
  %Requires \usepackage{graphicx}
  \begin{center}\includegraphics[scale=.3]{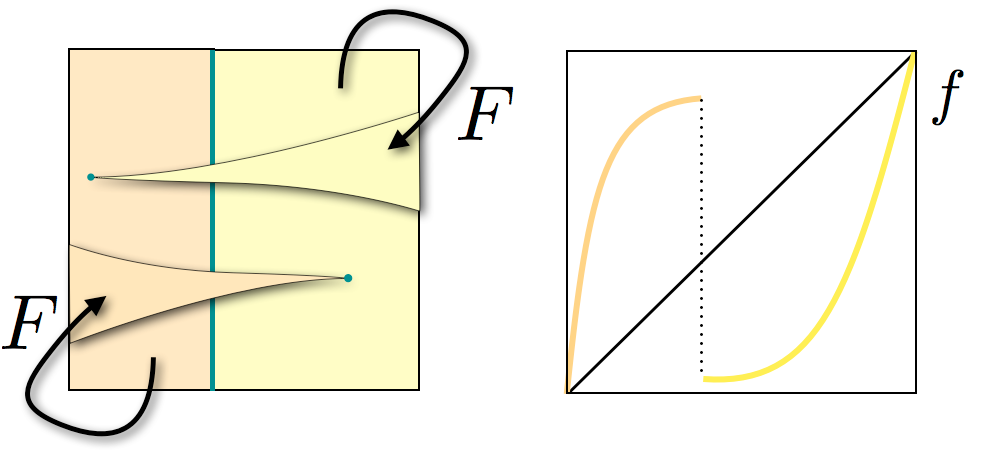}\\
  \caption{Poincaré and Induced One dimensional Maps}\label{duasfig}
  \end{center}
\end{figure}
%%%%%%%%%%%%%%%%%%%%%%%%%%%%%%%%%%%%%%%%%%%%%%

On the other hand, we can exhibit in $Q$ a foliation by one dimensional leaves, invariant by this Poincar\'e map, and such that this contracts exponentially its leaves. Then we have the induced one dimensional map like $f$ in Figure~\ref{duasfig}, that motivated the approach using the ideas in the work of Benedicks and Carleson \cite{BC1} adopted by Rovella to show in the \mbox{2-parameter} space a positive measure set of transitive chaotic attractors alternating their existences with open and dense sets of hyperbolic ones (see \cite{Rov93} for details).

\subsection{One dimensional Lorenz Maps}
We say that a $C^2$ map $f:[0,1]\setminus \{c\} \rightarrow[0,1]$, $0<c<1$, is a {\em Lorenz map} if $f(0)=0$, $f(1)=1$, $f'(x)>0$ $\forall\,x\in[0,1]\setminus \{c\}$. A Lorenz map is called {\em contracting} if $\lim_{x\to c}f'(x)=0$.

In the study of an expansive Lorenz map (or even a Lorenz map with bounded $\log|f'|$, see \cite{BM}), one can bypass most of the difficulties using the expansivity (or the distortion control in the case that $\log|f'|$ is bounded). For example, in these cases there are no wandering intervals. For a contracting Lorenz map, as $\lim_{x\to c}f'(x)=0$, the discontinuity at the critical point together with derivative equal to zero make the study of its dynamics much more complicated.

%%%%%%%%%%%%%%%%%%%%%%%%%%%%%%%%%%%%%%%%%%%%
\begin{figure}
\begin{center}
\includegraphics[scale=.3]{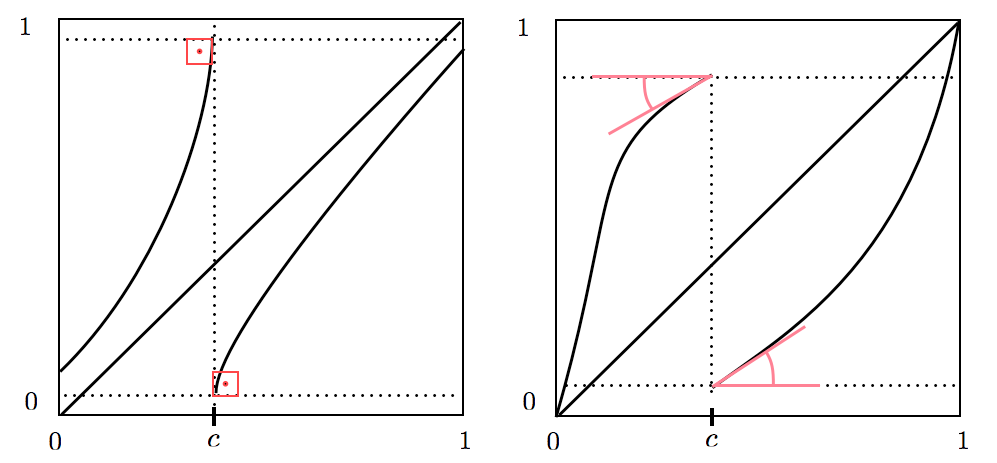}\\
\caption{}\label{LorenzOutrosTipos.png}
\end{center}
\end{figure}
%%%%%%%%%%%%%%%%%%%%%%%%%%%%%%%%%%%%%%%%%%%%%%

We will focus here in the study of the non-wandering set $\Omega(f)$ of the contractive Lorenz maps (see picture \ref{LorenzContracting.png}) and also of its topological attractors. 

%%%%%%%%%%%%%%%%%%%%%%%%%%%%%%%%%%%%%%%%%%%%
\begin{figure}
\begin{center}\label{LorenzContracting.png}
\includegraphics[scale=.2]{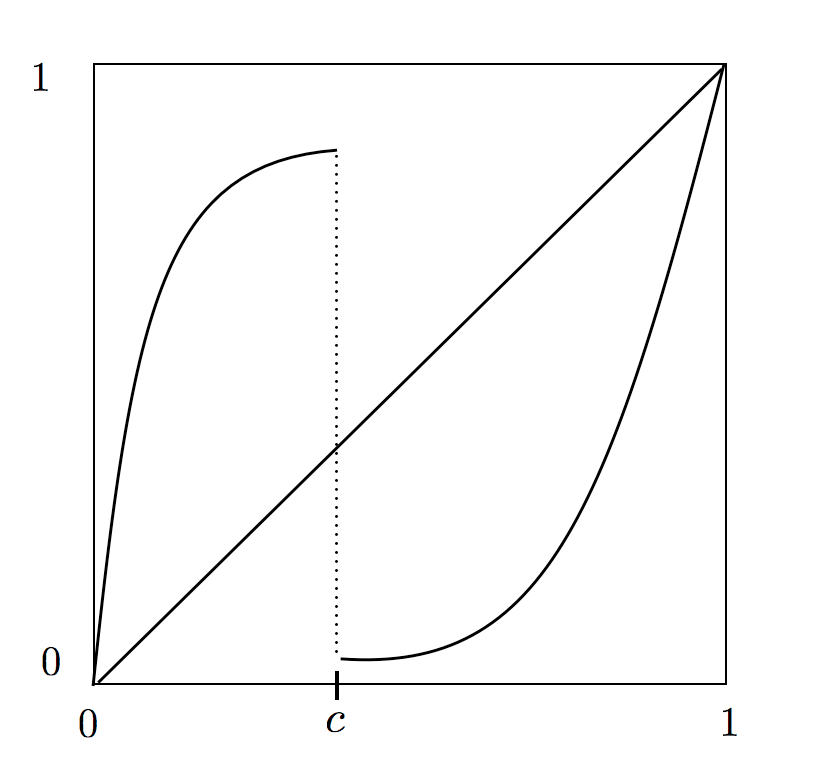}\\
\caption{}\label{LorenzContracting.png}
\end{center}
\end{figure}
%%%%%%%%%%%%%%%%%%%%%%%%%%%%%%%%%%%%%%%%%%%%%%

For maps of the interval, critical points and  critical values play fundamental roles in the study of dynamics. From this point of view, Lorenz maps are of hybrid type. Indeed, these maps have a single critical point, as the unimodal maps, but two critical values, as the bimodal ones.

To point out the ambivalence of these maps (with respect to unimodal and bimodal maps) let's consider two settings. 
The first one, is when we have some expansion in the orbits of the critical values (for example, a Collet-Eckmann condition). In this case we may expect the existence of absolutely continuous invariant measures $\mu$ with respect to Lebesgue measure. In this case, it is not difficult to show that the attractor will be the support of $\mu$ and it will be an orbit of interval.

This orbit of interval will contain the critical point in its interior. Furthermore, as in the unimodal case, there will be a single metrical attractor.

Indeed, the basin of this attractor is an open and dense set (in particular, residual) and has full measure.

%\begin{Notation}

%\end{Notation}

The second setting for Lorenz maps we want to discuss as a motivation is when we do not observe any expansion along the orbit of the critical values. Now, we may think that the points that come close to the critical point follow the critical orbits for a long time. As in this case the critical orbits are recurrent (or yet more, for example suppose we don't have any weak repeller nor any attracting periodic point, so that we know, by Ma\~ne's Theorem, that $c\in\omega_f(x)$ for Lebesgue almost every $x$) before they go very far from the orbit of the critical points these points will likely go again near the critical point, where it will have a strong contraction and come close again the orbit of the critical value.
That is, one may expect that $\omega_f(x)\subset\overline{\co^+_f(v_0)\cup\co^+_f(v_1)}$. 
Whenever  $\omega_f(v_0)\ne\omega_f(v_1)$, it is reasonable to expect that part of the points follow $\omega_f(v_0)$ and other part follow $\omega_f(v_1)$, building two attractors.

This reasoning can be reinforced by the following toy model, introduced to us by Charles Tresser, the {\em Plateau maps}. 

%%%%%%%%%%%%%%%%%%%%%%%%%%%%%%%%%%%%%%%%%%%
\begin{figure}
\begin{center}\label{plateauUNIMODAL}
\includegraphics[scale=.3]{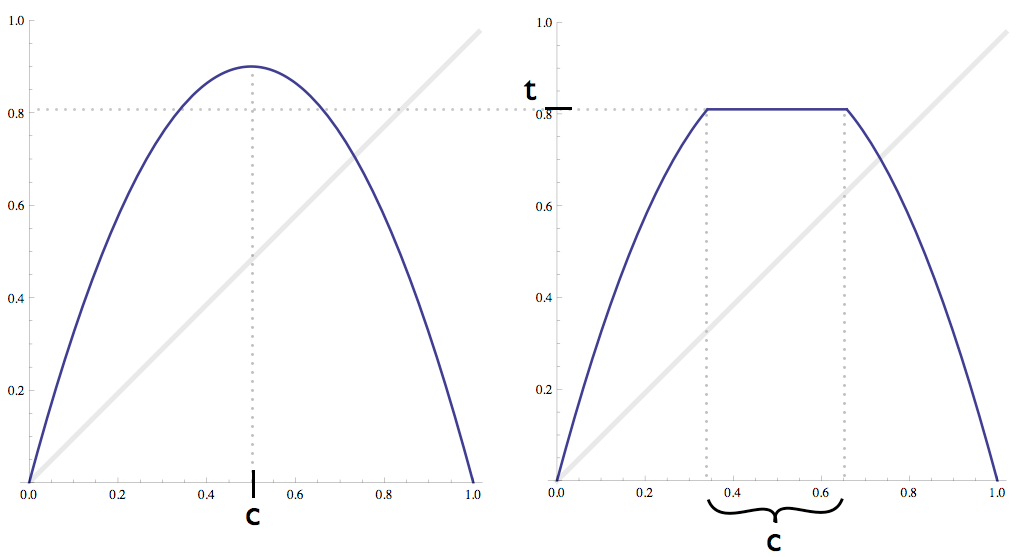}\\
\caption{}\label{plateauUNIMODAL}
\end{center}
\end{figure}
%%%%%%%%%%%%%%%%%%%%%%%%%%%%%%%%%%%%%%%%%%%

In the plateau model, the critical region is an interval.  For example, if $f$  is an unimodal map without any weak repeller and as in the picture (\ref{plateauUNIMODAL}), we can produce a plateau map by taking $f_t(x)=\min\{t,f(x)\}$.
The dynamics of an ``unimodal'' plateau map  split in two different dynamical blocks:
the set $\Lambda$ of the points whose orbits avoid the critical region $\cc(f_t)=[c_0,c_1]:=f^{-1}(t)$ and the set $\Gamma$ of the points that eventually fall in the critical region, i.e., $\Gamma:=\{x\,;\,\co_{f_t}^+(x)\cap\cc(f_t)\ne\emptyset\}$.
Using Ma\~ne's Theorem (as we have taken $f$ without any weak repeller), we obtain that Lebesgue almost every point belongs to $\Gamma$.
Thus, $\omega_{f_t}(x)=\omega_{f_t}(t)$ for Lebesgue almost all $x\in[0,1]$. Moreover, $\Gamma$ is residual on $[0,1]$ (indeed, the interior of $\Gamma$ is an open and dense set). 
This means that $\AA:=\omega_{f_t}(t)$ is at the same time the topological and the metrical attractor for $f_t$.

In 1979, Guckenheimer (\cite{G79}) proved the unicity of the topological attractors for the S-unimodal maps. The unicity of the metric attractors was proved by Blokh and Lyubich in the end of the 80s (in (\cite{BL89}) and (\cite{BL91})). 
This way, the number of attractors both topological and metric obtained for the unimodal case matches the number of attractors of the corresponding toy model.

The dynamics of ``Lorenz'' plateau maps will decompose in three blocks: (1) the set $\Lambda$ of points whose orbits avoid the critical region $\cc(f_t):=[c_0,c_1]$, (2) the set of points thats eventually go into the left side of the critical region and follow this critical value, i.e., $\Gamma_0:=\{x\,;\,\co_{f_t}^+(x)\cap[c_0,c)\}$ and (3) the points that visit the right side and follow its critical value, that is, $\Gamma_1:=\{x\,;\,\co_{f_t}^+(x)\cap(c,c_1]\}$. By Mañe, $\leb(\Lambda)=0$. On the other hand, $\leb(\Gamma_0)>0<\leb(\Gamma_1)$. Thus, if $\omega_{f_t}(c_0)\ne\omega_{f_t}(c_1)$ we will get two attractors $\Lambda_0:=\omega_{f_t}(c_0)$ and $\Lambda_1:=\omega_{f_t}(c_1)$, each one having a basin containing an open set.
Futhermore, one can choose the values $t_0$ and $t_1$ in such a way that both 
$\Gamma_0$ and $\Gamma_1$ are distinct Cantor sets, providing a situation in which more than simple periodic orbits, we have two non-periodic distinct attractors that attract distinct positive measure sets. %Two basins of non-periodic attractors, that we will see not occur to Lorenz maps, where the only case in which we can decompose the basin in two non-trivial pieces is in the case of occurrence of a pair of periodic attractors.

%%%%%%%%%%%%%%%%%%%%%%%%%%%%%%%%%%%%%%%%%%
\begin{figure}
\begin{center}
\includegraphics[scale=.35]{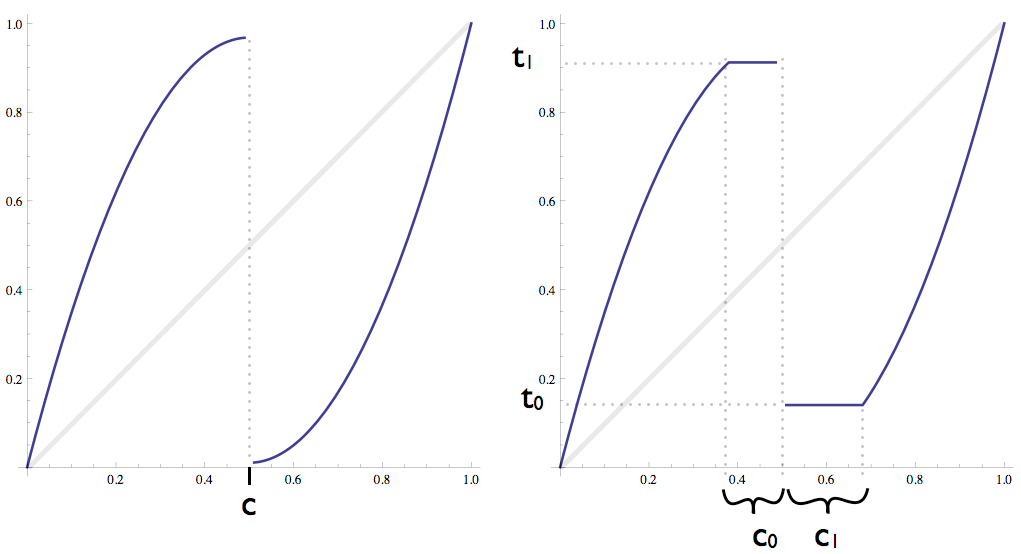}\\
\caption{}\label{plateauLORENZ}
\end{center}
\end{figure}
%%%%%%%%%%%%%%%%%%%%%%%%%%%%%%%%%%%%%%%%%%

Despite of this indication, we prove in Theorem~\ref{atratortopologico} that also in this second setting, there is one single topological attractor (at least, if $f$ has negative Schwarzian derivative), that is, the behavior of the contracting Lorenz looks like the unimodal maps, instead of the bimodal maps, that admits up to two attractors. 

In the sequel we will be devoted to the development of the understanding of the dynamics of Lorenz maps in a very classical sense, namely, in terms of the decomposition of their non-wandering sets.

A point $x$ is said to be {\em non-wandering} if for any neighborhood $U \ni x$, $\exists n \ge 1$ such that $f^n(U)\cap U\ne \emptyset$. The set of all non-wandering points is the {\em non-wandering set} $\Omega(f)$. 

Much of the understanding of the qualitative behavior of dynamical systems follows from the description of Axiom A diffeomorphisms  satisfying the no-cycle condition. These diffeomorphisms have non-wandering sets that decompose into a finite number of transitive sets and admit a filtration, and using this,  Smale \cite{Sm} has shown their stability. The same kind of structure was constructed for unimodal maps by Jonker and Rand \cite{JR}, that have non-wandering sets decomposed in a similar way (although possibly in an infinite countable number of pieces).

One may observe that this kind of result in the context of unimodal or multimodal maps is usually obtained applying Milnor-Thurston Theorem, that is no longer valid in our case, as we don't have a continuous map. So, we had to construct an alternative way to prove transitivity instead of using the semi-conjugacy with a piecewise affine map that is usually used for unimodal or multimodal maps. 

\newpage

\section{Main Results}\label{MainResults}

\begin{Definition}[Lorenz Maps]
We say that a $C^2$ map $f:[0,1]\setminus \{c\} \rightarrow[0,1]$, $0<c<1$, is a {\em Lorenz map} if $f(0)=0$, $f(1)=1$, $f'(x)>0$ $\forall\,x\in[0,1]\setminus \{c\}$. A Lorenz map is called {\em contracting} if $\lim_{x\to c}f'(x)=0$.
\end{Definition}

Given $n\ge1$, define $f^{n}(c_{-})=\lim_{x\uparrow c}f^{n}(x)$ and $f^{n}(c_{+})=\lim_{x\downarrow c}f^{n}(x)$. The critical values of $f$ are $f(c_{-})$ and $f(c_{+})$.

Given $x\in[0,1]\setminus\{f(c_{-}),f(c_{+})\}$, set the pre-image of $x$ as usually, that is, $f^{-1}(x)=\{y\in[0,1]\,;\,f(y)=x\}$. If $x\in\{f(c_{-}),f(c_{+})\}$ set $f^{-1}(x)=\{c\}\cup\{y\in[0,1]\,;\,f(y)=x\}$. Given a set $X\subset[0,1]$, define $f^{-1}(X)=\bigcup_{x\in X}f^{-1}(x)$. Inductively, define $f^{-n}(x)=f^{-1}(f^{-(n-1)}(x))$, where $n\ge2$.
The {\em pre-orbit} of a point $x\in[0,1]$ is the set $\co_{f}^{-}(x):=\bigcup_{n\ge0}f^{-n}(x)$, where $f^{0}(x):=x$.

Denote the positive orbit of a point $x\in[0,1]\setminus\co_{f}^{-}(c)$ by $\co_f^+(x)$, i.e., $\co_{f}^{+}(x)=\{f^j(x);j\ge0\}$. If $x\in\co_{f}^{-}(c)$, let $\co_{f}^{+}(x)=\{x,f(x),\cdots,f^{m_{x}-1}(x),c\}$, with $f^{m_{x}}(x)=c$. Define $\co_{f}^+(c_{-})=\{f^{n}(c_{-})\,;\,j\ge0\}$ and also $\co_{f}^+(c_{+})=\{f^{n}(c_{+})\,;\,j\ge0\}$.

The set of accumulation points of the positive orbit of $x\in[0,c)\cup\{c_{-},c_{+}\}\cup(c,1]$ is denoted by $\omega_f(x)$, the $\omega$-limit set of $x$.  The $\alpha$-limit set of $x$, $\alpha_{f}(x)$, is the set of points $y$ such that $y=\lim_{j\to\infty}x_{j}$ for some sequence $x_{j}\in f^{-n_{j}}(x)$ with $n_{j}\to+\infty$.

We say $I$ is {\em a wandering interval} of $f$ if $f^i(I)\cap f^j(I)=\emptyset$ for $i\ne j >0$ and the $\omega$-limit set of $I$ is not equal to a single periodic orbit.

%Define a {\em homterval} as being an interval on which $f^j$ is monotone $\forall j \ge 0$.

%Define the {\em Stable Set} of a point $p$ as the set $W^s(p):=\{x;d(f^j(x),f^j(p))\to0\}$.

Following Milnor in \cite{Milnor:1985ut}, we say that a  compact set $A$ is a {\em topological attractor} if 
its basin $\beta(A) = \{x; \omega_{f}(x) \subset A\}$ is residual in an open set  and also that each closed forward invariant subset $A'$ which is strictly contained in $A$ has a topologically smaller basin of attraction, i.e., $\beta(A) \setminus \beta(A')$ is residual in an open set. 

\begin{Definition}[Periodic Attractor]\label{PeriodicAttractor}
A periodic attractor is a finite set $\Lambda$ such that $\interior(\{x\,;\,\omega_{f}(x)=\Lambda\})\ne\emptyset$.
\end{Definition}

There are two types of periodic attractors: periodic orbits and the super-attrctors. A super-attractor is a finite set $\Lambda=\{p_{1},\cdots,p_{n},c\}$ such that $f(p_{i})=p_{i+1}$ for $1\le i<n$, $f(p_{n})=c$ and $\lim_{0<\varepsilon\downarrow0}f(c+\varepsilon)=p_{1}$ or $\lim_{0<\varepsilon\downarrow0}f(c-\varepsilon)=p_{1}$.

Given a periodic point $p$, say $f^n(p)=p$, we say that this periodic orbit $\co_f^+(p)$ is an {\em attracting periodic orbit} if $\exists \epsilon >0$ such that $(p,p+\epsilon)$ or $(p-\epsilon,p)\subset \beta(\co_f^+(p))$.

A {\em weak repeller} is a periodic point $p$ of $f$ such that it is non-hyperbolic and it is not a periodic attractor.

\begin{Definition}[Chaotic attractors]
We say that an attractor (topological or metrical) $\Lambda$ is a {\em chaotic attractor} if $\Lambda$ is transitive,  periodic orbits are dense in it ($\overline {Per(f)\cap \Lambda}=\Lambda$), its topological entropy $h_{top}(f|_\Lambda)$ is positive and $\exists \lambda >0$ and a dense subset of points $x \in \Lambda$ such that their {\em Lyapounov exponents}, $\exp_f(x)$, are greater than $\lambda$, where
\begin{equation}\label{eqExpLy}\exp_f(x):=\liminf \frac{1}{n}\log |Df^n(x)|.\end{equation}
\end{Definition}

\begin{Definition}[Wild attractors]
We say that an attractor (topological or metrical) $\Lambda$ is {\em wild} if $\Lambda$ is transitive but is not the maximal transitive set that contains it.  
%$\widetilde{\Lambda}\supsetneqq \Lambda$ . 
\end{Definition}

In \cite{Milnor:1985ut}, Milnor has given examples where the metrical attractor contains the topological one and vice-versa (see examples 5 and 6 of the Appendix of \cite{Milnor:1985ut}). He also asked if the topological and metrical attractors were always the same for unimodal maps with negative Schwarzian derivative. For $S$-unimodal maps, Milnor already knew from Guckenheimer's work \cite{G79} that if they are not the same, the metrical attractor has to be inside the topological one  and, as both are transitive, the metrical attractor must be a wild attractor by the definition above.
For quasi-quadratic maps (i.e., $f$ with $Sf<0$ and non-degenerate critical point: $|f''(c)|\ne0$, see details in \cite{Ly94}), Lyubich has shown in 1994 that the topological attractors obtained by Guckenheimer coincide with the metrical ones obtained in (\cite{BL89}). The question if the coincidence of topological and metrical attractors was a general phenomenon was open till 1996 when Bruin, Keller, Nowicki and van Strien \cite{BKNvS} exhibit examples of unimodal maps of the type $\lambda(1-|2x-1|^{\ell})$ (with big $\ell$) having metrical attractor distinct to the topological one.

We say that a Lorenz map $f$ is a {\em Cherry map} if there is a  neighborhood $J$ of the critical point such that the first return to $J$ is semi-conjugated to an irrational rotation. In this case, there is a minimal compact set $\Lambda$ such that $\omega_{f}(x)=\Lambda\ni c$ $\forall x\in J$. A {\em Cherry attractor} is a minimal compact set contained in the interior of its basin of attraction and such that $rot_{f}(x)=\rho\notin\QQ$ $\forall\,x\in\beta(\Lambda)$, where $rot_{f}(x)$ is the rotation number of $x$.

%
%A {\em restrictive interval} of pair of associated periods $(p,q)$ for $f$ is a closed interval $J\subset[0,1], c \in J$ with $J=J^-\cup\{c\}\cup J^+$, $J^-=J\cap(0,c)=[j_0,c)$, $J^+=J\cap(c,1)=(c,j_1]$ such that $J$, $f(J^-), ... , f^{p-1}(J^-)$ and $f(J^+), ... , f^{q-1}(J^+)$ have disjoint interiors; $f^{p}(J^-)\subset  J$ and $f^{q}(J^+)\subset  J$, $f^{p}(j_0)=j_0$ and $f^{q}(j_1)=j_1$.

A renormalization interval for $f$ is an open interval $J=(a,b)\ni c$ such that the first return map to $(a,b)$ is conjugated to a Lorenz map. The points of the boundary of a renormalization interval $J=(a,b)$ are always periodic points and $$f^{\period(a)}([a,c))\subset[a,b]\supset f^{\period(b)}((c,b]).$$
%
%Definition. Let f : I ! I be a multimodal map. A closed proper subinterval J of I is called restrictive with period n   1 for f if
%1. the interiors of J, . . . , fn 1(J) are disjoint;
%2. fn(J) ? J, fn(@J) ? @J;
%3. at least one of the intervals J, . . . , fn 1(J) contains a turning point;
%4. J is maximal with respect to these properties: if J0   J is a closed interval which is strictly contained in I and such that the previous properties also hold for J0 (for the same integer n) then J0 = J.
%In the unimodal case, restrictive intervals are also called central. We say that J is a maximal restrictive interval if there exists no restrictive interval J0 (whose period might be di?erent from the period of J) which strictly contains J. The
%reason to introduce this notion is that it allows us to consider pieces of the dynamics on a finer scale:
%
%, and $J$ is maximal with this property. We say $f$ is {\em renormalizable}  and $r(x) =f^{p}(x);{x<0}$ and $f^{q}(x);{x>0}$ is the {\em renormalization} or {\em return map} of $f$ to $J$.
%obs chamei de restritivo aqui o central. welington define como restritivo qq um da orbita do mesmo 
An attractor $\Lambda$ of a contracting Lorenz map $f$ is a {\em Solenoidal attractor} (or {\em Solenoid})  if  $\Lambda\subset\bigcap_{n=0}^\infty K_n$, where $K_n=\bigg(\bigcup_{j=0}f^{\period(a_{n})}([a_{n},c))\bigg)\cap\bigg(\bigcup_{j=0}^{\period(b_{n})}f^{j}((c,b_{n}])\bigg)$ and $\{J_n=(a_{n},b_{n})\}_{n}$ is an infinite nested chain of renormalization intervals. 

% TRESSER SUGERE CHAMAR ESTES ATRATORES SOLENOIDAIS DE "THUR", POIS SERIA QUEM PRIMEIRO OS ESTUDOU

A Contracting Lorenz map $f:[0,1]\setminus\{c\} \to [0,1]$ is called {\em non-flat} if there exist constants $\alpha>1$, $a,b\in[0,1]$ and $C^2$ orientation preserving  diffeomorphisms $\phi_{0}:[0,c]\to[0,a^{1/\alpha}]$ and $\phi_{1}:[c,1]\to[0,b^{1/\alpha}]$ such that $$f(x)=\begin{cases}a-(\phi_{0}(c-x))^\alpha&\text{ if }x<c\\
1-b+(\phi_{1}(x))^\alpha&\text{ if }x>c\end{cases}.$$

\begin{maintheorem}
\label{baciastopologicas}
Let $f$  be a non-flat contracting Lorenz map having neither weak repellers nor periodic attractors. Then there is a compact set $\Lambda$ that is transitive, positively invariant (indeed $f(\Lambda)=\Lambda$) such that its basin $\beta(\Lambda)$ is a residual set in the whole interval. We have that $\Lambda$ is one and only one of the following types:

\begin{enumerate}

\item {\em Cherry attractor} and in this case $\omega_{f}(x)=\Lambda$ in an open and dense set of points $x\in [0,1]$. 
\item {\em Solenoidal attractor} and in this case $\omega_{f}(x)=\Lambda$ in a residual set of points $x\in [0,1]$. 
\item {\em Chaotic attractor} that can be of two kinds:
\begin{enumerate}
\item {\em Finite union of intervals}, in this case $\omega_{f}(x)=\Lambda$ in a residual set of points $x\in [0,1]$. 
\item {\em Cantor set} and in this case there are wandering intervals. 

\end{enumerate}

\end{enumerate}

\end{maintheorem}

\begin{maintheorem}
\label{teoalfalim}
Let $f$ be a non-flat contracting Lorenz map without weak repellers nor periodic attractors. Then $f$ has a single topological attractor~$\Lambda$. Moreover,  
$f$ has no wandering interval if and only if $\alpha_{f}(x)=[0,1], \forall x \in \Lambda$

\end{maintheorem}

Next theorem goes deeper in the classification provided by Theorem~\ref{baciastopologicas}, as it distinguishes two possible situations for item (3)(b) of that theorem, as it didn't state the Cantor set $\Lambda$ was $\omega_{f}(x)$ for a residual set of $x \in [0,1]$, but only that its basin contains such a set. Indeed this is not generally true, as we may construct an example of a wild attractor $\Lambda'$ as defined above, contained but different to $\Lambda$ that (3)(b) says that exists. That is, situation (3)(b) splits in two possibilities. Either $\Lambda$ attracts a residual set whose $\omega$-limit coincides with $\Lambda$, or it properly contains another Cantor set that has this property. And we want to emphasize it by stating the two possibilities separately.

We say that a $C^3$ map $f$  has negative Schwarzian derivative if its Schwarzian derivative, $Sf$, is negative in every point $x$ such that $Df(x)\ne0$, where

\begin{equation}\label{defschwarz}
Sf(x)=\frac{D^3 f(x)}{D f(x)}-\frac{3}{2}\bigg( \frac{D^2 f(x)}{D f(x)} \bigg)^2
\end{equation}

\begin{maintheorem}
\label{atratortopologico}
Let $f$ be a $C^3$ non-flat contracting Lorenz map with negative Schwarzian derivative.
If $f$ has a periodic attractor $\Lambda$ then either $\beta(\Lambda)$ is an open and dense set with full Lebesgue measure or there is one more periodic attractor $\Lambda'$ such that $\beta(\Lambda\cup\Lambda')$ is open and dense with full Lebesgue measure.

If $f$ does not have periodic attractor then there is a single transitive topological attractor $\Lambda$ with $\omega_{f}(x)=\Lambda$ for a residual set of points $x \in [0,1]$ and it is one of the following types:

\begin{enumerate}
\item {\em Cherry attractor}. 
\item {\em Solenoidal attractor}. 
\item {\em Chaotic attractor of one of the two following kinds:}

\begin{enumerate}

\item {\em Finite union of intervals} 
\item {\em Cantor set} (this case occurs only if there is a wandering interval).

\end{enumerate}
\item {\em Wild Cantor attractor} (also only with the occurrence of wandering intervals). 

\end{enumerate}

\end{maintheorem}

At this point it worths mentioning that (metrical, at least) wild Cantor attractors do occur, as we know we can relate a symmetric unimodal map to a Lorenz map (see Section~\ref{SecEmbeding} of the Appendix), translating most of the features of one into the other, and so we can build the example given by \cite{BKNvS} for a Lorenz map.

\begin{maincorollary}
\label{atrattopologicoXmetrico}
Let $f$ be a $C^3$ non-flat contracting Lorenz map with negative Schwarzian derivative without periodic attractors. 

\begin{enumerate}
\item The topological attractor contains the metrical one.
\item If the topological attractor is not a cycle of intervals then the topological attractor and the metrical one coincide.
\end{enumerate}

\end{maincorollary}

Finally we obtain the stratification of the non-wandering set for the Lorenz maps, $\Omega(f)$.

\begin{maintheorem}[Spectral Decomposition of Lorenz Maps]\label{spdec}
Let $f$ be a $C^{3}$  non-flat contracting Lorenz map with negative Schwarzian derivative. Then, there is $n_f \in \NN \cup \{\infty\}$, the number of strata of the decomposition of $f$, and a collection of compact sets $\{\Omega_{n}\}_{0\le n\le n_{f}}$ such that:
\begin{enumerate}
\item $\Omega(f)=\bigcup_{0\le j \le n_f}\Omega_j$.
\item Each $\Omega_n$ is a forward invariant set {\em ($f(\Omega_{n})=\Omega_{n}$)} and 

\begin{itemize}

\item[(i)] 
$\Omega_0  = \begin{cases}
[0,1] & \text{ if $f(c_{+})=0$ and $f(c_{-})=1$} \\
\{0\} & \text{ if $f(c_{+})>0$ and $f(c_{-})=1$}\\
\{1\} & \text{ if $f(c_{+})=0$ and $f(c_{-})<1$}\\
\{0,1\} & \text{ if $f(c_{+})>0$ and $f(c_{-})<1$}
\end{cases}$

\item $\Omega_{n}\cap\Omega_{m}=\emptyset$  $\forall\,0\le n\ne m<n_{f}$;
\item if $\Omega_{n}\cap\Omega_{m}\ne\emptyset$ for $n<m$ then $(n,m)=(n_{f}-1,n_{f})$ and in this case $\Omega_{n_{f}-1}\cap\Omega_{n_{f}}$ consist of one or two periodic orbits and $c_{-}$ or $c_{+}$ are pre-periodic.

%EH ISSO MESMO??

\item[(ii)] $f$ restricted to $\Omega_j$ is topologically transitive for all $0 <j < n_f$.

\item[(iii)] $\Omega_{n_f}$ is either a transitive set or the union of a pair of attracting periodic orbits.
\end{itemize}

%\item The topological entropy of f, $h_{top}(f)$ is zero if and only if $\Omega_j$ consists of periodic orbits for every finite $j\le n_f$.

%\item  If $\Omega_j$ contains non-periodic points then $f^{a_J p_j+b_j q_j}:\Omega_{j+1}\to\Omega_{j+1}$ is $\Omega$-semi-conjugate to a piecewise linear map with constant slopes (and the semi-conjugacy is almost injective, see the proof below).

\item
For each $0<j<n_{f}$ there is a decomposition of $\Omega_j$ into closed sets $$\Omega_j= X_{0,j}\cup X_{1,j}\cup\cdots\cup X_{\ell_j,j}$$ such that $\# ( X_{a,j}\cap X_{b,j} )\le1$ when $a\ne b$, the first return map to $X_{0,j}$ is topologically exact (in particular, topologically mixing) and,
for each  $i\in\{1,\cdots,\ell_{j}\}$, there is  $1\le s_{i,j}\le \ell_{j}$ such that $f^{s_{i,j}}(X_{i,j})\subset X_{0,j}$.

\end{enumerate}

\end{maintheorem}

\subsubsection*{Related works on contracting Lorenz maps}

The initial results on contracting Lorenz maps (and contracting Lorenz flows) are from beginning of the decade of 1980. Among the main names in this decade and the first half of the 1990's, we cite C. Tresser, A. Arneodo, L. Alsed\`a, A. Chenciner, P. Coullet, J-M. Gambaudo, M. Misiurewicz, A. Rovella, R.F. Williams (see \cite{ACT,CGT84,GT85,Gambaudo:1986p2422,Tresser:1993uf,Rov93}).
From the end of the 20th century until now the main contributions  are from  M. Martens and W. de Melo \cite{MM} and G. Keller and M. St. Pierre \cite{StP}. We also can cite results on the existence of SRB-measures (R. Metzger \cite{Metzger:2000vp}), combinatorial structure of contracting Lorenz maps (Labarca and Moreira \cite{Labarca:2010p1505,Labarca:2006p1486}), thermodynamical formalism (M.J. Pac\'ifico and M. Todd \cite{Pacifico:2010cc}), construction of contracting Lorenz maps (and flows) in higher dimensions (V. Araujo, A. Castro, M.J. Pac\'ifico, and V. Pinheiro \cite{Araujo:2011p2744}).

\newpage

\newpage

\section{Some Useful Results on Schwarzian Derivative}

We begin this section by stating some known results that will be useful in the sequel. Recall the definition of Schwarzian derivative given in equation (
\ref{defschwarz}).

An important property of maps with negative Schwarzian derivative is the Minimum Principle:
\begin{Proposition}[Minimum Principle, see \cite{MvS}] \label{MinP} Let T be a closed interval with end-points $a,b$ and $f:T \to \mathbb{R}$ a map with negative Schwarzian derivative ($Sf<0$). If $Df(x)\ne 0$ for all $x\in T$ then $$|Df(x)|>\min\{|Df(a)|,|Df(b)|\},\ \forall x\in (a,b).$$
\end{Proposition}

One of the consequences of the Minimum Principle is that maps with negative Schwarzian derivative do not admit weak repellers:

\begin{Corollary}[See \cite{MvS}] \label{weakrepellers}
If $U\subset\RR$ is an open set and $f:U\to\RR$ is a $C^{3}$ map with $Sf<0$ then $f$ does not admit weak repellers.
\end{Corollary}

\begin{Theorem}[Singer]\label{ThSinger} If $I\subset\RR$ is an interval and $f:U\to I$ is a $C^{3}$ map with negative Schwarzian derivative then
\begin{enumerate}
\item the immediate basin of any attracting periodic orbit contains either a critical point of f or a boundary point of  $I$;
\item each neutral periodic point is attracting;
\item there exists no interval of periodic points.
\end{enumerate}
In particular, the number of non-repelling periodic orbits is bounded if the number of critical points of f is finite. 
\end{Theorem}

The corollary below is a slight generalization of Singer's result for maps defined on an interval.

\begin{Corollary}\label{CoSinger} Let $U\subset I$ be an open subset of an interval $I=[a,b]$ containing $\partial I$.  Let $f:U\to I$ be a $C^{3}$ map with negative Schwarzian derivative such that $f(\partial I)\subset \partial I$ and $\lim_{U\ni x\to y}f'(x)=0$ $\forall\,y\in\partial U\setminus\partial I$. If we define the critical set of $f$, $\cc_{f}$, by $$\cc_{f}=\{y\in\overline{U}\,;\,\lim_{U\ni x\to y}f'(x)=0\}$$  then we obtain the same conclusion of Singer's Theorem, that is,
\begin{enumerate}
\item the immediate basin of any attracting periodic orbit contains either a critical point of f or a boundary point of  $I$;
\item each neutral periodic point is attracting;
\item there exists no interval of periodic points.
\end{enumerate}
In particular, the number of non-repelling periodic orbits is bounded if the number of critical points of f  is finite. 
\end{Corollary}
\dem If $p$ is an attracting periodic orbit of period $n$, let $T\ni p$ be the connected component of the domain of $f^{n}$ (i.e., $Dom(f^{n}):=\bigcap_{j=0}^{n-1}f^{-j}(U)$) that contains $p$. As $p$ is an attracting fixed point to $f^{n}|_{T}$, we can apply Singer's Theorem and obtain that the immediate basin of $p$ contains a point $q\in\partial T$ or $c\in\cc_{f^{n}}=\{y\in\overline{U}\,;\,\lim_{U\ni x\to y}(f^{n})'(x)=0\}$. Of course, if  $q\in\partial T$ belongs to the immediate basin of $p$ for $f^{n}|_{T}$ then $q$ also belongs to the immediate basin of $\co_{f}^{+}(p)$ with respect to $f$. In the second case, the chain rule says that $\lim_{U\ni x\to c}f'(f^{j}(x))=0$ for some $0\le j<n$. Thus, $c':=\lim_{U\ni x\to c}f^{j}(x)\in\cc_{f}$ and it is in the immediate basin of $p':=f^{j}(p)\in\co_{f}^{+}(p)$ with respect to $f$.

The items 2 and 3 follow also from applying Singer's Theorem to the restriction of $f^{n}$ to the connected components of $Dom(f^{n})$.

\cqd

\begin{Corollary} If $f$ is a $C^{3}$ contracting Lorenz map with negative Schwarzian derivative then $f$ has, at most, two periodic attractors.
\end{Corollary}

Combining Singer's Theorem above with Mañe's Theorem (see Section \ref{transpgap}) (and also the fact that every expanding set in the interval has zero Lebesgue measure), it easily follows that if a $S$-unimodal map $f:[0,1]\to[0,1]$ has a periodic attractor then its basin is open and dense in $[0,1]$ and has full Lebesgue measure. In contrast, as a contracting Lorenz map $f:[0,1]\setminus\{c\}\to[0,1]$ can not be continuously extended to the interval, the proof that for a contracting Lorenz map with negative Schwarzian derivative the existence of a periodic attractor implies that Lebesgue almost every point is attracted to a periodic attractor does not follows easily from the theorems cited above. This result can be found in St. Pierre's Ph.D. thesis (\cite{StP}).

\begin{Theorem}[\cite{StP}]
\label{stpierreperiodic}
Let $f:[0,1]\setminus\{c\}\to[0,1]$ be a $C^{3}$ non-flat contracting Lorenz map with negative Schwarzian derivative. If $f$ has a periodic attractor, then there is an open and dense set $B\subset[0,1]$ with full Lebesgue measure ($\leb(B)=1$) such that every $x\in B$ belongs to the basin of some periodic attractor.
\end{Theorem}

\newpage

\section{The Dynamics of $c$-phobic Sets}\label{phobic}

\subsection{The Transport of Gaps}\label{transpgap}

Given an interval $I$, let $\mathring{I}$ be its interior (when we say ``an interval'' we will be assuming that it is non-trivial, i.e., with non-empty interior). If $ \exists k \ge0 $
such that $c\in f^k ({\mathring{I}})$,
define $\theta(I)$, the {\em order} of $I$, by $$\theta(I)=\min\{k \ge 0\,|\,c\in f^k ({\mathring{I}})\}.$$
If there is no such $k$, let  $\theta(I)= \infty$.

\begin{Theorem}[Mañe \cite{Man85}]\label{ThMane}
Let $f:[a,b]\to\RR$ be a $C^{2}$ map. If $U$ is an open neighborhood of the critical points of $f$ and every periodic point of $[a,b]\setminus U$ is hyperbolic and expanding then $\exists\,C>0$ and $\lambda>1$ such that $$|Df^{n}(x)|>C \lambda^{n}$$ for every $x$ such that $\{x,\cdots,f^{n-1}(x)\}\cap U=\emptyset$.
\end{Theorem}

\begin{Theorem}[See Theorem 2.6 in Chapter III of \cite{MvS}]\label{lebzero} Every uniformly expanding set of a $C^{1+}$ map of the interval has zero Lebesgue measure.
\end{Theorem}

\begin{Lemma}\label{Remark98671oxe}Let $f:[0,1]\setminus\{c\}\to[0,1]$ be a $C^{2}$ contracting Lorenz map. If $f$ does not have any weak repeller nor  periodic attractor then $$c\in\omega_{f}(x)\text{ for Lebesgue almost every }x.$$
Furthermore, if $J\subset[0,1]$ is an open set with $c\in J$ then $\Lambda_{J}:=\{x\,;\,\co^{+}_{f}(x)\cap J=\emptyset\}$ is an uniformly expanding set and $\leb(\Lambda_{J})=0$.
Also, $\#(\co^{+}_{f}(x)\cap J)=\infty$ for Lebesgue almost all $x$.
\end{Lemma}

\dem Let $J\subset[0,1]$ be an open set with $c\in J$.
Consider any $C^{2}$ map $g:[0,1]\to[0,1]$ such that $g|_{[0,1]\setminus J}=f|_{[0,1]\setminus J}$ (see Figure~\ref{Mane}). By Mañe's theorem $\Lambda_{J}$ is a uniformly expanding set for $g$. Thus, $\Lambda_{U}$ is a uniformly expanding set for $f$, because $g|_{\Lambda_{J}}=f|_{\Lambda_{J}}$. By Theorem~\ref{lebzero}, $\leb(\Lambda_{J})=0$.

Now we can state that in this setting $c \in \omega_{f}(x)$ for Lebesgue almost all $x \in [0,1]$. This follows from the fact that $\forall x$ such that $ c\not\in \omega_{f}(x)$, $\exists J \ni c$ such that $x \in \Lambda_J$. The set of all points whose $\omega$-limit do not contain $c$ is $\bigcup_{J\ni c}\Lambda_J$ for some $J$. By definition, if $K \supset J \ni c$, $\Lambda_J \supset \Lambda_K$, then if we take a sequence of nested intervals $J_n$ containing $c$ with length going to zero, say, $J_n=(c-1/n,c+1/n)$ then every J in the collection that defined $\bigcup_{J\ni c}\Lambda_J$ contains some interval of the type $J_n$ and as $J \supset J_n$ implies $\Lambda_J \subset \Lambda_{J_n}$, $\bigcup_{J\ni c}\Lambda_J \subset \bigcup_{J_n, n\in \NN}\Lambda_{J_n}$, a countable union of zero Lebesgue measure sets, hence $\bigcup_{J\ni c}\Lambda_J$ has Lebesgue measure zero.

Also, $\#(\co^{+}_{f}(x)\cap J)=\infty$ for Lebesgue almost all $x$. Consider $x_1$ the first point in the orbit of $x$ that visit $J$, consider $U_1$ to be any interval of the kind $J_n=(c-1/n,c+1/n)$ that doesn't contain $x_1$, then consider $x_2$ to be the first point of the orbit of $x_1$ that visit $U_1$, and so on, we build a sequence of different points in the orbit of $x$ converging to $c$.
\cqd

%%%%%%%%%%%%%%%%%%%%%%%%%%%%%%%%%%%%%%%%%%%%%%
\begin{figure}[H]
  \begin{center}\includegraphics[scale=.23]{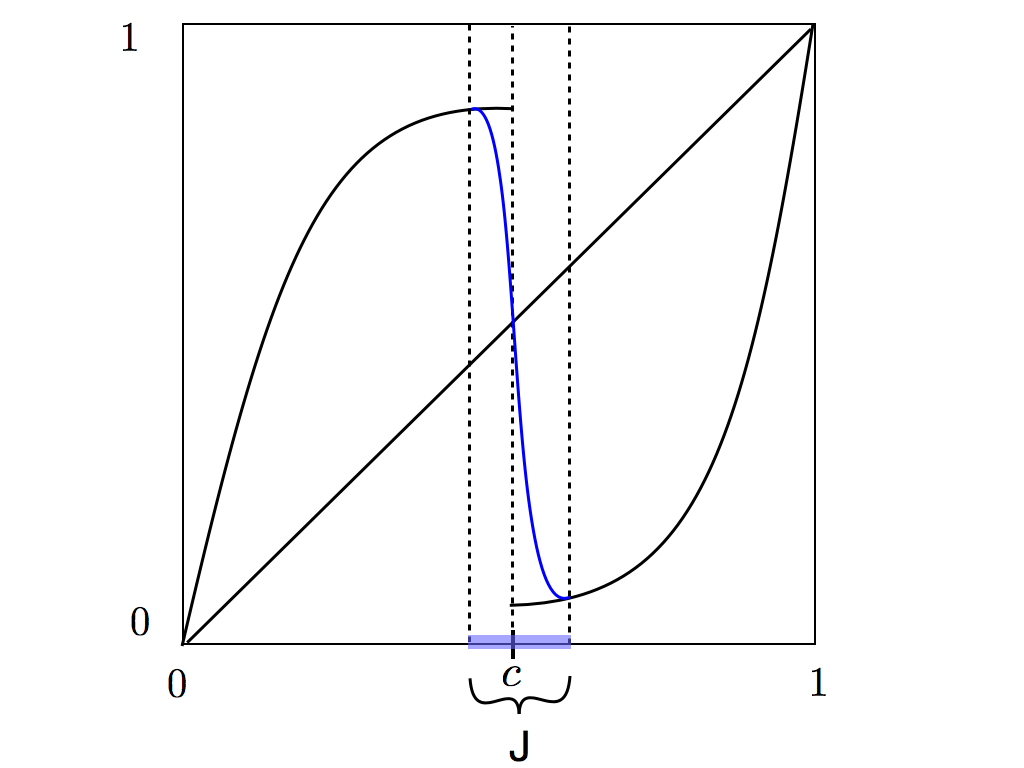}\\
 \caption{}\label{Mane}
  \end{center}
\end{figure}
%%%%%%%%%%%%%%%%%%%%%%%%%%%%%%%%%%%%%%%%%%%%%%

\begin{Definition}[$c$-Phobic sets]\label{Def535353}
Given $J\subset[0,1]$ an open set with $c\in J$ we will call $\Lambda_{J}=\{x\,;\,\co^{+}_{f}(x)\cap J=\emptyset\}$ the {\em J-phobic set} of $[0,1]$, the expanding set generated by $J$ that avoids it, and the {\em c-phobic set} of $[0,1]$ the set of points that do not infinitely approach $c$: $\Lambda_{c}=\{x; c\not\in\omega_{f}(x)\}$ , and we just saw both have Lebesgue measure zero.
\end{Definition}

An open interval $I=(a,b)$ containing the critical point $c$
is called a  {\em nice interval} of $f$ if
$\co_{f}^{+}(\partial I)\cap I=\O$ (where $\co_{f}^{+}(X)$ denotes the positive orbit of $X$ by $f$, that is,
$\co_{f}^{+}(X)=\bigcup_{x \in X}\{f^j(x), j \in \NN \}$).

We will denote the {\em set of nice intervals of $f$} by
$\mathcal{N}=\mathcal{N}(f)$ and the set of nice intervals whose
border belongs to the set of periodic points of $f$ by
$\mathcal{N}_{per}=\mathcal{N}_{per}(f)$, that is,
$\mathcal{N}_{per}=\{I\in\mathcal{N}\ \|\ \partial I\subset Per(f)\}$.

\begin{Definition}[Gaps generated by a nice interval]\label{Def765091}Given a nice interval $J$, let $C_J$ be the set of connected components of $[0,1] \setminus \Lambda_{J}$. An element of $C_{J}$ is called a {\em gap of $\Lambda_J$} or {\em a gap generated by $J$}.
\end{Definition}

\begin{Proposition}\label{Prop765091}Let $J$ be a nice interval of a $C^{2}$ contracting Lorenz map $f:[0,1]\setminus\{c\}\to[0,1]$. If $f$ does not have neither periodic attractors nor weak repellers then:
\begin{enumerate}
\item $\Lambda_{J}$ is a uniformly expanding set, $\leb(\Lambda_{J})=0$ and $\leb(\bigcup_{I\in C_{J}}I)=1$;
\item $\theta(I)=\min\{\ell\ge0\,;\,f^{\ell}(I)\cap J\ne\emptyset\}$ $\forall\,I\in C_{J}$;
\item $f^{\theta(I)}|_{I}$ is a diffeomorphism and $f^{\theta(I)}(I)=J$ $\forall\,I\in C_{J}$.
\end{enumerate}
\end{Proposition}
\dem The first item follows from Lemma~\ref{Remark98671oxe}.
Given $I\in C_{J}$, let $n=\min\{j\ge0\,;\,f^{j}(I)\cap J\ne\emptyset\}$ and let $$i(\ell)=\begin{cases}0&\text{ if }f^{\ell}(I)\subset[0,c)\\
1&\text{ if }f^{\ell}(I)\subset(c,1]
\end{cases}$$ be the itinerary of $I$ for $0\le\ell<n$. Set $T_{0}=[0,c)$ and $T_{1}=(c,1]$. As $f^{n}|_{I}=f|_{f^{n-1}(I)}\circ f|_{f^{n-2}(I)}\circ\cdots\circ f|_{I}$ and $f|_{f^{\ell}(I)}=f|_{T_{i(\ell)}}|_{f^{\ell}(I)}$ is a diffeomorphism $\forall0\le\ell<n$ then $f^{n}|_{I}$ is a diffeomorphism.

Write $(a,b)=I$. As $a,b\in\Lambda$ and $f(\Lambda)\subset\Lambda$, we get $f^{n}(a),f^{n}(b)\notin J$. On the other hand, $c\in f^{n}(I)=(f^{n}(a),f^{n}(b))$. So $f^{n}(I)\supset J$. Suppose that $f^{n}(I)\supsetneqq J$. Thus, $\exists\,p\in f^{n}(I)\cap\partial J$. In this case let $q=(f^{n}|_{I})^{-1}(p)\in I$. Note that $f^{\ell}(q)\cap J=\emptyset$ $\forall0\le\ell<n$, as $f^{\ell}(q)\in f^{\ell}(I)$. Furthermore, as $J$ is a nice interval, $f^{j}(\partial J)\cap J=\emptyset$ $\forall\,j\ge0$. So $f^{\ell}(q)\notin J$ $\forall\,j\ge0$. That is, $q\in\Lambda_{J}$. But this is impossible, as $q\in I$ and $I$ is a connected component of $[0,1]\setminus\Lambda_{J}$. Thus, $f^{n}(I)=J$. Finally, as $c\in f^{n}(I)$ and $c\notin f^{\ell}(I)$ $\forall0\le\ell<n$ (because $f^{\ell}(I)\cap J=\emptyset$ $\forall0\le\ell<n$), we get that $n=\theta(I)$. 
\cqd

\subsection{Cylinders, branches, extensions and etc.}\label{SecCyBrEx}   

Given an interval $I$,  If $\theta(I)< \infty$, we define $\widehat{I}$, the {\em branch domain} for $I$,
as the closure of the maximal open interval $\mathcal{I}$ containing the interior of $I$ such that $f^{\theta(I)} |_{\mathcal{I}}$ is monotone.
If $\theta(I)= \infty$, define $\widehat{I}$ as the closure of the maximal interval $T$ containing the interior of $I$ such that $f^k |_{T} $ is monotone for every $k\ge0$.

When $W\subset[0,1]$ is such that $\theta(W)<\infty$ we define the {\em branch} of $f$ associated to $W$,
denoted by $F_W$, as the continuous extension of $f^{\theta(W)}|_{\mathcal{W}}$ to $\widehat{W}$, where $\mathcal{W}$ is the interior of $\widehat{W}$. Of course that $\widehat{\widehat{W}}=\widehat{W}$ and so,
$$F_{W}=F_{\widehat{W}},$$ for all (non trivial) interval $W$.

We will denote the image of a map $g$ by $Im(g)$. Of course that $$Im(F_{I})=Im(F_{\widehat{I}})=F_{I}(\widehat{I})=F_{\widehat{I}}(\widehat{I}).$$

If $J\in\cn$, let us denote  $$\widehat{C_J} = \{\widehat{I}\, | \, I \in C_J\}.$$

If $p \in Per(f)$ define $I(p)$ as the maximal interval containing $p$ such that $f^{period(p)} |_{I(p)} $ is monotone.

Given a periodic {\em nice} interval $J=(a,b)\in\mathcal{N}_{per}$
let $\cu_{J}$ be the set of all periodic {\em nice} interval
$(a',b')\in\mathcal{N}_{per}$ containing the closure of $J$, $[a,b]$,
such that $period(a')\le period(a)$ and $period(b')\le period(b)$.

\begin{Definition}[$\xi J$: Root of a Periodic Nice Interval J]\label{defxidej}
Let $J=(a,b)\in\mathcal{N}_{per}$ and suppose that $\# Per_{n}(f)<+\infty$ $\forall\,n\ge1$, where $Per_{n}(f)=\{p\,;\,f^{n}(p)=p\}$. Define the {\em root} of $J$ denoted by $\xi J$ as the
the smallest interval belonging to
$\cu_{J}$, i.e., $\xi J$ is the interval $I\in\cu_{J}$ such that $I\subset V$ for all $V\in\cu_{J}$. Of course that $$\xi J=\bigcap_{V\in\cu_{J}}V.$$
\end{Definition}

Set $\xi(0,1)=(0,1)$. Thus, whenever there is only
a finite number of periodic points for a given
period, the root of any periodic {\em nice} interval is
well defined.

\subsection{On the structure of $c$-phobic sets}

In this section we will study the set $\Lambda_{J}$, $J=(p,q)\in\cn_{per}$ when $p$ and $q$ belong to the same orbit. In this case, as we will show, the set $\Lambda_{J}$ is not a perfect set. Indeed $\forall I \in C_J$, $\partial I$ consists of isolated points of $\Lambda_J$. Furthermore, we will characterize $\xi J$ with the combinatorial information of $\partial J$ and show that if a point $x\in\xi J\setminus J$ does not belongs to the pre-orbit of $p$ then $x$ eventually falls into $J$, i.e., $\co^{+}_{f}(x)\cap J\ne\emptyset$.

\begin{Lemma}
\label{23/02/07} Let $f:[0,1]\setminus\{c\}\to[0,1]$ be a $C^{2}$ contracting Lorenz map.
If $J=(p,q) \in \mathcal{N}_{per}$ and $\Lambda_{J}$ does not contain periodic attractors or weak repellers then the following statements are equivalent:
\begin{enumerate}
\item $\Lambda_J$ is not a perfect set;
\item $\exists I\in C_J\setminus\{J\}$ such that $\partial I \cap \partial J \ne \emptyset$;
\item $\co_{f}^{+}(p)=\co_{f}^{+}(q)$;
\item $\exists (\tilde{p},\tilde{q})\in C_J$ such that $\co_{f}^{+}(\tilde{p})=\co_{f}^{+}(q)$ or $\co_{f}^{+}(p)=\co_{f}^{+}(\tilde{q})$;
\item $\forall I \in C_J$, $\partial I$ consists of isolated points of $\Lambda_J$.
\end{enumerate}
\end{Lemma}

\dem
%23/02/07

(1) $\implies$ (2). If $\Lambda_J$ is not perfect, there exists $x \in \Lambda_J$, $x$ an isolated point. So, $\exists W_0=(a,x)$ and $W_1=(x,b)$, $W_0, W_1 \in C_J$. Let's consider $s=\theta(W_0)<\theta(W_1)=u$. In this case, $f^s(W_0)=J=(p,q)$ and $f^s(W_1)=(q,f^s(b))$ is as required. The case $s>u$ is analogous.

%CORTEI ESSA FRASE ABAIXO Q TAVA EM CIMA. DESNECESSARIA
%In the case $s=\theta(W_0)>\theta(W_1)=u$, we would have $f^u(W_0)=(f^u(a),p)$ and $f^u(W_1)=(p,q)=J$.

(2) $\implies$ (3). Suppose that $I=(a,p) \in C_J$ (the case $I=(q,a)$ is analogous). As $f$ is orientation preserving and $f^{\theta(I)}(I)=J$ we have that  $f^{\theta(I)}(p)=q$ and then $\co_{f}^{+}(p)=\co_{f}^{+}(q)$.

(3) $\Longleftrightarrow$ (4). We have only to check that (4) implies (3), because $J\in C_J$.
As $p$ and $q$ are periodic points, if $\co_{f}^{+}(\tilde{p})=\co_{f}^{+}(q)$ (the other case is analogous) then $\tilde p$ is a periodic point.
As $f$ is orientation preserving where monotonous, $(f^n(\tilde p),f^n(\tilde q))=f^n((\tilde p,\tilde q))=(p,q)$, where $n=\theta((\tilde p,\tilde q))$.
Thus $\co_{f}^{+}(q)=\co_{f}^{+}(\tilde p)=\co_{f}^{+}(p)$.

(3) $\implies$ (5). Suppose that $I=(a,b) \in C_J$ and that $a$ is not an isolated point (the case of $b$ is analogous). Then we have that there exists a sequence $I_n=(p_n,q_n) \in C_J$ such that $p_n<q_n<a$ and $lim \,  p_n = a$. We also have that there is a neighborhood $V$ of $[a,b]$ such that $f^{\theta(I)}|_V$ is a diffeomorphism. For every $n$ big enough we have that $I_n \subset V $ and then $I'_n=(p'_n,q'_n)=(f^{\theta(I)}(p_n),f^{\theta(I)}(q_n))=f^{\theta(I)}(I_n)\in C_J$. Observe that $p'_n<q'_n<p$ and $lim \, p'_n=p$. Let $l$ be such that $f^l(p)=q$ and $W$ is a neighborhood of $p$ such that $f^l|_W$ is a diffeomorphism. To every $n$ big enough we will have that $I'_n \subset W$. Observe also that to $n$ sufficiently big, $p<f^l(p'_n)<f^l(q'_n)<f^l(p)=q$. That is, $f^l(I'_n) \subsetneqq J$, what is an absurd, as $I'_n \in C_J$.

(5) $\implies$ (1). Straightforward.

\cqd

\begin{Corollary}Let $f:[0,1]\setminus\{c\}\to[0,1]$ be a $C^{2}$ contracting Lorenz map.
If $J=(p,q)\in\cn_{per}$ and $\Lambda_{J}$ does not contain periodic attractors or weak repellers then $\Lambda_{J}$ is a Cantor set if and only if $\co^{+}_{f}(p)\ne\co^{+}_{f}(q)$.
\end{Corollary}

It follows from Lemma \ref{23/02/07} that if $J=(a,b) \in \mathcal{N}_{per}$ with $\co_{f}^{+}(a)=\co_{f}^{+}(b)$ that $a$ and $b$ are isolated points of $\Lambda_{J}$. So, there is some $(a_{-1},a)\in C_{J}$, i.e., $(a_{-1},a)$ is the connected component of $[0,1]\setminus\Lambda_{J}$ with $a$ in its boundary and different to $J$. Analogously,  there is some $(a_{-2},a_{-1})\in C_{J}$ and so on.  
Thus, there is a sequence $a_n$ such that $...<a_{-(m+1)}<a_{-m}<...<a_{-1}<a_0=a<b=a_1<a_{2}<...<a_n<a_{n+1}<...$ and $A_n=(a_n,a_{n+1})$ belong to $C_J, \forall n \in \ZZ $. About these we can state the following

\begin{Theorem}[The structure of $C_{J}$ and $\xi J$ for $J$ imperfect]\label{imperfeito} Let $f:[0,1]\setminus\{c\}\to[0,1]$ be a $C^{2}$ contracting Lorenz map. 
Let $J=(a,b) \in \mathcal{N}_{per}$  be such that $\Lambda_{J}$ does not contain periodic attractors or weak repellers. Let $(\alpha, \beta )=\xi J$. If $\co_{f}^{+}(a)=\co_{f}^{+}(b)$ then
\begin{enumerate}
\item  $period(\alpha)=\min \{u | f^u(A_{-1})=A_j, j \geq 0 \}$;
\item $I(\alpha)\supset[\alpha,a]$;
\item $F_{I(\alpha)}([\alpha,a])\supset [\alpha,b]$;
\item $F_{I(\alpha)}(A_{n})=A_{n+\gamma+1}$ $\forall n\le-1$, where $\gamma\ge0$ is given by $A_{\gamma}=F_{I(\alpha)}(A_{-1})$;
\item if $I\in C_J$ and $I\subset(\alpha,a)$ then $\exists\,i\in\{0,...,\gamma\}$ and $n\ge1$ such that
$F_I=F_{A_i}\circ (F_{I(\alpha)})^n$;
\item $\{I\in C_J$ $|$ $I\subset(\alpha,a)\}$ $=$ $\{A_n\}_{n\le-1}$.
\end{enumerate}
Analogously we have
\begin{enumerate}
\item[($1'$)] $period(\beta)=\min \{v | f^v(A_{1})=A_k, k \leq 0 \}$;
\item[($2'$)] $I(\beta)\supset[b,\beta]$;
\item[($3'$)] $F_{I(\beta)}([b,\beta])\supset [a,\beta]$;
\item[($4'$)]  $F_{I(\beta)}(A_{n})=A_{n+\rho-1}$ $\forall n\ge1$, where $\rho\le0$ is given by $A_{\rho}=F_{I(\beta)}(A_{1})$;
\item[($5'$)] if $I\in C_j$ and $I\subset(b,\beta)$ then $\exists\,i\in\{\rho,...,0\}$ and $n\ge1$ such that
$F_I=F_{A_i}\circ (F_{I(\beta)})^n$;
\item[($6'$)]  $\{I\in C_J$ $|$ $I\subset(b,\beta)\}$ $=$ $\{A_n\}_{n\ge1}$.
\end{enumerate}
\end{Theorem}

\dem
%23/02/07

Let $l=min\{k \geq 0 | f^k(A_{-1})=A_j,$ for some $j \ge0\}$ and $\gamma\ge0$ such that $f^l(A_{-1})=A_{\gamma}$.
As $A_{-1} \in C_J$, $f^{\theta(A_{-1})}$ is a diffeomorphism in a neighborhood of $\overline{A_{-1}}$, and as
$l \leq \theta(A_{-1})$, we have that $f^l$ is a diffeomorphism in a neighborhood of $\overline{A_{-1}}$.

\begin{claim}
$f^l$ is a diffeomorphism in a neighborhood of $[a_{-k},a_0]$ $\forall\,k\ge1$ and also, $f^l(A_{-j})=A_{\gamma+1-j}$, $\forall\,j \ge 1$.
\end{claim}
\dem We have seen this holds for $k=1$ and we will show by induction that if it holds for $n$ it also holds for $n+1$.

Suppose this wasn't true. Then, $\exists\,q \in A_{-(n+1)}$ and $s<l$ such that $f^s(q)=c$.

In this case, as $f^i(A_{-n})\cap J=\emptyset$ $\forall\,0\le i < \theta(A_{-n})$ (because $A_{-n}\in C_J$), either $b \in f^s(A_{-(n+1)})$ or $b=f^s(a_{-n})$.
In the first situation, it follows that $f^{l-s}(b) \in f^l(A_{-(n+1)})$ and so, $f^{l-s+\theta(f^l(A_{-(n+1)}))}(b)\in J$, what contradicts the fact that $J$ is nice and $b$ never visits $J$. The second situation implies $f^s(A_{-n})= A_{1}$. As $s<l$, and as by the induction hypothesis $f^l$ is a diffeomorphism in a neighborhood of $A_{-n}$, so is $f^s$, and as $A_{-n},...,A_{-1}$ is a sequence of contiguous intervals such that $f^s(A_{-n})= A_{1}$ and $A_{-(n-1)}$ is in the right side of $A_{-n}$, then $f^s(A_{-(n-1)}) \cap A_{2} \neq \emptyset$ ($A_{2}$ is in the right side of $A_{1}$). As all these $A_j$ and their images by $f$ before reaching $J$ are components of $C_J$, if they intersect any other, they are the same component, so $f^s(A_{-(n-1)})=A_{2}$. Successively we have that $f^s(A_{-1})= A_{n}, n>0 $, what denies the minimality of $l$. Then both cases lead to an absurd.

The contiguousness argument provides $f^l(A_{-j})=A_{\gamma+1-j}$ as stated.
\cqd

Consider $W$ the maximal interval containing $A_{-1}$ such that $f^l |_W$ is continuous. Let $F$ a continuous extension of  $f^l |_W$ to $\overline W$. We have shown above that $A_j \subset W$, $\forall j<0$. Sequence $\{ a_n \}_{n<0}$ converges to a point $a_{-\infty} \in \overline W$.
As $F$ is continuous in this compact set $\overline{W}$, $F(a_{-\infty})$ $=$ $F(\lim_{n\to-\infty} \, a_n)$ $=$ $\lim_{n\to-\infty} F(\, a_n)$
$=$ $\lim_{n\to-\infty} \, f^l(a_n)$ $=$ $\lim_{n\to-\infty} \, a_{n+1}= a_{-\infty}$. If $a_{-\infty} \in \partial W$, $a_{-\infty}$ would be a periodic super-attractor. As we are supposing $f$ has no kind of attractors, $a_{-\infty}$ belongs to the interior of $W$. Thus, $f^l(a_{-\infty})=F(a_{-\infty})=a_{-\infty}$.

\begin{claim}
$period(a_{-\infty})=l$.
\end{claim}
\dem Suppose that $f^s(a_{-\infty})=a_{-\infty}$ for some $0<s<l$. Taking n sufficiently big and for a neighborhood $V_{a_{-\infty}}$, $\exists A_{-n} \in V_{a_{-\infty}}$ such that $f^s(A_{-n})=A_{-m} \subset V_{a_{-\infty}}$ (as we have that the components $C_J$ have to map exactly onto components of $C_J$) and $n>m$ (for otherwise $a_{-\infty}$ would be an attractor).

This way we may again use an argument of transposing the contiguous intervals by finite steps that we used above, what give us
$f^s(A_{-n})=A_{-m} \Rightarrow f^s(A_{-1})=A_{(n-m)-1}$.
As we have taken $l$ minimum such that this condition occurs, then $s=l$.
\cqd

At this point of the proof, we already  have
\begin{enumerate}
\item[(i)] $period(a_{-\infty})=l$;
\item[(ii)] $F_{I(a_{-\infty})}$ $=$ $f^l|_{I(a_{-\infty})}$ with
$I(a_{-\infty})\supset\overline{\bigcup_{n\le-1}A_n}$ $=$ $[a_{-\infty},a]$;
\item[(iii)] $F_{I(a_{-\infty})}([a_{-\infty},a])$ $=$
$\overline{\bigcup_{n\le\gamma}A_n}$ $\supset$ $\overline{\bigcup_{n\le0}A_n}$ $=$ $[a_{-\infty},b]$.
\end{enumerate}

We also have
\begin{equation}\tag{iv}
F_{I(a_{-\infty})}(A_n)=A_{n+\gamma+1}\,\forall\,n\le-1.
\end{equation}

\begin{claim}
$\co_{f}^{+}(a_{-\infty})\cap\bigcup_{n\in\ZZ}\overline{A_n}=\emptyset$.
\end{claim}
\dem Otherwise $\exists\, i>0$ and $j\in\ZZ$ such that $f^{ i}(a_{-\infty})\in[a_ j,a_{ j+1})$.
In this case, let $V=(a_{-\infty},\delta)$ with $\delta>0$ small such that
$f^ i|_V$ is a homeomorphism and $f^ i\big((a_{-\infty},\delta)\big)\subset(a_ j,a_{ j+1})=A_ j$.
Note that one can find $m<<0$ such that $\overline{A_m}\subset V$, because $A_n=(a_n,a_{n+1})$ and $a_n\downarrow a_{-\infty}$
when $n\to-\infty$. So, $f^{i+\theta(A_j)}|_{\overline{A_m}}$ is a
homeomorphism and $\overline{f^{i+\theta(A_j)}(A_m)}=f^{i+\theta(A_j)}(\overline{A_m})\subset
f^{\theta(A_j)}(A_j)= J$, but this is impossible as $A_m\in C_J$.
\cqd

A completely analogous reasoning proves that sequence $\{ a_n \}_{n>0}$ converges to a periodic point
$a_{+\infty}>b$
with period $r=\min \{v | f^v(A_{1})=A_k, k \leq 0 \}$ and such that
$\co_{f}^{+}(a_{+\infty})\cap\bigcup_{n\in\ZZ}\overline{A_n}=\emptyset$.

Moreover, taking $\rho\le0$ so that
$f^r(A_1)=A_{\rho}$, we have
\begin{enumerate}
\item[(i$'$)] $period(a_{+\infty})=r$;
\item[(ii$'$)] $F_{I(a_{+\infty})}$ $=$ $f^r|_{I(a_{+\infty})}$ with
$I(a_{+\infty})\supset\overline{\bigcup_{n\ge1}A_n}$ $=$ $[b,a_{+\infty}]$;
\item[(iii$'$)] $F_{I(a_{+\infty})}([b,a_{+\infty}])$ $=$
$\overline{\bigcup_{n\le\rho}A_n}$ $\supset$ $\overline{\bigcup_{n\ge0}A_n}$ $=$ $[a,a_{+\infty}]$.
\end{enumerate}

We also have
\begin{equation}\tag{iv$'$}
F_{I(a_{-\infty})}(A_n)=A_{n+\rho-1}\,\forall\,n\ge1.
\end{equation}

\begin{claim}
$(a_{-\infty},a_{+\infty})=\xi J$.
\end{claim}
\dem
Note that $(a_{-\infty},a_{+\infty})=\bigcup_{n\in\ZZ}\overline{A_n}$ is a periodic nice interval
(because $\big(\co_{f}^{+}(a_{-\infty})\cup\co_{f}^{+}(a_{+\infty})\big)$ $\cap$ $\bigcup_{n\in\ZZ}\overline{A_n}$
$=$ $\emptyset$) with $period(a_{-\infty})<period(a)$ and $period(a_{+\infty})<period(b)$ (observe that by construction $period(a_{-\infty})=l=min\{k \geq 0 | f^k(A_{-1})=A_j,$ for some $j \ge0\} \le \theta(A_{-1})< period(a)$) .

Thus, to prove the claim we only have to show that there is no periodic nice interval $(p,q)$
such that $(a_{-\infty},a_{+\infty})\supsetneqq(p,q)\supset[a,b]$.

Indeed, if exists $(p,q)\in\cn_{per}$ such that $(a_{-\infty},a_{+\infty})\supsetneqq(p,q)\supset[a,b]$ then
$\exists\,x\in\{p,q\}$ such that $x\in(a_{-\infty},a)\cup(b,a_{+\infty})$. In
this case, $x\in \overline{A_w}$
for some $w \in \ZZ$ and then $f^{\theta(A_w)}(x)\in f^{\theta(A_w)}(\overline{A_w})=[a,b]\subset(p,q)$,
but this is impossible as $(p,q)$ is a nice interval.
\cqd

As $\alpha=a_{-\infty}$ and $\beta=a_{+\infty}$, it follows from (i),(i$'$),(ii),(ii$'$),(iii),(iii$'$),(iv) and (iv$'$)
the items (1),(1$'$),(2),(2$'$),(3),(3$'$),(4), (4$'$), (6) and (6$'$) of the Proposition.

\begin{claim}
If $(\alpha,a)\supset I\in C_J$ then  $F_I=F_{A_i}\circ (F_{I(\alpha)})^k$ for some $i\in\{0,...,\gamma\}$ and $k\ge1$.
\end{claim}
\dem Note that $\{I\in C_J$ $|$ $I\subset(\alpha,a)\}$ $=$ $\{A_n\}_{n\le-1}$.
Given $(\alpha,a)\supset I\in C_J$,
write $I=A_m$, with $m\le-1$, and
let $k=\min\{s\ge0$ $|$ $(F_{I(\alpha)})^s(A_m)\notin\{A_n\}_{n\le-1}\}$. Of course that $k\ge1$.
By (iv), $$F_{I(\alpha)}(\{A_n\}_{n\le-1})=\{F_{I(\alpha)}(A_n)\}_{n\le-1}=\{A_n\}_{n\le\gamma}=$$
$$=\{A_n\}_{n\le-1} \cup \{A_0,...,A_{\gamma}\}.$$
So, $(F_{I(\alpha)})^k(A_m)$ $\in$ $\{A_0, ...,A_{\gamma}\}$. Writing $A_i=(F_{I(\alpha)})^k(A_m)$, we get
$J=F_{A_i}(A_i)=F_{A_i}\circ(F_{I(\alpha)})^k(A_m)$ and so, $F_I=F_{A_m}=F_{A_i}\circ(F_{I(\alpha)})^k$.
\cqd

The claim above proves item ($5$) of the Proposition and a completely analogous reasoning proves item ($5'$), completing the proof of the Proposition.

\cqd

\begin{Corollary}\label{Corollaryimperfeito} Let $f:[0,1]\setminus\{c\}\to[0,1]$ be a $C^{2}$ contracting Lorenz map. 
Let $J=(a,b) \in \mathcal{N}_{per}$ with $\co_{f}^{+}(a)=\co_{f}^{+}(b)$  and such that $\Lambda_{J}$ does not contain periodic attractors or weak repellers. If $p\in Per(f)\cap\xi J$ then either $p\in\co_f^+(a)$ or $\co_f^+(p)\cap J\ne\emptyset$. Furthermore, if $x\in\xi J$ and $\co_{f}^{+}(x)\cap J=\emptyset$ then $x\in\co_{f}^{-}(a)$.
\end{Corollary}
\dem
Write $(\alpha, \beta )=\xi J$. It follows from items (6) and (6$'$) of the Theorem~\ref{imperfeito} that $(\alpha,\beta)=\bigcup_{n\in\ZZ}A_n\uplus\bigcup_{n\in\ZZ}\partial A_n$.
If $p\in A_j$ for some $j$ then $f^{\theta(A_j)}(p)\in f^{\theta(A_j)}(A_j)=J$.
On the other hand, if $p\notin A_j$ $\forall\,j\in\ZZ$ then $p\in\partial A_j$ for some $j\in\ZZ$.
Thus, $f^{\theta(A_j)}(p)\in f^{\theta(A_j)}(\partial A_j)=\{a,b\}$.
 This implies that $x\in\xi J$ and $\co_{f}^{+}(x)\cap J=\emptyset$ then $x\in\co_{f}^{-}(a)$. As a consequence, if $p\in\xi J$ is periodic and $\co_{f}^{+}(p)\cap J=\emptyset$ 
then $p\in\co_f^+(a)$ $(=\co_f^+(b))$.
\cqd

\newpage

\section{First Return Maps to Nice Intervals}\label{Section98767}

Given an interval $J=(a,b)$ denote the {\em first return map} to $J$ by $\cf_{J}:J^{*}\to J$. That is, $\cf_{J}(x)=f^{R(x)}(x)$, where $J^{*}=\{x\in J$ $;$ $\co^{+}_{f}(f(x))\cap J\ne\emptyset\}$ and $R(x)=\min\{j\ge1$ $;$ $f^{j}(x)\in J\}$. By Lemma~\ref{Remark98671oxe}, if $c\in J$ then $\leb(J\setminus J^{*})=0$. Let $\cp_{J}$ be the collection of connected components of $J^{*}$. 

%%%%%%%%%%%%%%%%%%%%%%%%%%%%%%%%%%%%%%%%%%%%%%
%\begin{figure}
\begin{center}\label{PrimeiroRetorno.png}
\includegraphics[scale=.26]{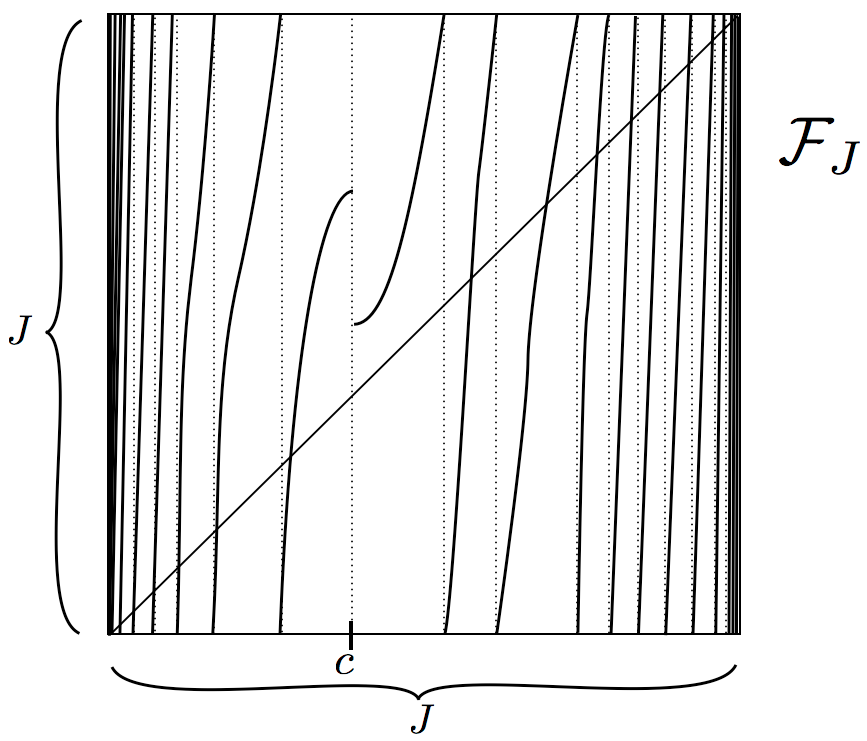}\\
\end{center}
%\end{figure}
%%%%%%%%%%%%%%%%%%%%%%%%%%%%%%%%%%%%%%%%%%%%%%

\begin{Lemma}\label{Lemma8388881}
If $J=(a,b)$ is a non-empty open interval then every $I\in\cp_{J}$ is an open set and $R|_{I}$ is constant.
\end{Lemma}
\dem  Let $I\in\cp_{J}$ and $x\in I$. Let $n=R(x)$. $f^n$ is a continuous function (defined on $[0,1]\setminus\bigcup_{k=0}^{n-1}f^{-j}(c)$). Thus, every point $y\notin\bigcup_{k=0}^{n-1}f^{-j}(c)$ sufficiently close to $x$ will return to $J$ at the same time of $x$. Thus, $I$ is open and, as $I$ is connected, $R(y)=R(x)$ for every $x\in I$.
\cqd

\begin{Lemma}\label{Lemma8388881a}
Let $J=(a,b)$ be an interval with $c\in\overline{J}$.
The following three statements are true.
\begin{enumerate}
\item If $a<c$ and $\co_f^+(a)\cap(a,b)=\emptyset$ then
\begin{enumerate}
\item $\big((p,q)\in\cp_{J}\text{ and }p\ne c\big) \Rightarrow \cf_{J}((p,q))=(a,f^{R|_{(p,q)}}(q));$
\item $\big((p,q)\in\cp_{J}\text{ and }q<c\big) \Rightarrow \cf_{J}((p,q))=J.$
\end{enumerate}

\item If $b>c$ and $\co_f^+(b)\cap(a,b)=\emptyset$ then
\begin{enumerate}
\item $\big((p,q)\in\cp_{J}\text{ and }q\ne c\big) \Rightarrow \cf_{J}((p,q))=(f^{R|_{(p,q)}}(p),b);$
\item $\big((p,q)\in\cp_{J}\text{ and }p>c\big) \Rightarrow \cf_{J}((p,q))=J.$
\end{enumerate}
\item If $J$ is a nice interval then $$\big(I\in\cp_{J}\text{ and }c\notin\partial I\big) \Rightarrow \cf_{J}(I)=J.$$
\end{enumerate}
\end{Lemma}
\dem\,
Of course (3) is a consequence of (1) and (2). Assume that $I=(p,q)\in\cp_{J}$ and $p\ne c$. Let $n=R|_{I}$.
 As $\co_{f}^{+}(a)\cap J=\emptyset$, if $p=a$ then $f^{n}(p)\le a$. If $f^{n}(p)<a$ then $f^{n}(p+\varepsilon)<a$ for $\varepsilon>0$ sufficiently small. This is an absurd, as $n$ is a return time to $p+\varepsilon\in I$. Thus, $f^{n}(p)=a$ whenever $p=a$. So, consider now $a<p$. As $n$ is the first return time of $I$ to $(a,b)$, $f^{j}(I)\cap (a,b)=\emptyset$ for every $0<j<n$.
Thus,  $f^{j}(p)\ne c$ $\forall\,0\le j<n$.
Thus $f^{n}$ is continuous in $(p-\delta,p+\delta)$ for a sufficiently small $\delta>0$.
As a consequence, if $a<f^{n}(p)<b$ then, taking $\delta>0$ small, $n$ will be the first return time for $(p-\delta,q)$ to $(a,b)$, contradicting $I\in\cp_{J}$.
So, $f^{n}(p)=a$ proving (a).
Yet in the hypothesis of (1), suppose $q<c$. Then, we also are in the hypothesis of (a), and similarly to what we've done to p, as $q\ne c$, as $n$ is the first return time of $I$ to $(a,b)$, $f^{j}(I)\cap (a,b)=\emptyset$ for every $0<j<n$ implies also that  $f^{j}(q)\ne c$ $\forall\,0\le j<n$ and (similar reasoning) that $f^{n}(q)=b$, proving (b)

Analogously (2) follows. 
\cqd

%%%%%%%%%%%%%%%%%%%%%%%%%%%%%%%%%%%%%%%%%%%%%%
\begin{figure}
\begin{center}
\includegraphics[scale=.35]{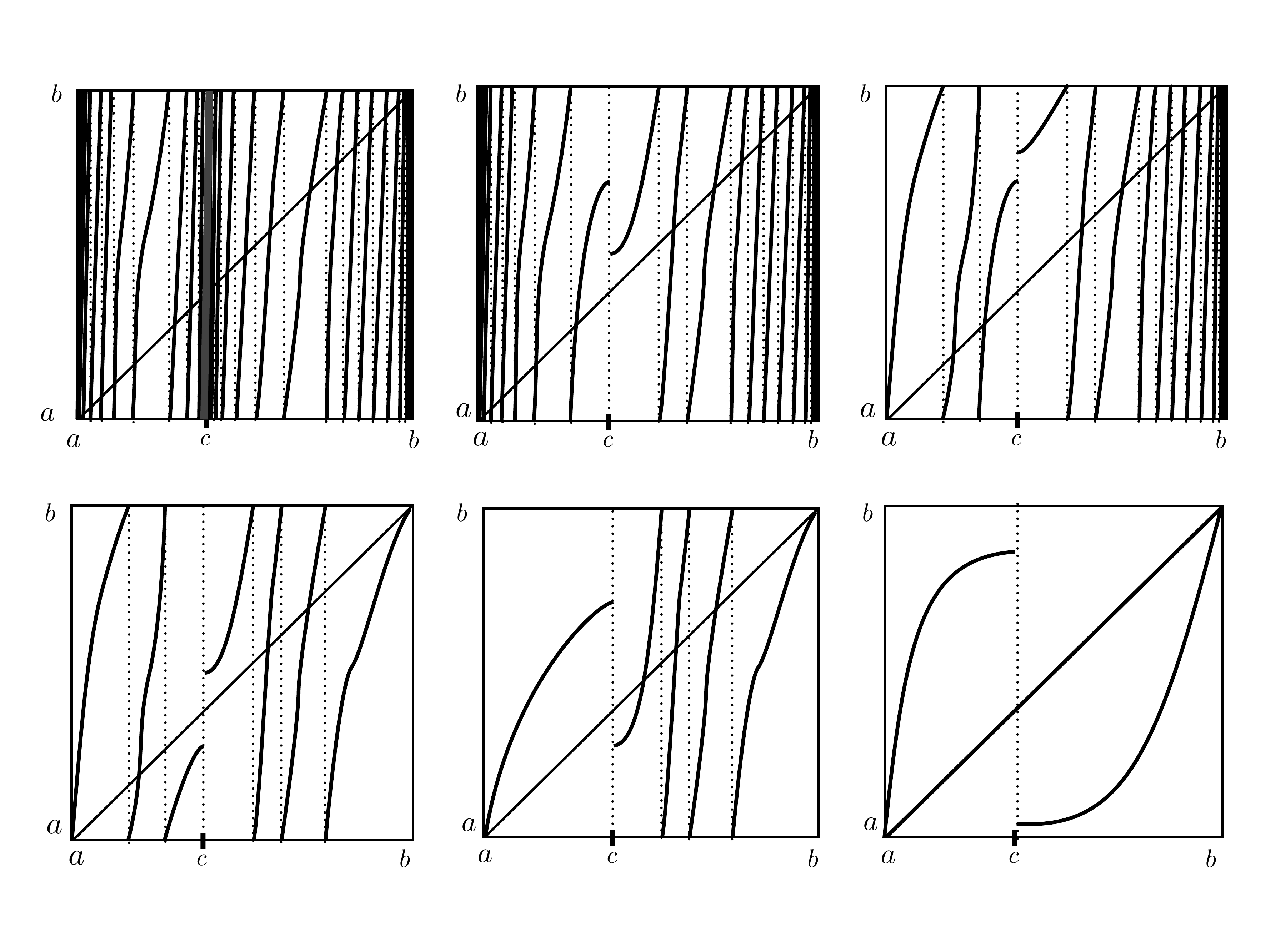}\\
\caption{Some examples of first return maps to a nice interval.}\label{ExemplosPR.png}
\end{center}
\end{figure}
%%%%%%%%%%%%%%%%%%%%%%%%%%%%%%%%%%%%%%%%%%%%%%

\begin{Corollary}\label{Corollary8388881b}
If $J=(a,b)$ is an interval with $c\in\overline{J}$ then the following statements are true.
\begin{enumerate}
\item If $a<c$ and $\co_f^+(a)\cap(a,b)=\emptyset$ then$$a\in\partial I\text{ for some }I\in\cp_{J}\Leftrightarrow\,\,a\in Per(f).$$
\item If $b>c$ and $\co_f^+(b)\cap(a,b)=\emptyset$ then $$b\in\partial I\text{ for some }I\in\cp_{J}\Leftrightarrow\,\,b\in Per(f).$$
\end{enumerate}
\end{Corollary}
\dem
If $I=(a,q)\in\cp_{J}$ (the case $I=(q,b)$ is analogous) and $n=R|_{I}$, it follows from Lemma~\ref{Lemma8388881a} that 
$f^n(a)=\cf_{J}(a)=a$. That is, $a$ is a periodic point.

Now suppose that $a\in Per(f)$ or $a$ is a super attractor (the proof for $b$ is analogous).
Thus, there is $n>0$ such that $\lim_{\delta \downarrow 0} f^{n}(a+\delta)=a$ and $f^{j}(a)\notin[a,b)\ni0$ $\forall0<j<n$.
As $f^n$ is well defined, continuous and monotone on $(a,a+\varepsilon)$ for some $\varepsilon>0$ and as $f$ preserves orientation, we get $f^{n}(x)\in(a,b)$ for every $x>a$ sufficiently close to $a$ and that $f^{j}(x)\notin(a,b)$ $\forall,0<j<n$. Thus, there is some $I=(a,q)\in\cp_{J}$.
\cqd

\begin{Lemma}\label{LemmaHGFGH54}
If $J=(a,b)$ is a nice interval then there are sequences $a_{n},b_{n}\in \overline{J}\cap Per(f)$ such that
\begin{enumerate}
\item $\lim_{n}a_{n}=a$ and $\lim_{n}b_{n}=b$;
\item $\co^{+}_{f}(a_{n})\cap(a_{n},b)=\emptyset$ and $\co^{+}_{f}(b_{n})\cap(a,b_{n})=\emptyset$.
\end{enumerate}
\end{Lemma}
\dem
We will show the existence of a sequence $a_{n}\in \overline{J}\cap Per(f)$ with $\lim_{n}a_{n}=a$ such that $\co^{+}_{f}(a_{n})\cap(a_{n},b)=\emptyset$.  Assume that $a\notin Per(f)$, otherwise take $a_{n}=a$. Let $I_{0}=(p_{0},q_{0})\in\cp_{J}$ such that $I_{0}\subset(a,c)$. By Lemma~\ref{Lemma8388881a}, as $a$ is not periodic we get $p_{0}\ne a$. Thus there is some $I_{1}\in\cp_{J}$ with $I_{1}\subset (a,p_{0})$. In particular, $c\ne\partial I_{1}$. Again by Lemma~\ref{Lemma8388881a} we get $\cf_{J}(I_{1})=f^{n_{1}}(I_{1})=J$. Thus, there is a fixed point $a_{1}\in\overline{I_{1}}$ of $f^{n_{1}}|_{\overline{I_{1}}}$. As $n_{1}=R_{J}(I_{1})$ it follows that $f^{j}(a_{1})\notin(a,b)$ for every $0<j<n$ and so, $\{a_{1}\}=\co^{+}(a_{1})\cap(a,b)$. From this we get $\co^{+}(a_{1})\cap(a_{1},b)=\emptyset$. Again, writing $I_{1}=(p_{1},q_{1})$, it follows as before that $a\ne p_{1}$ and so there is some $I_{2}\in\cp_{J}$ such that $I_{2}\subset(a,p_{1})$. Proceeding as before we get a periodic point $a_{2}\in\overline{I_{2}}$ satisfying the statement. Inductively we get a sequence $a_{n}\searrow a$ of periodic points with $\co^{+}(a_{n})\cap(a_{n},b)=\emptyset$.
Analogously one can get the sequence $b_{n}\nearrow b$.

\cqd

\begin{Lemma}[Homterval lemma]
\label{homtervals}
Let $f:[0,1]\setminus\{c\}\to[0,1]$ be a contracting Lorenz map and $I$ an open interval, $I\subset[0,1]$. Then either $\exists\,n$ such that $f^n|_I$ is a homeomorphism and $c\in f^n(I)$ or $I$ is a wandering interval or there is a periodic attractor $\Lambda$ such that $I\cap\beta(\Lambda)\ne\emptyset$ or there is a weak repeller.
\end{Lemma}

\dem
This is a well known result, and it is usual to define a {\em homterval} (as named by Misiurewicz) as being an interval on which $f^j$ is monotone $\forall j \ge 0$, see, for example, \cite{G79}. 

Suppose $I$ is not a wandering interval. Then, there will be $k<\ell$ for which $f^k(I)\cap f^\ell(I)\ne\emptyset$, define $f^k(I)=(a,b)$. If there is no such $n$ such that  $c\in f^n(I)$ we may consider the union $(a,b)\cup (f^{\ell-k}(a),f^{\ell-k}(b))$, and this is an interval, as their intersection is non-empty. 

Then $T:=\bigcup_{j=0}^{\infty}f^{{(\ell-k)}j}(I)$ is a positively invariant homterval. As $f$ does not have weak repellers (in particular, $f^{j}$ is not the identity in intervals $\forall\,j>0$), if $F$ is the continuous extension of  $f^{\ell-k}|_{T}$ to $\overline{T}$ then  $F$ has a fixed point on $p\in\overline{T}$. This implies that there is a periodic attractor $\Lambda$ with $p\in\Lambda$. Furthermore, $\exists\,\varepsilon>0$ such that  $B_{\varepsilon}(p)\cap\overline{T}\subset\beta(\Lambda)$.
\cqd

\begin{Lemma}
\label{noitenoite}
Let $f:[0,1]\setminus\{c\}\to[0,1]$ be a $C^{2}$ non-flat contracting Lorenz map.
If $Per(f)\cap (0,1)=\emptyset$ then either $f$ has an attracting periodic orbit (indeed, at least one of the fixed points is an attractor) or $\omega_{f}(x)\ni c \,\forall\, x \in (0,1)$.% (that is, $f$ is Cherry-like)
\end{Lemma}

\dem
Suppose this is not the case. That means we can pick a point $x \in (0,1)$ such that $\omega_{f}(x)\not\ni c $. Let $(a,b)$ be the connected component of $[0,1]\setminus\overline{\co^+_f(x)}$ containing $c$. Then $(a,b) \subsetneq (0,1)$ is a nice interval  and then $Per(f)\cap (a,b)\ne\emptyset$ (Lemma~\ref{LemmaHGFGH54}), absurd.
\cqd

\begin{Lemma}[Lemma 3.36 of \cite{StP}]\label{LemmaStP}If $f:[0,1]\setminus\{c\}\to[0,1]$ is a $C^{2}$ non-flat contracting Lorenz map then every wandering interval accumulates on both sides of the critical point. In particular, a wandering interval cannot contain any interval of the form $(-r,c)$ or $(c,r)$.
\end{Lemma}

\dem
Suppose we have a wandering interval $J$ that doesn't accumulates on the right side of the critical point, say, it never enters a neighborhood $(c,c+\varepsilon)$. So, we can modify $f$ to coincide with the original function out of this interval, but being $C^2$ in this interval (see Figure~\ref{errantepng}). This way, the modified function is a  $C^{2}$ function that will have a wandering interval, but it can't happen, according to Theorem A of Chapter IV of \cite{MvS}.

%%%%%%%%%%%%%%%%%%%%%%%%%%%%%%%%%%%%%%%%%%%%
\begin{figure}
\begin{center}
\includegraphics[scale=.3]{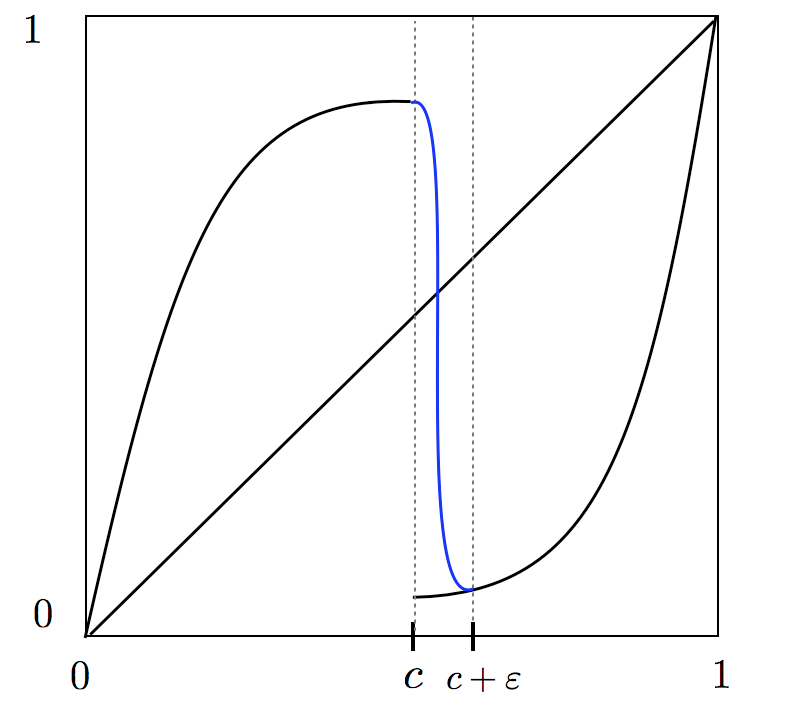}\\
\caption{}\label{errantepng}
\end{center}
\end{figure}
%%%%%%%%%%%%%%%%%%%%%%%%%%%%%%%%%%%%%%%%%%%%%%

\cqd

\begin{Lemma}\label{AcumulacaoDePer}
If $f:[0,1]\setminus\{c\}\to[0,1]$ is a non-flat $C^{2}$ contracting Lorenz map without periodic attractors or weak repellers then either $\exists\,\delta>0$ such that $c\in\omega_{f}(x)$ $\forall\,x\in(c-\delta,c+\delta)$ or $$\overline{Per(f)\cap(0,c)}\ni c\in\overline{(c,1)\cap Per(f)}.$$

\end{Lemma}

\dem
Suppose that $f$ does not have periodic attractors or weak repellers and suppose also that $\nexists\delta>0$ such that $c\in\omega_{f}(x)$ $\forall\,x\in(c-\delta,c+\delta)$. In this case, by Lemma~\ref{noitenoite}, $Per(f)\cap(0,1)\ne\emptyset$. As $f$ does not have periodic attractors or weak repellers, $\co_{f}^{+}(x)\cap(0,c)\ne\emptyset\ne(c,1)\cap\co_{f}^{+}(x)$ $\forall\,x\in(0,1)\setminus\{c\}$. Thus $Per(f)\cap(0,c)\ne\emptyset\ne(c,1)\cap Per(f)$.

Let $a=\sup Per(f)\cap(0,c)$ and $b=\inf Per(f)\cap(c,1)$. We know that $0<a\le c\le b<1$. If $a=b$ the proof is done. So suppose that $a\ne b$. We may assume that $0<a<c\le b<1$ (the other case is analogous).

We claim that $\co_{f}^{+}(a_{-})\cap(a,b)=\emptyset=(a,b)\cap\co_{f}^{+}(b_{+})$. Indeed, if there is a minimum $\ell\ge1$ such that $f^{\ell}(a_{-})\in(a,b)$ then $\emptyset\ne f^{\ell}((a-\varepsilon,a)\cap Per(f))\subset(a,b)$, contradicting the definition of $a$ and $b$. With the same reasoning we can show that $\co_{f}^{+}(b_{+})\cap(a,b)=\emptyset$.

Note that $\exists n>0$ such that $f^{n}((a,c))\cap(a,c)\ne\emptyset$. Indeed, $(a,c)$ can not be  a wandering interval (Lemma~\ref{LemmaStP}) and as $f$ does not have periodic attractors, it follows from the homterval lemma (Lemma~\ref{homtervals}) that $f^{n}((a,c))\ni c$ for some $n\ge1$. Let $\ell$ be the smallest integer bigger than $0$ such that $f^{\ell}((a,c))\cap(a,c)\ne\emptyset$. As $\co_{f}^{+}(a)\cap(a,c)=\emptyset$, we get $f^{\ell}((a,c))\supset(a,c)$. Thus there is a periodic point $p\in[a,c)$ with period $\ell$. By the definition of $a$, it follows that $p=a$.

We claim that $f^{\ell}((a,c))\subset(a,b)$. If not  let $q_{0}\in f^{\ell}((a,c))\cap Per(f)\cap[b,1)$. Let $q=\min \co_{f}^{+}(q_{0})\cap(c,1)$ and $q'=(f^{\ell}|_{(a,c)})^{-1}(q)$. Of course that $a<q'<c<q$ and that $(q',q)$ is a nice interval. Thus, by Lemma~\ref{LemmaHGFGH54}, $Per(f)\cap(q',c)\ne\emptyset$ and this contradicts the definition of $a$.

Note that $f^{\ell}((a,c))\ni c$, otherwise $f$ would have periodic attractors or weak repellers. As a consequence of this and of the claim above, $b>c$. 

As $b>c$, $(a,b)$ is a nice interval. We already know that $f^{\ell}(a)=a$. Moreover, by the definition of $b$ and Lemma~\ref{LemmaHGFGH54}, $b$ also must be a periodic point. So, let $r=\period(p)$. From the same reasoning of the claim above, we get $f^{r}((c,b))\subset((a,b))$.

Thus the first return map to $[a,b]$ is conjugated to a contracting Lorenz map $g:[0,1]\setminus\{c_{g}\}\to[0,1]$. As $\nexists\delta>0$ such that $c\in\omega_{f}(x)$ $\forall\,x\in(c-\delta,c+\delta)$, it follows that $\exists\,x\in[0,1]$ such that $c_{g}\notin\omega_{g}(x)$. So, it follows from Lemma~\ref{noitenoite} that $Per(g)\cap(0,1)\ne\emptyset$. As a consequence, $Per(f)\cap(a,b)\ne\emptyset$. This contradicts the definition of $a$ and $b$, proving the lemma.

\cqd

\begin{Proposition}[``Variational Principle'']\label{LemmaVarPric}
Let $f:[0,1]\setminus\{c\}\to[0,1]$ be a non-flat $C^{2}$ contracting Lorenz map without periodic attractors or weak repellers. Suppose that $\not\exists\delta>0$ such that $c\in\omega_{f}(x)$ $\forall\,x\in(c-\delta,c+\delta)$.

Given $\varepsilon>0$ there exists a unique periodic orbit minimizing the period of all periodic orbit intersecting $(c-\varepsilon,c)$. Analogously, there exists a unique periodic orbit minimizing the period of all periodic orbit intersecting $(c,c+\varepsilon)$.
\end{Proposition}
\dem
As $Per(f)\cap(c-\varepsilon,c)\ne\emptyset$ (Lemma~\ref{AcumulacaoDePer}), let
\begin{equation}\label{Eq10011}
n=\min\{\period(x)\,;x\in Per(f)\cap(c-\varepsilon,c)\}
\end{equation}
and suppose that there are $p_0,q_0\in Per_n(f)\cap(c-\varepsilon,c)$ such that $\co_f^+(p_0)\ne\co_f^+(q_0)$. Let $p=\max\{\co_f^+(p_0)\cap(c-\varepsilon,c)$ and $q=\max\{\co_f^+(q_0)\cap(c-\varepsilon,c)$. Thus, $\co_f^+(p)\cap(p,c)=\emptyset=\co_f^+(q)\cap(q,c)$. We may assume that $q<p$.

%%%%%%%%%%%%%%%%%%%%%%%%%%%%%%%%%%%%%%%%%%%%%%
\begin{figure}[H]
\begin{center}\label{Fig47488}
\includegraphics[scale=.25]{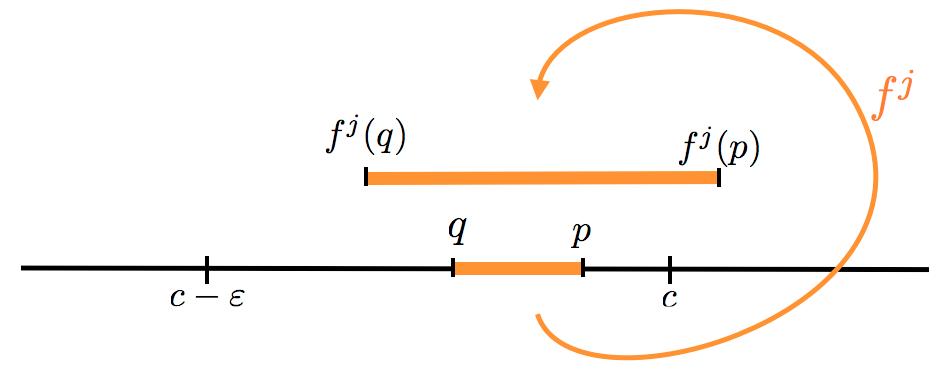}\\
\caption{}\label{Fig47488}
\end{center}
\end{figure}
%%%%%%%%%%%%%%%%%%%%%%%%%%%%%%%%%%%%%%%%%%%%%%

Note that $f^n$ can not be monotone on $(q,p)$. Indeed, otherwise, if $f^n$ is monotone on $(q,p)$ then $f^n([q,p])=[q,p]$. As $f$ doesn't have any weak repeller, $f^n$ can not be the identity on $[q,p]$ and so, $f^n([q,p])=[q,p]$ implies in the existence of an attracting fixed point for $f^n$ on $[q,p]$. But this is impossible, as we are assuming that $f$ does not have a finite attractor.

As $f^n$ is not monotone on $(q,p)$, there is $0<j<n$ such that $f^j$ is monotone on $(q,p)$ and $c\in f^j((q,p))$. Thus, $f^j(q)<c<f^j(p)$. Moreover, $f^j(q)<q$ (because $\co_f^+(q)\cap(q,c)=\emptyset$ and $j<n$). Thus, $f^j((q,p))\supset(q,p)$ (see Figure~\ref{Fig47488}) and this implies in the existence of a periodic point $a\in[q,p]\subset(c-\varepsilon,c)$ with period $j<n$, contradicting (\ref{Eq10011}).

The proof for the case $(c,c+\varepsilon)$ is analogous.

\cqd

\newpage

\section{Renormalization and  Cherry-like maps}

\begin{Definition}[Left and right renormalizations]
Let $J=(a,b)\in\cn$ and let $g:J^{*}\to J$ be the map of first return to $J$. We say that $f$ is renormalizable by the left side with respect to $J$ (or, for short, $J$-left-renormalizable) if $(a,c)\subset J^{*}$ (this means that $g|_{(a,b)}=f^{n}|_{(a,b)}$ for some $n\ge1$).
Analogously we say that $f$ is renormalizable by the right side with respect to $J$ (or, for short, $J$-right-renormalizable) if $(c,b)\subset J^{*}$.
\end{Definition}

\begin{Definition}If the first return map to an interval $\overline{J}\ne[0,1]$ is conjugated to a Lorenz map, we say that {\em $f$ is renormalizable} with respect to $J$.
\end{Definition}
Note that $f$ is renormalizable with respect to $J$ if and only if $J\in\cn_{per}$ and $f$ is renormalizable by both sides (left and right) with respect to $J$. Moreover, it is easy to check the following result.

\begin{Lemma}\label{Lemma09090863}
Let $J=(a,b)\in\cn$. The following statements are equivalent.
\begin{enumerate}
\item $f$ is renormalizable with respect to $J$.
\item $(\,\overline{J}\,)\,^{*}=[a,c)\cup(c,b]$.
\item $c\in\partial I$ $\forall\,I\in (\overline J)^{*}$.
\item $a$ and $b$ are periodic points, $$f^{\period(a)}([a,c))\subset[a,b]\supset f^{\period(b)}((c,b]).$$
\end{enumerate}
\end{Lemma}

\dem
$(1)\implies(2)$ If $f$ is renormalizable with respect to $J=(a,b)\in\cn$, this means that $\exists m$, $n$ such that $f^m((a,c))\subset(a,b)$ and $f^n((c,b))\subset(a,b)$, $f^m(a)=a$ and $f^n(b)=(b)$. So, we have $(\,\overline{J}\,)\,^{*}=[a,c)\cup(c,b]$.
$(2)\iff(3)$ Straightforward. $(2)\implies(4)$ If $(\,\overline{J}\,)\,^{*}=[a,c)\cup(c,b]$, by Corollary~\ref{Corollary8388881b} that $a$, $b \in Per(f)$. As the first return map to $\overline J$, $R_{\overline{J}}$, is constant on $[a,c)$ and also in $[c,b]$ and as $R_{\overline{J}}(a)=\period(a)$ and $R_{\overline{J}}(b)=\period(b)$, then $f^{\period(a)}([a,c)\subset[a,b]\supset f^{\period(b)}((c,b])$. $(4)\implies(1)$  Straightforward.

\cqd

The interval involved in a (left/right) renormalization is called an interval of  (left/right) renormalization. A map $f$ is {\em non renormalizable} if it does not admit any interval of renormalization.

\begin{Definition}[Renormalization Cycle]\label{rencyc} Given a renormalization interval $J=(a,b)$, define the renormalization cycle associated to $J$ (or generated by $J$) as $$U_J=\bigg(\bigcup_{i=0}^{\period(a)}f^{i}((a,c))\bigg)\cup\bigg(\bigcup_{i=0}^{\period(b)}f^{i}((c,b))\bigg).$$
\end{Definition}

\begin{Definition}[Nice Trapping Region]\label{nicetrap} We call a set $K_J$ a {\em nice trapping region} if it is the union of gaps of $\Lambda_J$ (that is, connected component of $[0,1]\setminus \Lambda_J$, see Lemma \ref{Remark98671oxe}) such that each gap contains one interval of the renormalization cycle.
\end{Definition}

\begin{Lemma}\label{Lemma545g55} Let $J=(a,b)$ be a renormalization interval of a contracting Lorenz map $f:[0,1]\setminus\{c\}\to[0,1]$, with $\ell=\period(a)$ and $r=\period(b)$. If 
$f^{\ell}(x)>x$ $\forall x \in{(a,c)}$,  $f^{r}(x)<x$ $\forall x \in{(c,b)}$ and $\lim_{x \uparrow c}f^{\ell}(x)>c > \lim_{x\downarrow c}f^{r}(x)$ then  $$\co^{+}_{f}(x)\cap(a,c)\ne\emptyset\ne\co^{+}_{f}(x)\cap(c,b)\hspace{0.5cm}\forall\,x\in J\setminus\co^{-}_f(c).$$
Therefore, the positive orbit $\co^{+}_{f}(x)$ of any $x\in J\setminus\co^{-}_f(c)$ intersects each connected component of the renormalization cycle $U_J$ (and also each connected component of the nice trapping region $K_{J}$).
\end{Lemma}
\dem
Let $\cf:J^*\to J$ be the first return map to the renormalization interval $J$. In this case $J^*=(a,c)\cup(c,b)$ and $$\cf(x)=\left\{\begin{array}{cc}
f^{\ell}(x) & \text{if }x\in(a,c)\\
f^{r}(x) & \text{if }x\in(c,b)
\end{array}\right..$$

For example, suppose that there is $y\in(a,c)\setminus\co^{-}_f(c)$ such that $\cf^n(y)\in(a,c)$ $\forall\,n\ge0$. That is, $a<f^{\ell\,n}(y)<c$ for all $n\ge0$.
As $f^{\ell}|_{(a,c)}$ is an increasing map, we get $f^{\ell}(a)=a<y<f^{\ell}(y)<f^{2\ell}(y)<\cdots<f^{\ell\,n}(y)\cdots<c$. This implies that $\lim_{n\to\infty}f^{\ell\,n}(y)$ is a attracting fixed point to $f^{\ell}$, contradicting our hypothesis. Thus, there is some $n>0$ such that $f^{\ell\,n}(y)=\cf^n(y)\in(c,b)$. Analogously, if  $y\in(c,b)\setminus\co^{-}_f(c)$ then $\cf^m(y)\in(a,c)$ for some $m>0$.
\cqd

\begin{Corollary}\label{Corollary545g55}  Let $f:[0,1]\setminus\{c\}\to[0,1]$ be a $C^{2}$ contracting Lorenz map without periodic attractors or weak repellers.
For any given $J=(a,b)$ renormalization interval $f$, we have that  $$\co^{+}_{f}(x)\cap(a,c)\ne\emptyset\ne\co^{+}_{f}(x)\cap(c,b)\hspace{0.5cm}\forall\,x\in J\setminus\co^{-}_f(c).$$
\end{Corollary}

\dem
Let $\ell=\period(a)$ and $r=\period(b)$. As we don't have periodic attractors, we don't have any super attractor, then $\lim_{x \uparrow c}f^{\ell}(x)>c > \lim_{x\downarrow c}f^{r}(x)$. 
If there is $x\in(a,c)$ such that $f^{\ell}(x)\le x$, then $f^{\ell}|_{(a,c)}$ will have a fixed point that is either an attracting fixed point or weak repeller for $f^{\ell}|_{(a,c)}$, but this contradicts the hypothesis that $f$ doesn't have any attracting periodic point or weak repeller. The same reasoning applies to $f^{\ell}|_{(c,b)}$, and therefore the hypothesis of the corollary implies in the hypothesis of Lemma \ref{Lemma545g55}. 
\cqd

\begin{definition}[Linked Intervals] We say that two open intervals $I_{1}$ and $I_{2}$ are linked if $\partial I_{0}\cap I_{1}\ne\emptyset\ne I_{0}\cap\partial I_{1}$.
\end{definition}

\begin{Lemma}\label{renormalinks} Let $f:[0,1]\setminus\{c\}\to[0,1]$ be a $C^{2}$ non-flat contracting Lorenz map without periodic attractors or weak repellers. Two renormalization intervals of $f$ can never be linked. Moreover, if $J_{0}$ and $J_{1}$ are two renormalization intervals and $J_{0}\ne J_{1}$ then  either $\overline{J_{0}}\subset J_{1}$ or $\overline{J_{1}}\subset J_{0}$. In particular, $\partial J_{0}\cap\partial J_{1}=\emptyset$.

\end{Lemma}
\dem

Write $J_0=(a_0,b_0)$ and $J_1=(a_1,b_1)$. First note that $J_{0}$ and $J_{1}$ can not be linked. Indeed, if they were linked, we would either have $a_{0}<a_{1}<c<b_{0}<b_{1}$ or $a_{1}<a_{0}<c<b_{1}<b_{0}$. We may suppose that $a_{0}<a_{1}<c<b_{0}<b_{1}$. In this case, $a_{1}\in J_{0}$ and by Corollary~\ref{Corollary545g55}.
$\emptyset\ne\co_{f}^{+}(a_{1})\cap(c,b_{0})\subset\co_{f}^{+}(a_{1})\cap(a_{1},b_{1})=\co_{f}^{+}(a_{1})\cap J_{1}$ contradicting the fact that $J_{1}$ is a nice interval.

As $J_{0}\cap J_{1}\ne\emptyset$ (because both contains the critical point) and as $J_{0}$ and $J_{1}$ are not linked, in follows that either $J_{0}\supset J_{1}$ or $J_{0}\subset J_{1}$. We may suppose that $J_{0}\supset J_{1}$. In this case, as $J_{0}\ne J_{1}$ we have three possibilities: either $a_{0}< a_{1}<c<b_{1}= b_{0}$ or $a_{0}= a_{1}<c<b_{1}< b_{0}$ or $J_{0}\supset\overline{J_{1}}$. If $a_{0}< a_{1}<c<b_{1}= b_{0}$, we can use again Corollary~\ref{Corollary545g55} to get  $\co_{f}^{+}(a_{1})\cap J_{1}\ne\emptyset$. On the other hand, if $a_{0}= a_{1}<c<b_{1}< b_{0}$, the same  Corollary~\ref{Corollary545g55} implies that $\co_{f}^{+}(b_{1})\cap J_{1}\ne\emptyset$. In both cases we get a contradiction to the fact that $J_{1}$ is a nice interval. Thus, the remaining possibility is the true one.

\cqd

\begin{Lemma}\label{semnome} Let $f:[0,1]\setminus\{c\}\to[0,1]$ be a $C^{2}$ non-flat contracting Lorenz map without periodic attractors or weak repellers. Furthermore, suppose that $\nexists\delta>0$ such that $c\in\omega_{f}(x)$ $\forall\,x\in(c-\delta,c+\delta)$.
If $J=(a,b)$ is a renormalization interval then there is a unique periodic orbit that minimizes the period inside $J$, i.e., if $p,q\in Per(f)\cap J$ and $\period(p)=\period(q)=\min\{\period(x)\,;\,x\in Per(f)\cap J\}$ then $\co^{+}_{f}(p)=\co^{+}_{f}(q)$. Moreover, if $L_p\in\cn_{per}$ is the connected component of $[0,1]\setminus\co^+_f(p)$ containing $c$, where $\period(p)=\min\{\period(x)\,;\,x\in Per(f)\cap J\}$, then
\begin{enumerate}
\item $\co^+_f(x)\cap L_p\ne\emptyset$ for all $x\in J\setminus\co^-_f(p)$;
\item $J'\subset L_p$ for all renormalization interval $J'\subset J$.
\end{enumerate}
\end{Lemma}
\dem
Let $n=\min\{\period(x)\,;\,x\in Per(f)\cap J\}$. Let $p\in J\cap Per(f)$ be such that $\period(p)=n$. It follows from Corollary~\ref{Corollary545g55} that $\co_{f}^{+}(p)\cap(a,c)\ne\emptyset\ne(c,b)\cap\co_{f}^{+}(p)$. Now, from the ``Variational Principle'' (Lemma~\ref{LemmaVarPric}) we get $\co_{f}^{+}(p)$ is the unique orbit of period $n$ intersecting $(a,c)$ and also the unique intersecting $(c,b)$. Thus, it is the only periodic orbit with period $n$ intersecting $J$.

Let $L_p\in\cn_{per}$ be the connected component of $[0,1]\setminus\co^+_f(p)$ containing $c$. By the minimality of the period of $p$ in $J$ and the definition of $\xi J$ (Definition~\ref{defxidej}), it follows that $\xi L_{p}=J$. From the corollary of Theorem~\ref{imperfeito} (Corollary~\ref{Corollaryimperfeito}), we get $\co^+_f(x)\cap L_p\ne\emptyset$ for all $x\in J\setminus\co^-_f(p)$. Of course this also implies that if $J'\subset J$ is a renormalization interval then $J'\subset L_{p}$. Indeed, if $J'\not\subset L_{p}$ then $\partial J'\cap J\setminus L_{p}\ne\emptyset$. As $\co_{f}^{+}(\partial J')\cap L_{p}\subset \co_{f}^{+}(\partial J')\cap J'=\emptyset$, it follows from Corollary~\ref{Corollaryimperfeito} that $\partial J'\cap\co^{-}_{f}(p)\ne\emptyset$. But as $J'$ is a nice interval, we get $\partial J'\cap \partial L_{p}\ne\emptyset$. But this is impossible by Lemma~\ref{renormalinks}.

\cqd

\begin{Definition}[$\infty$-renormalizable]We say that $f$ is $\infty$-renormalizable if $f$ has infinitely many different renormalization intervals.
\end{Definition}

\begin{Definition}[Essential periodic attractors] A periodic attractor $\Lambda$ is called essential if its local basin contain $c^{-}$ or $c^{+}$. Precisely, if $\exists\,p\in\Lambda$ such that $(p,c)$ or $(c,p)$ is contained in $\beta(\Lambda)=\{x\,;\,\omega_{f}(x)\subset\Lambda\}$ (the basin of $\Lambda$).
\end{Definition}

If a periodic attractor in not essential, it is called {\em inessential}. Note the, if $f$ is $C^{3}$ and has negative Schwarzian derivative then, by Singer's Theorem, $f$ does not admit inessential periodic attractors.

\begin{Proposition}
%[muito bom este lemma!]
\label{Lemma1110863}Suppose that $f:[0,1]\setminus\{c\}\to[0,1]$ is a $C^{2}$ non-flat contracting Lorenz map that does not admit inessential periodic attractors or weak repellers. 
If $J_{n}$ is a sequence of renormalization intervals with $J_{n}\supsetneqq J_{n+1}$ then $\bigcap_{n}J_{n}=\{c\}$.
\end{Proposition}
\dem

Let $J=\bigcap_{n}J_{n}$. Write $(a,b)=\interior J$. Suppose for example that $a\ne c$ (the case $b\ne c$ is analogous). Given $x\in(a,b)$ let $R(x)=\min\{j>0\,;\,f^{j}(x)\in(a,b)$. As $J_n=(a_n,b_n)$ are renormalization intervals, then $(a_n,c)$ only returns to  $J_n$ at $\period(a_n)$ (and $(c,b_n)$ at the period of $b_n$), that is, the first return is at the time $\period(a_n)$. So, as $R(x)\ge\min\{\period(a_n),\period(b_{n})\}\to\infty$. Thus $R(x)=\infty$ $\forall\,x\in(a,b)$. 
As $f^{j}((a,c))\cap(a,b)=\emptyset$ $\forall\,j>0$ (because $R\equiv\infty$) then $f^{j}|_{(a,c)}$ is a homeorphism $\forall\,j$. By Lemma~\ref{LemmaStP}, $(a,c)$ is not a wandering interval. As $f$ does not have weak repellers, it follows from Lemma~\ref{homtervals} that there is a periodic attractor $\Lambda$ with $I\cap\Lambda\ne\emptyset$. As $f$ does not have inessential  periodic attractors, there is some $q\in\Lambda$ such that $(q,c)$ or $(c,q)\subset\beta(\Lambda)$
As $q$ is periodic, $q\notin[a,b]$. Thus, $q<a_{n}<c$ for some $n$ or $c<b_{n}<q$. In any case, we get an absurd.

\cqd

\begin{Corollary}\label{Corolary989982}
Suppose that $f:[0,1]\setminus\{c\}\to[0,1]$ is a $C^{2}$ non-flat contracting Lorenz map that does not admit inessential periodic attractors or weak repellers. If there exists $p\in(0,1)$ such that $\alpha_{f}(p)\ni c\notin\overline{\co_{f}^{+}(p)}$ then $f$ is not an infinitely renormalizable map.
\end{Corollary}

\dem

Suppose that $T_{n}$ is a sequence of two by two distinct renormalizable intervals. By Proposition~\ref{Lemma1110863}, $\bigcap_{n}T_{n}=\{c\}$.  
For each $n\in\NN$, let $0<r_n,\ell_n\in\NN$ be such that $f^{\ell_n}(T_n\cap(0,c))\subset T_n$ and $f^{r_n}(T_n\cap(c,1))\subset T_n$ and let $$U_{n}=T_n\cup\bigg(\bigcup_{j=1}^{\ell_n-1}f^j((T_n\cap(0,c))\bigg)\cup\bigg(\bigcup_{j=1}^{r_n-1}f^j((T_n\cap(c,1))\bigg).$$

If $p\in U_{n}$ $\forall\,n$ then $c\in\omega_{f}(p)$, contradicting our hypothesis. Thus, one can find some $n\ge0$ such that $p\notin U_{n}$. But this is an absurd, because $c\in\alpha_{f}(p)$ and so, $\co_{f}^{-}(p)\cap T_{n}\ne\emptyset$.

\cqd

\begin{Theorem}[The Solenoid attractor]\label{Lemma88609}
Let $f:[0,1]\setminus\{c\}\to[0,1]$ be a $C^{2}$ non-flat contracting Lorenz map without inessential periodic attractors or weak repellers. If $f$ is $\infty$-renormalizable then there is a compact minimal set $\Lambda$, with $c\in\Lambda\subset\bigcap_{J\in\cR }K_J$ such that $\omega_f(x)=\Lambda$ $\forall\,x\in[0,1]$ with $c\in\omega_f(x)$, where $\cR $ is the set of renormalization intervals of $f$ and $K_J$ is the nice trapping region of $J\in\cR $.
\end{Theorem}
\dem
As there are no inessential periodic attractors or weak repellers and as $f$ is $\infty$-renormalizable, we conclude that $f$ does admit periodic attractors.

Write $\cR =\{J_n\}_{n\in\NN}$, with $J_1\supsetneqq J_2\supsetneqq J_3\supsetneqq\cdots$.
Note that $J_n\supset\overline{J_{n+1}}$ $\forall\,n$ and also
\begin{equation}
K_{J_n}=\interior(K_{J_n})\supset\overline{K_{J_{n+1}}}\;\;\forall\,n\in\NN.
\end{equation}
Thus, $$\Delta:=\bigcap_{n\in\NN}\overline{K_{J_n}}=\bigcap_{n\in\NN}K_{J_n}.$$

As each $K_n$ is a trapping region ($f(K_n)\subset K_n$), it is easy to see that $\omega_f(x)\subset\Delta$, whenever $c\in\omega_f(x)$. Indeed, if  $c\in\omega_f(x)$ then $\co^+_f(x)\cap J_n\ne\emptyset$ for every $n\in\NN$, because $\{c\}=\bigcap_{n}J_n$ (see Proposition~\ref{Lemma1110863}). Thus $\omega_f(x)\subset\overline{K_n}$ $\forall\,n$.

Let $\ck_n$ be the collection of connected components of $K_{J_n}$ and $\ck_n(y)$ be the element of $\ck_n$ containing $y$ (Figure~\ref{KnBasica.png}), for any given $y\in\Delta$.
Let $\Lambda$ be the (closed) set of $y\in\Delta$ such that there is a sequence $\Delta\ni y_n\to y$ and $\NN\ni k_n\to\infty$ with $\lim_n\diameter(\ck_{k_n}(y_n))=0$. 
Given any $x\in[0,1]$ with $c\in\omega_f(x)$, we have $\co^+_f(x)\cap J_n\ne\emptyset$  $\forall\,n\in\NN$ and,  by Proposition~\ref{Lemma1110863}, $\co^+_f(x)$ intersects every element of $\ck_n$ $\forall\,n$. As a consequence, $\co^+_f(x)$ is dense in $\Lambda$.  That is,
\begin{equation}\label{eq6783453}
\Delta\supset\omega_f(x)\supset\Lambda\text{\, for every $x$ such that }c\in\omega_f(x).
\end{equation}

%%%%%%%%%%%%%%%%%%%%%%%%%%%%%%%%%%%%%%%%%%%%%%
\begin{figure}
\begin{center}\label{KnBasica.png}
\includegraphics[scale=.2]{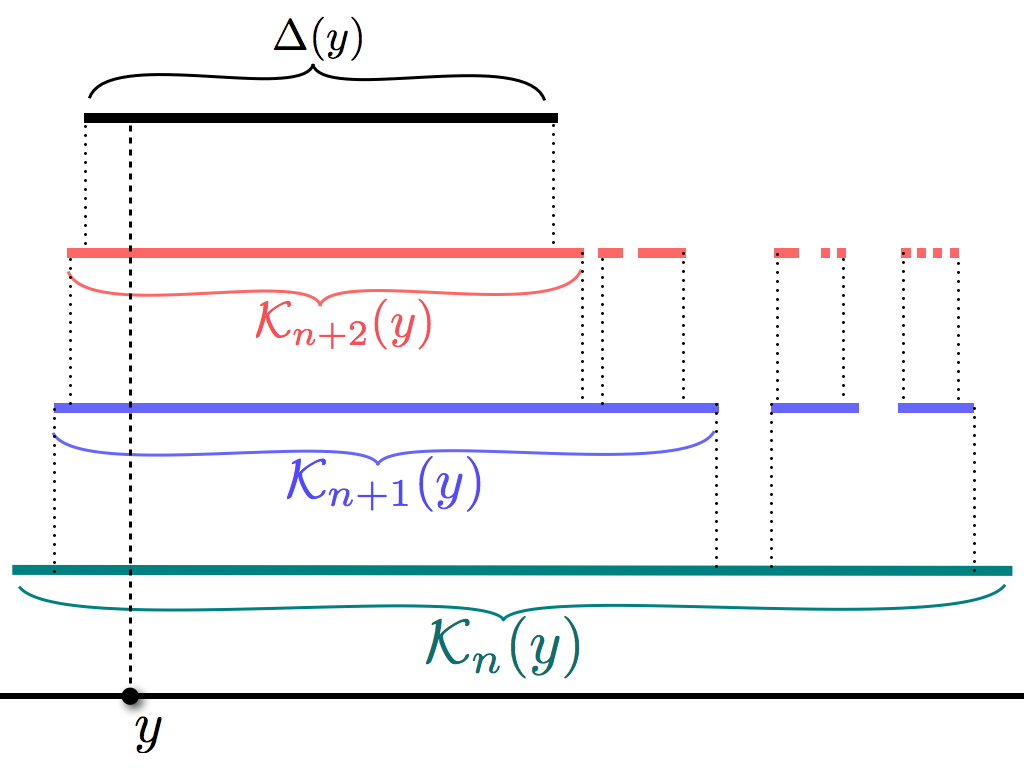}\\
\caption{}\label{KnBasica.png}
\end{center}
\end{figure}
%%%%%%%%%%%%%%%%%%%%%%%%%%%%%%%%%%%%%%%%%%%%%%

\begin{claim} Define $\Delta(y)$ as the connected component of $\Delta$ containing $y$.
If $\interior(\Delta(y))\ne\emptyset$, $y\in\Delta$, then $\interior(\Delta(y))$ is a wandering interval.
\end{claim}
\dem[Proof of the claim]

If $f^s(\Delta(y))\cap J_n\ne\emptyset$ then $f^s(\ck_n(y))\cap\ck_n(y)\ne\emptyset$ and so, $f^s(\Delta(y))\subset f^s(\ck_n(y))\subset J_n$.
Thus, if $c\in f^s(\Delta(y))$ then we will have $f^s(\Delta(y))\subset \bigcap_nJ_n=\{c\}$ (Proposition~\ref{Lemma1110863}), an absurd. This implies, that $c\notin f^s(\Delta(y))$ $\forall s\in\NN$. From Lemma~\ref{homtervals}, we get that $\interior(\Delta(y))$ is a wandering interval.
\cqd

Now consider $y\in\Delta\setminus\Lambda$. We will show that if $c \in \omega_f(x)$, then $y \not\in \omega_f(x)$. In this case there is some $\varepsilon>0$ such that $B_{\varepsilon}(y)\cap\Delta=B_{\varepsilon}(y)\cap\Delta(y)$.
Note that $\Delta(y)\ne\{y\}$, otherwise $\lim_n\diameter(\ck_n(y))=0$.
This implies that $\interior(\Delta(y))\ne\emptyset$ and so, by the claim above, $\interior(\Delta(y))$ is a wandering interval. This implies that $\omega_f(x) \cap \interior(\Delta(y))=\emptyset$, $\forall x$.

Thus, $B_{\varepsilon}(y)\cap\Delta\cap\Omega(f)=B_{\varepsilon}(y)\cap\Delta(y)\cap\Omega(f)\subset\{y\}$. Suppose that $y\in\omega_f(x)$ for some $x$ such that $c\in\omega_f(x)$.
In this case, as $\Delta\supset\omega_f(x)$, we conclude that $y$ is a isolated point of $\omega_f(x)$ (indeed, as $\omega_f(x)\subset\Delta\cap\Omega(f)$, we have $y \in \omega_f(x)\cap B_{\varepsilon}(y)=\omega_f(x)\cap B_{\varepsilon}(y)\cap\Delta\cap\Omega(f)=\omega_f(x)\cap B_{\varepsilon}(y)\cap\Delta(y)\cap\Omega(f)\subset\{y\}$, then this set is $\{y\}$).

%%%%%%%%%%%%%%%%%%%%%%%%%%%%%%%%%%%%%%%%%%%%%%
\begin{figure}
\begin{center}\label{Kn.pdf}
\includegraphics[scale=.3]{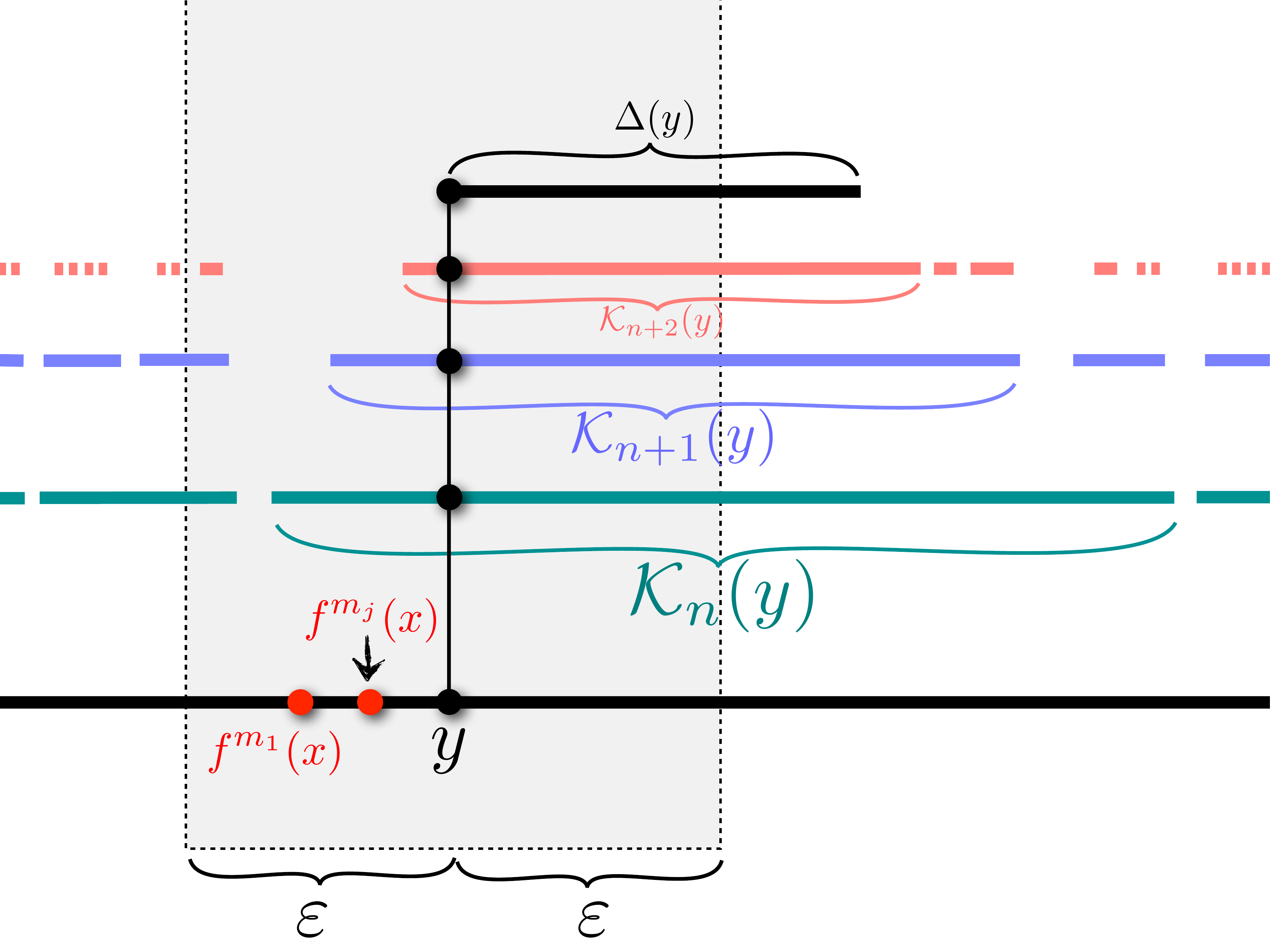}\\
\caption{}\label{Kn.pdf}
\end{center}
\end{figure}
%%%%%%%%%%%%%%%%%%%%%%%%%%%%%%%%%%%%%%%%%%%%%%

As $y\notin\interior(\Delta(y))$, we may suppose that $\Delta(y)=[y,b]$ (the case $\Delta(y)=[a,y]$ is analogous). Taking $\varepsilon>0$ small enough, we can assume that $y+\varepsilon<b$. Let $n\ge1$ be such that $y-\varepsilon<k_{n,0}(y)<y$, where $(k_{n,0},k_{n,1}):=K_n(y)$. Let $m_j\in\NN$ be such that $k_{n,0}<f^{m_1}(x)<f^{m_2}(x)<\cdots<f^{m_j}(x)\nearrow\,y$ and $\co_f^+(x)\cap(k_{n,0},y)=\{f^{m_1}(x),f^{m_2}(x),f^{m_3}(x),\cdots\}$ (Figure~\ref{Kn.pdf}). 

Choose $j_0$ big enough so that $m_j>m_1$ $\forall\,j\ge j_0$. Given $j\ge j_0$, let $I_j=(t_j,f^{m_1}(x))$ be the maximal interval contained in $(k_{n,0},f^{m_1}(x))$ such that $f^{m_j-m_1}|_{I_j}$ is a homeomorphism. If $k_{n,0}<t_j$, there is some $1\ge s<m_j-m_1$ s.t. $f^s((t_j,f^{m_1}(x)))=(c,f^{m_1+s}(x))$. By Corollary~\ref{CORaeroporto2} of the Appendix, $\#\co_f^+(x)\cap (c,f^{m_1+s}(x))=\infty$, but this will imply that $\#\co_f^+(x)\cap(k_{n,0},f^{m_1}(x))\ge\#\co_f^+(x)\cap(t_j,f^{m_j}(x))=\infty$. An absurd. Thus, $t_j=k_{n,0}$  and so, $I_j=(k_{n,0},f^{m_1}(x))\,\,\forall\,j\ge j_0.$

As a consequence, $f^j|_{(k_{_{n,0}},f^{m_1}(x))}$ is a homeomorphism $\forall\,j\in\NN$  (because $f^{m_j-m_1}|_{(k_{_{n,0}},f^{m_1})}=f^{m_j-m_1}|_{I_j}$ is a homeomorphism $\forall\,j\ge j_0$). But this contradicts the homterval lemma (Lemma~\ref{homtervals}), as $(k_{_{n,0}},f^{m_1}(x))$ cannot be a wandering interval ($k_{n,0}$ is pre-periodic) and as $f$ does not have periodic attractors.

 For short, if $c\in\omega_f(x)$ then $y\notin\omega_f(x)$ for all $y\in\Delta\setminus\Lambda$. So, by (\ref{eq6783453}), $\omega_f(x)=\Lambda$ when $c\in\omega_f(x)$. Finally, as $\Lambda\subset\bigcap_{J\in\cR }K_J$ and $c\in\omega_f(x)$ for every $x\in\bigcap_{J\in\cR }K_J$ then $\omega_f(x)=\Lambda$ $\forall\,x\in\Lambda=da$. That is, $\Lambda$ is minimal and we conclude the proof.

\cqd

\begin{Remark}\label{Remark12327890}
Let $x,y\in[0,1]\setminus\{c\}$, $\delta>0$ and $j\in\NN$. If $f^{j}|_{(y-\delta,y+\delta)}$ is an homeomorphim then $y\in\alpha_{f}(x)$ $\iff$ $f^{j}(y)\in\alpha_{f}(x)$.
\end{Remark}

\begin{Lemma}\label{Lemma01928373} Let $f:[0,1]\setminus\{c\}\to[0,1]$ be a $C^{2}$ non-flat contracting Lorenz map without periodic attractor. If $c\in\alpha_{f}(p)$ for some $p\ne c$ then $\co^{-}_{f}(p)\cap(-\delta,c)\ne\emptyset$ and $\co^{-}_{f}(p)\cap(c,\delta)\ne\emptyset$ $\forall\delta>0$.
\end{Lemma}
\dem
Suppose that $c\in\alpha_{f}(p)$, $p\in[0,c)\cup(c,1]$.  We may assume that $\co^{-}_{f}(p)\cap(-\delta,c)\ne\emptyset$ $\forall\delta>0$ (the other case is analogous). In this case we have to show that $\co^{-}_{f}(p)\cap(c,\delta)\ne\emptyset$ $\forall\delta>0$. Suppose by contradiction that $\co^{-}_{f}(p)\cap(c,\delta_{0})=\emptyset$ for some $\delta_{0}>0$. In this case, as $\alpha_{f}(p)$ is compact, there is some $q>0$ such that $(c,q)$ is a connected component of $[0,1]\setminus\alpha_{f}(p)$.
\begin{Claim}\label{Claim434344a}$f^{j}\big((c,q)\big)\cap(c,q)=\emptyset$ $\forall\,j>0$.
\end{Claim}
{\em Proof of the claim:} Suppose there is $j>0$ such that
\begin{equation}\label{Equation5393573894}
f^{j}\big((c,q)\big)\cap(c,q)\ne\emptyset.
\end{equation}
Let $\ell$ be the smallest $j$ satisfying (\ref{Equation5393573894}). In this case $f^{\ell}|_{(c,q)}$ is a homeomorphism.
If $f^{\ell}\big((c,q)\big)\subset(c,q)$ then $f$ admits a periodic attractor or a super-attractor, contradicting our hypothesis. Thus there is some $x\in\{c,q\}\cap f^{\ell}\big((c,q)\big)$. As both $c$ and $q$ are accumulated by pre-images of $p$, it follows that $x$ is also accumulated by pre-images of $p$. So, $\alpha_{f}(p)\cap(c,q)\ne\emptyset$  (Remark~\ref{Remark12327890}), contradicting that $(c,q)$ is contained in the complement of $\alpha_{f}(p)$. 
$\square$(end of the proof of the Claim~\ref{Claim434344a})

It follows from the claim above that $f^{j}|_{(c,q)}$ is a homeomorphism for every $j>0$. Moreover, $(c,q)$ is a wandering interval. Indeed, if $f^{j}\big((c,q)\big)\cap f^{k}\big((c,q)\big)\ne\emptyset$, with $j<k$, then $f^{j}\big((c,q)\big)\not\supset f^{k}\big((c,q)\big)$, because $f^{j}\big((c,q)\big)\supset f^{k}\big((c,q)\big)$ implies the existence of a periodic attractor or a super-attractor, contradicting again our hypothesis. Thus, there is $x\in\{f^{j}(c),f^{j}(q)\}$ belonging to $f^{k}\big((c,q)\big)$. As $f^{j}(c)$ and $f^{j}(q)\in\alpha_{f}(p)$ we get $\big(f^{j}|_{(c,q)}\big)^{-1}(x)\in\alpha_{f}(p)\cap(c,q)$ (Remark~\ref{Remark12327890}), contradicting again that $(c,q)$ is contained in the complement of $\alpha_{f}(p)$.

As $(c,q)$ being a wandering interval is a contradiction to Lemma~\ref{LemmaStP}, we have to conclude that $\co^{-}_{f}(p)\cap(c,\delta)\ne\emptyset$ $\forall\delta>0$.
\cqd

\begin{Definition}[Cherry-like]
We say that $f$ is a {\em Cherry-like map} if it does not have a periodic or super attractor and there is $\delta>0$  such that $c\in\omega_{f}(x)$ for every $x\in(c-\delta,c+\delta)$.
\end{Definition}

\begin{Lemma}\label{Lemma549164}
Let $f:[0,1]\setminus\{c\}\to[0,1]$ be a contracting Lorenz map that doesn't have weak repellers or periodic attractors. Let $p\in(0,1)$ be such that $c\notin\overline{\co_{f}^{+}(p)}$ and let $(p_{1},p_{2})$ be the connected component of $(0,1)\setminus\overline{\co_{f}^{+}(p)}$ containing the critical point $c$. If $f$ does not have a super-attractor then, given $y\in\co_{f}^{-}(p)$, $$\bigcup_{j\ge0}f^{j}(y,y+\varepsilon)\supset (p_{1},c)\,\,\text{ and }\,\,\bigcup_{j\ge0}f^{j}(y-\varepsilon,y)\supset (c,p_{2}),$$ $\forall\varepsilon>0$.\end{Lemma}

\dem

Given $\delta>0$, one can extend $f|_{[0,1]\setminus (c-\delta,c+\delta)}$ as a smooth $g$ on $[0,1]$ and apply Mañe's Theorem (note that $f$ does not have weak repellers) to $\co_{g}^{+}(x)=\co_{f}^{+}(x)$ $\forall\,x\in[0,1]$ with $x\in\Lambda_{\delta}=\{x\,;\,\co_{f}^{+}(x)\cap(c-\delta,c+\delta)=\emptyset\}$. Thus, this is an uniformly expanding set and so $Leb(\Lambda_{\delta}) = 0$. As a consequence, $\co^{+}_{f}(x)\cap (c-\delta,c+\delta)\ne\emptyset$ for Lebesgue almost all $x$.

As $\leb((p,p+\varepsilon))$ and $\leb((p-\varepsilon,p))>0$ 
whenever $\forall\varepsilon>0$, there are $j_{1},j_{2}\ge0$ such that $f^{j_{1}}|_{(p-\varepsilon,p)}$ and $f^{j_{2}}|_{(p,p+\varepsilon)}$ are homeomorphisms and $f^{j_{1}}((p-\varepsilon,p))\cap(c-\delta,c+\delta)\ne\emptyset\ne f^{j_{2}}((p,p+\varepsilon))\cap (c-\delta,c+\delta)$. Moreover, $f^{j_{2}}|_{(p,p+\varepsilon)}\supset(p_{1},c-\delta)$ and $f^{j_{1}}((p-\varepsilon,p))\supset(c+\delta,p_{2})$, because $\co_{f}^{+}(p)\cap(p_{1},p_{2})=\emptyset$ and $f^{j_{1}}|_{(p-\varepsilon,p)}$ and $f^{j_{2}}|_{(p,p+\varepsilon)}$ preserve orientation.

%{\color{red} ESSE LEMA PODE SER ESCRITO USANDO O LEMA DOS HOMTERVALS!!!!!!!!}

As a consequence, $$\bigcup_{j\ge0}f^{j}\big((p,p+\varepsilon)\big)\supset \bigcup_{\delta>0}(p_{1},c-\delta)=(p_{1},c)$$ and $$\bigcup_{j\ge0}f^{j}\big((p-\varepsilon,p)\big)\supset\bigcup_{\delta>0}(c+\delta,p_{2})=(c,p_{2}).$$

Suppose that $y\in f^{-s}(p)$ for some $s\ge1$. There is $r>0$ such that $f^{s}|_{(y,y+r)}$ and $f^{s}|_{(y-r,y)}$
are homeomorphisms. For instance, suppose that $f^{s}|_{(y,y+r)}$ is a homeomorphism. In this case, $f^{s}((y,y+r))=(p,p+\varepsilon)$ with $\varepsilon=f^{s}(y+r)-p$. Thus, $$\bigcup_{j\ge0}f^{j}\big((y,y+r)\big)\supset\bigcup_{j\ge0}f^{j}\big(f^{s}((y,y+r))\big)=\bigcup_{j\ge0}f^{j}\big((p,p+\varepsilon)\big)\supset(p_{1},c).$$
\cqd

Observe that the hypothesis in the lemma below does not restrict to the case of avoiding periodic attractors or weak repeller, and because of this it will be used at another point ahead in a context without any such restriction.

%PASSAR O LEMA ABAIXO PARA ANTES, NAO?

\begin{Lemma}\label{jotaxrenormaliza}
Let $f:[0,1]\setminus\{c\}\to[0,1]$ be a contracting Lorenz map. Write $v_{1}=\sup f([0,c))$ and $v_{0}=\inf f((c,1])$.
Given any $x$, $v_{0} < x < v_{1}$, let $J_{x}=(x_{1},x_{2})$ be the connected component of $[0,1]\setminus\alpha_{f}(x)$ that contains the critical point $c$. If $J_{x}\ne\emptyset$ then $J_{x}$ is a renormalization interval and $\partial J_{x}\subset\alpha_{f}(x)$.
\end{Lemma}
\dem
First note $\alpha_{f}(x)\supset\{0,1\}$ because, as $x \in (v_{0},v_{1})$, $0=\lim_{n\to\infty}(f|_{[0,c)})^{-n}(x)$ and $1=\lim_{n\to\infty}(f|_{(c,1]})^{-n}(x)$. Thus, $J_{x}$ is an open interval. Moreover, $\partial J_{x}\subset\alpha_{f}(x)$.

We claim that $J_{x}$ is a nice interval. Otherwise, consider $n$ the smallest integer $n>0$ such that $f^{n}(\partial J_{x})\cap J_{x}\ne\emptyset$. Let $i\in\{1,2\}$ be so that $f^{n}(x_{i})\in J_{x}$. As $f^{j}(x_{i})\notin J_{x}$ $\forall0\le j<n$, there is $\varepsilon>0$ such that $f^{n}|_{(x_{i}-\varepsilon,x_{i}+\varepsilon)}$ is a homeomorphism. From Remark~\ref{Remark12327890} it follows that $f^{n}(x_{i})\in\alpha_{f}(x)$, contradicting $\alpha_{f}(x)\cap J_{x}=\emptyset$. Thus, $J_{x}\in\cn$.

Now let us check that $J_{x}$ is a renormalization interval. If not, it follows from Lemma~\ref{Lemma09090863} that one can find a connected component $I=(t_{1},t_{2})$ of the domain of the first return map to $J_{x}$ such that $c\notin\partial I$. By Lemma~\ref{Lemma8388881a}, $f^{k}(I)=\cf_{J_{x}}(I)=J_{x}$, where $k=R_{J_{x}}(I)$. Note that $t_{1}$ or $t_{2}\in(x_{1},x_{2})$. Suppose that $t_{1}\in(x_{1},x_{2})$ (the case $t_{2}\in(x_{1},x_{2})$ is similar). As $c\notin\partial I$ (and $f^{j}(t_{1})\notin J_{x}$ $\forall0<j<k$), there is some small $\delta>0$ such that $f^{k}|_{(t_{1}-\delta,t_{1}+\delta)}$ is a homeomorphism. As $f^{k}(t_{1})=x_{1}\in\alpha_{f}(x)$, it follows from Remark~\ref{Remark12327890} that $t_{1}\in\alpha_{f}(x)$. But this is impossible as $\alpha_{f}(x)\cap J_{x}=\emptyset$.
\cqd

\begin{Corollary}\label{Cor111} Let $f:[0,1]\setminus\{c\}\to[0,1]$ be a $C^{2}$ non-flat contracting Lorenz map. Suppose that $f$ doesn't have a periodic attractor nor a weak repeller. If $p\in Per(f)$ then either
$\overline{\co^{-}_{f}(p)\cap(0,c)}\ni c \in\overline{(c,1)\cap\co^{-}_{f}(p)}$ or the connected  component of $[0,1]\setminus\alpha_{f}(p)$, $J_{p}$, is non empty and it is a renormalization interval.\end{Corollary}

\begin{Notation}[$\cl_{Per}$, $\cl_{Sol}$ and $\cl_{Che}$]
Let $\cl_{Per}$ denote the collection of contracting Lorenz maps having periodic attractors. The set of all $\infty$-renormalizable contracting Lorenz maps will be denoted by $\cl_{Sol}$ and let $\cl_{Che}$ be the set of all Cherry-like contracting Lorenz maps.

\end{Notation}

\begin{Lemma}
\label{vizinhanca}
Let $f:[0,1]\setminus\{c\}\to[0,1]$ be a $C^{2}$ non-flat contracting Lorenz map with no weak repeller and $f \notin \cl_{Per}\cup\cl_{Sol} \cup \cl_{Che}$ then $c\in\alpha_{f}(p)$ for some $p\in Per(f)$.
\end{Lemma}

\dem
If $f$ is not renormalizable let $I=(0,1)$, otherwise let $I=(a,b)$ be the smallest renormalization interval of $f$ (we are assuming that $f\notin \cl_{Per}\cup\cl_{Sol} \cup \cl_{Che}$).
By lemma \ref{noitenoite} we can pick a point $p \in (a,b)$ that is periodic. So, we have that $ p \in \alpha_{f}(p)$. As a consequence, it follows from Corollary~\ref{Cor111} that $\overline{\co_{f}^{-}(p)\cap(0,c)}\ni c \in\overline{(c,1)\cap \co_{f}^{-}(p)}$. Indeed, if the pre-orbit of $p$ is not accumulating on $c$ by both sides, then $J_{p}\ne\emptyset$ is a renormalization interval. In this case, as $p\in \alpha_{f}(p)$, we get $J_{p}\subsetneqq(a,b)$. This is an absurd,  as $(a,b)$ is the smallest renormalization interval.

\cqd

\begin{Proposition}[Long branches lemma]\label{Proposition008345678}Let $f:[0,1]\setminus\{c\}\to[0,1]$ be a $C^{2}$ non-flat contracting Lorenz map.
Suppose that $f$ does not admit a periodic attractor nor a weak repeller. If $\alpha_{f}(p)\ni c \notin\omega_{f}(p)$ for some $p\ne c$ then there exists $\varepsilon>0$ such that $\overline{ \co^{-}_{f}(x)\cap(0,c)} \ni c \in \overline{ \co^{-}_{f}(x)\cap(c,1)}$ for every $0<|x-c|<\varepsilon$. 

Moreover, $f$ is not $\infty$-renormalizable, $f$ is not Cherry-like and $Per(f)$ $\cap$ $(c-\delta,c)$ $\ne$ $\emptyset$ $\ne$ $Per(f)$ $\cap$ $(c,c+\delta)$ $\forall\delta>0$.
\end{Proposition}
\dem
Suppose by contradiction that the main statement is not true. That is, $c\in\overline{W}$, where $$W=\{x\,;\,c\notin  \overline{ \co^{-}_{f}(x)\cap(0,c)} \text{ or }c\notin \overline{ \co^{-}_{f}(x)\cap(c,1)}\}.$$

By Lemma~\ref{Lemma01928373}, if $\co^{-}_{f}(x)$ accumulates on one side of $c$ then  $\co^{-}_{f}(x)$ will accumulate on $c$ by both sides. Then, $W=\{x; c \notin \alpha_f(x)\}$.

Let $(p_{1},p_{2})$ be the connected component of $[0,1]\setminus\overline{\co_{f}^{+}(p)}$ that contains $c$.  
Choose a sequence $\co_{f}^{-}(p)\ni y_{n}\to c$. As $f$ does not have a periodic attractor, taking a subsequence if necessary, we get by Lemma~\ref{Lemma549164} that 
\begin{equation}\label{EQ2345654a}
\bigcup_{j\ge0}f^{j}\big((y_{n},y_{n}+\varepsilon)\big)\supset (p_{1},c)\,\,\forall\varepsilon>0,\forall\,n>0
\end{equation} and that
\begin{equation}\label{EQ2345654b}
\bigcup_{j\ge0}f^{j}\big((y_{n}-\varepsilon,y_{n})\big)\supset (c,p_{2})\,\,\forall\varepsilon>0,\forall\,n>0.
\end{equation}

As $c$ is accumulated by $W$, say by the left side (the other case is analogous), choose some $q\in(p_{1},c)\cap W$. It follows from (\ref{EQ2345654a}) that $\bigcup_{j\ge0}f^{j}\big((y_{n},c)\big)\supset (p_{1},c)\ni q$ $\forall\,n>0$ (we are taking $\varepsilon=|y_{n}-c|$ in (\ref{EQ2345654a})). Thus, there is a sequence $y_{n}<q_{n}<c$ and $i_{n}\to\infty$ such that $f^{i_{n}}(q_{n})=q$ $\forall\,n$. This implies that $c\in\alpha_{f}(q)$. But this is an absurd because $q\in W$.

Therefore, we can not have $c\in\overline{W}$ and this proves the main part of the Proposition. By Corollary~\ref{Corolary989982}, $f$ cannot be $\infty$-renormalizable and as $\omega_{f}(y)=\omega_{f}(p)\not\ni c$ for all $y\in\co^{-}_{f}(p)$, it follows that $f$ cannot be Cherry-like. Finally, let us show that $Per(f)\cap(c-\delta,c)\ne\emptyset\ne Per(f)\cap(c,c+\delta)$ $\forall\delta>0$. For this, let $J_{n}$ be the connected component of $(0,1)\setminus\bigcup_{j=0}^{n-1}f^{-j}(p)$ containing the critical point $0$. It easy to see that $J_{n}$ is a nice interval $\forall\,n$. As it follows from Lemma~\ref{LemmaHGFGH54} that  $Per(f)\cap\overline{J_{n}}\cap(-\infty,c)\ne\emptyset\ne(c,+\infty)\cap\overline{J_{n}}\cap Per(f)$ $\forall\,n$, we conclude the proof.

\cqd

Observe that is also true that $f$ being Cherry-like also implies $Per(f) \cap (u,v)=\emptyset$,  $(u,v)$ being the last interval or renormalization.  

It worths observing that the former definition given to Cherry-like maps is closely related to the usual one ({\em A Cherry map is a contracting Lorenz map $f:[0,1]\setminus\{c\}\to[0,1]$ such that the restriction of f to the invariant interval $[f(c_{+}), f(c_{-})]$ is semi-conjugated to a irrational rotation of a circle, the restriction to $(f(0_{+}), f(0_{-}))$ is injective but not necessarily surjective.}), according to the next Lemma. In it we will use the concept of {\em rotation interval} that generalizes the concept of rotation number for endomorphisms that was introduced by Newhouse, Palis and Takens in \cite{NPT83}. Based on their construction, Tresser and Gambaudo \cite{CGT84, GT85}  adapted it to the discontinuous and non-injective maps setting.

%%%%%%%%%%%%%%%%%%%%%%%%%%%%%%%%%%%%%%%%%%%%%%
\begin{figure}[H]
\begin{center}\label{CherryMap.pdf}
\includegraphics[scale=.25]{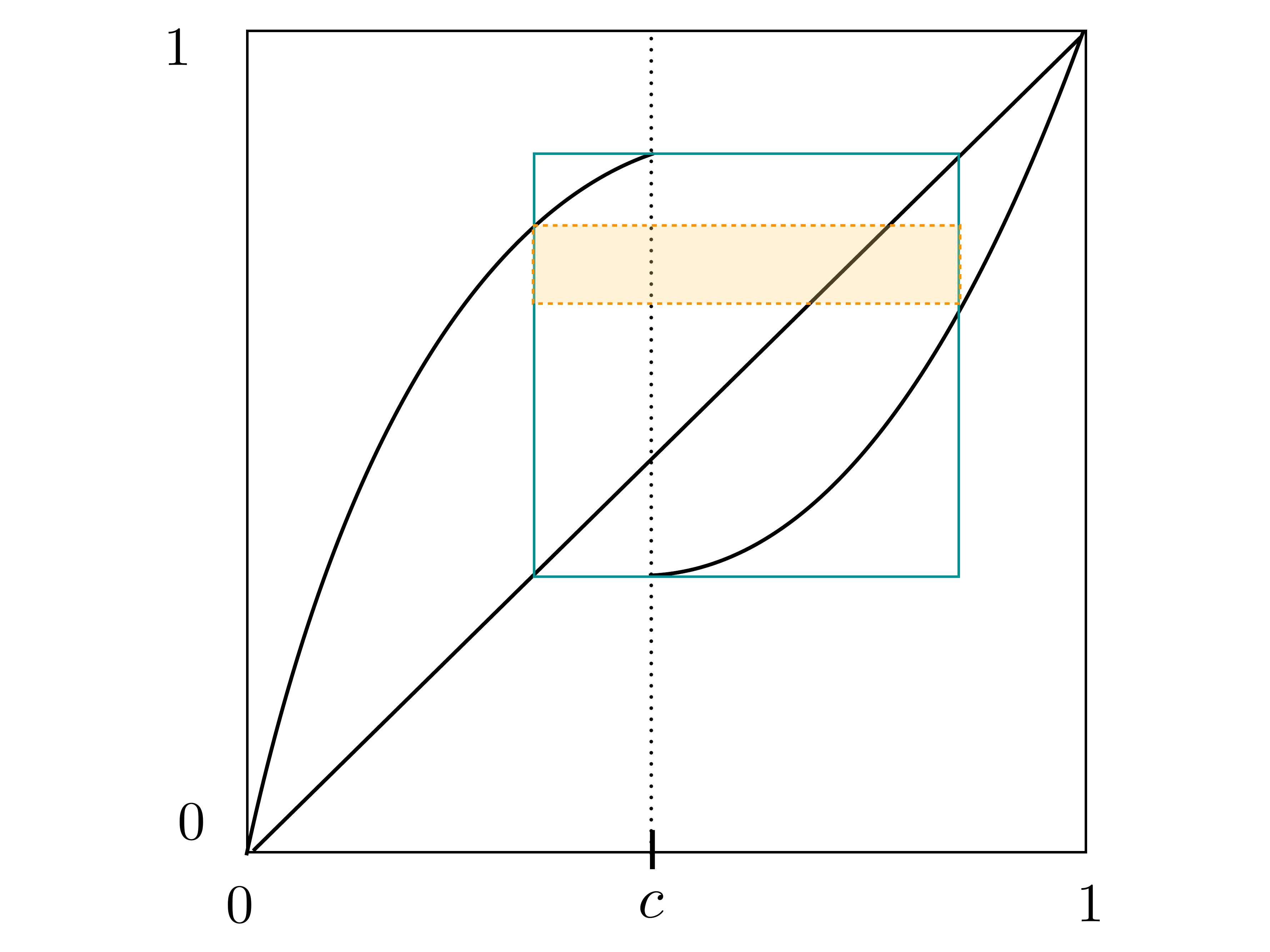}\\
\end{center}
\end{figure}
%%%%%%%%%%%%%%%%%%%%%%%%%%%%%%%%%%%%%%%%%%%%%%

\begin{Lemma}
\label{rotirra}Let $f:[0,1]\setminus\{c\}\to[0,1]$ be a contracting Lorenz map.
$f$ is Cherry-like if and only if there is a renormalization interval $I$ such that $rot(F)\cap \QQ =\emptyset$, where $F(x)=f^{R(x)}(x)$ is the first return map to $\overline{I}$ and $rot$ denotes the rotation interval of $f$.% That is, $f$ can be renormalized to a Cherry map.
\end{Lemma}

\dem
As $f$ is a Cherry-like, then $Per(f)\cap(0,1)=\emptyset$.
This implies that $rot(f|_{(v_0,v_1)}) \cap \QQ =\emptyset$, for if $\exists \gamma \in rot(f)$ then $\exists x$ such that $rot(x)=\gamma$ and then if $\gamma \in \QQ \cap rot(f)$ there shall exist $x$ periodic with this rotation number.
The converse is also true, for if $f$ is not Cherry-like, as we observed, $Per(f)\cap (0,1)\ne \emptyset$ (up to a renormalization, perhaps), and this gives $Per(f)|_{(v_0,v_1)}\ne\emptyset$. But then, $rot(f|_{(v_0,v_1)})\cap \QQ\ne \emptyset$
\cqd

In Section~\ref{SecAtCherry-like} in the Appendix, we study the attractors of the Cherry-like maps.

\newpage

\section{Topological Attractors}

This section is divided in two subsection. In the the first one, {\em The structure of the Topological Attractors} (Section~\ref{TheSofTA}),
we will be studying the topological attractors for the contracting Lorenz maps. The core result of this subsection will be Theorem \ref{cicloint}, from which in the second subsection (Section~\ref{ProofABCD}), we will obtain the main theorems: Theorem \ref{baciastopologicas}, \ref{teoalfalim} and \ref{atratortopologico}.

\subsection{The structure of the Topological Attractors}
\label{TheSofTA}

In all Section~\ref{TheSofTA}, $f$ will be a $C^{2}$ non-flat contracting Lorenz map $f:[0,1]\setminus\{c\}\to[0,1]$.
\vspace{0.05cm}

\begin{Lemma}
\label{216}
If $f$ does not have weak repellers or periodic attractors then 
$$
\alpha_f(x) \ni c \Rightarrow \alpha_f(x) \supset \Omega(f)
$$
\end{Lemma}
\dem
Let $x$ such that $\alpha_f(x) \ni c$ and given $y \in \Omega(f)$ consider any neighborhood $T$ of $y$. As $y$ is non-wandering, there is $z \in T$, (we may assume $z \not \in \co^-_f (c) \cup \co^-_f(Per(f))$) and $j \in \NN$ such that $f^j(z) \in T$. 

We claim that there is $s \in \NN$ such that $f^s((z,f^j(z))) \ni c$. Indeed, if $f^s((z,f^j(z))) \not \ni c$ for any $s$, then $f^t|_{[z,f^j(z)]}$ is a well defined homeomorphism, for any $t$. Observe that  $\bigcup_{s \ge 0}f^{js}(z,f^j(z)]$ is an interval, say, $(z,q)$. Moreover $f^j((z,q)) \subset (z,q)$. But this would imply the existence of a periodic or super attractor, what proves the claim.
 
Let $s$ be as above. As $x$ is such that $\alpha_f(x) \ni c$, we have that $\co^-_f (x) \cap f^s(z,f^j(z)) \ne \emptyset$ and then $\co^-_f (x) \cap T \supset \co^-_f (x) \cap (z,f^j(z)) \ne \emptyset$
As $T$ is any neighborhood, we can conclude $y \in \alpha_f(x)$.
 
\cqd

For a Lorenz map $f \notin \cl_{Per}\cup\cl_{Sol} \cup \cl_{Che}$ without any weak repeller, let us define $$\EE = \{x \in (0,1); \alpha_{f}(x) \ni c\}.$$
By Lemma~\ref{vizinhanca} and Proposition~\ref{Proposition008345678}, $\EE$ contains a neighborhood of $c$. In the next lemma, consider $(a,b) \subset \EE$ to be the maximal interval containing $c$.

\begin{Lemma}
$\exists \ell$ and $r>0$ such that $f^\ell((a,c))\subset(a,b)\supset f^r((c,b))$
\end{Lemma}

\dem
Suppose $\forall j>0,j \in \NN$, is such that $f^j((a,c))\cap (a,b)= \emptyset$. We have that $f((a,c))$ is a positive measure set that is contained in the set of points that never visit the neighborhood $(a,b)$ of the critical point. By Ma\~ne's Theorem, this set has Lebesgue measure zero. Then there is a minimum k such that $f^k((a,c))\cap (a,b) \ne \emptyset$.

Suppose now that $f^k((a,c))\not\subset (a,b)$. For example, $b \in  f^k((a,c))$. Then $\exists y \in f^k((a,c))$ such that $y \not \in \EE$. But in this case there is a pre-image $w$ of $y$ in $(a,b)$, that is a subset of $\EE$, so, as $\alpha_f(w) \ni c$, $\alpha_f(y) \ni c$, absurd. 
\cqd

For a Lorenz map $f \notin \cl_{Per}\cup\cl_{Sol} \cup \cl_{Che}$ without any weak repeller and $\ell$ and $r$ as given by the former lemma, we define 
\begin{equation}\label{regiaoarmadilha}
\UU = (a,b) \cup \big(\bigcup_{j=1}^{\ell-1}f^j((a,c))\big) \cup \big(\bigcup_{j=1}^{r-1}f^j((c,b))\big)\ni c
\end{equation}
and we have that $\UU$ is a trapping region, that is, $f( \UU \setminus \{c\}) \subset \UU$.

It worths observing that given a non-renormalizable Lorenz map $f$ defined in an interval $I$ such that $f$ has a trapping region, any point in this interval eventually reaches this trapping region when iterated by $f$. Also, the non-wandering set within this interval is necessarily inside $\UU$ (if this was not true, we would have a non-wandering point $y$ out of the trapping region that would reach it, by the former observation, and couldn't leave this region to be back - and it should, as it is non-wandering).

\begin{Lemma} Let $f$ be a non-renormalizable Lorenz map  defined in $I$ and  $\UU \subset I$ so that $f(\UU)\subset \UU$, then $\forall x \in I$ $\exists k >0$ such that $f^k(x) \in \UU$.  
\end{Lemma}

\dem

Suppose there is a point $x \in I$ such that $f^j(x)\not\in \UU \forall j>1$.
We consider the set $\omega_{f}(x)$, $\omega_{f}(x) \cap \UU = \emptyset$, and this is an invariant expanding set. So we have periodic points inside it, say $p$, and $p$ and its orbit never visit $\UU$. Consider now $\alpha_{f}(p)$, this set contains $p$ and so $J_p$ can't be $I$ and it can't also be an empty set, as there couldn't be the case of $\alpha_{f}(p)\cap \UU \ne \emptyset$, as $\alpha_{f}(p)$ couldn't have points in a trapping region as $p$ is not within it. So, $J_x$ is a smaller renormalization interval within $I$, contradiction to the fact that $f$ was non-renormalizable.

\cqd

\begin{Corollary}
\label{sobrealfa}
For $f \notin \cl_{Per}\cup\cl_{Sol} \cup \cl_{Che}$ without any weak repeller we have that 
$\alpha_{f}(x) \supset \Omega(f)$ $ \forall x \in \UU$.
\end{Corollary}
\dem
As Lemma \ref{216} says $\alpha_f(x) \supset \Omega(f)$ to any $x$ such that $c \in \alpha_f(x)$, this holds for any $x$ in $\UU$, as this is contained in $\EE$.
\cqd

\begin{Lemma}For $f \notin \cl_{Per}\cup\cl_{Sol} \cup \cl_{Che}$ without any weak repeller, if $\alpha_f(x) \ni c$ then $\alpha_f(x)\cap\UU\subset\Omega(f)\cap\UU$.
\end{Lemma}

\dem
Consider $x$ such that $\alpha_f(x) \ni c$. Given $y \in \alpha_f(x)$ consider any neighborhood $V$ of $y$. We may assume $V \subset \UU$. 

\begin{Claim}
$y \in \overline{(V\setminus\{y\})\cap\co^-_f(x)}$.
\end{Claim}
\dem
On the contrary, $\exists \epsilon >0$ such that $B_\epsilon(y)\cap\co^-_f(x)=\{y\}$. In this case, we have that $\exists n_1<n_2<...<n_j \to \infty$ such that $f^{n_j}(y)=x$. 
Then $$x=f^{n_2}(y)=f^{n_2-n_1}(f^{n_1}(y))=f^{n_2-n_1}(x).$$

% VER SE SIMPLIFICA COMO HOMTERVAL

Observe that if $f^s(B_\epsilon(y))\not\ni c\, \forall s$ then writing $(\alpha, \beta)=f^{n_1}(B_\epsilon(y))$ we have
$$x \in (\alpha,\beta) \text{ and } f^{k(n_2-n_1)}((\alpha,\beta))\not\ni c \forall k.$$
Taking $(x,\gamma)=\bigcup_{k\ge1}f^{k(n_2-n_1)}((x,\beta))=\bigcup_{k\ge1}(x,f^{k(n_2-n_1)}(\beta))$ we have $f^{n_2-n_1}|_{(x,\gamma)}$ homeomorphism and $f^{n_2-n_1}((x,\gamma))\subset(x,\gamma).$

But this would imply the existence of attracting periodic orbits, that are considered not to exist. Then, we necessarily have that $\exists s$ such that $f^s(B_\epsilon(y))\ni c$. 

As $c\in \alpha_{f}(x)$ we would have that $\#\co_f^-(x)\cap B_\epsilon(y)=\infty$. Again an absurd, proving the claim.
\cqd
Because of the claim we may assume that $y\in \overline{(y,1)\cap V\cap \co^-_f(x)}$ (the proof for the case $y\in$ $\overline{(0,y)\cap V\cap \co^-_f(x)}$ is analogous).

We may take $x_2<x_1\in(y,1)\cap V \cap\co_f^-(x)$ such that $f^{n_2}(x_2)=x=f^{n_1}(x_1)$ with $n_1<n_2$.

\begin{Claim}
$\exists s \in \NN$ such that $f^s([x_1,x_2))\ni c$
\end{Claim}
\dem
If $c\notin f^s([x_1x_2))$ $\forall\,s\ge0$ then $$f^{k(n_2-n_1)}([f^{n_2-n_1}(x),x))=f^{k(n_2-n_1)+n_2}([x_1,x_2))\not\ni c,\; \forall k \in \NN.$$

As $f$ preserves orientation, $f^{k(n_2-n_1)}|_{[f^{n_2-n_1}(x),x)}$ is a homeomorphism $\forall x$, so we have $f^{(k+1)(n_2-n_1)}(x)<f^{k(n_2-n_1)}(x) \, \forall k>0.$

Then, $\bigcup_{k\ge0}f^{k(n_2-n_1)}([f^{n_2-n_1}(x),x))$ is an interval $(\gamma,x)$. Besides that, $f^{n_2-n_1}|_{(\gamma,x)}$ is a homeomorphism and $f^{n_2-n_1}((\gamma,x))\subset(\gamma,x)$. 

But this is an absurd, because it would imply the existence of attracting periodic orbits.
\cqd

Let $s\in \NN$ such that $f^s([x_1,x_2))\ni c$. As $x_1 \in \UU$, we have that $\co^-_f(x_1)$ accumulates in $c$ by both sides. Then, $\co_f^-(x_1)\cap f^s([x_1,x_2))\ne \emptyset$.

This implies that $\exists x_1' \in \co^-_f(x_1)\cap[x_1,x_2)\subset V$, say $x_1' \in f^{-t}(x_1)\cap V$. Then,
$$
f^t(V)\cap V \ne \emptyset
$$

As $V$ is a neighborhood of $y \in \UU$ that was arbitrarily taken, we may conclude that $y \in \Omega(f)$.

\cqd

\begin{Corollary}
\label{poiupoiu} For $f \notin \cl_{Per}\cup\cl_{Sol} \cup \cl_{Che}$ without any weak repeller,
$\alpha_{f}(x)\cap \UU = \Omega(f)\cap \UU, \forall x \in \UU$.
%In particular, given any $x \in \Lambda$, $\overline{\co^-_f(x)\cap\Lambda}=\Lambda$. 
\end{Corollary}

\begin{Corollary}
\label{compon} For $f \notin \cl_{Per}\cup\cl_{Sol} \cup \cl_{Che}$ without any weak repeller, then any connected component of $\UU \setminus \Omega(f)$ is a wandering interval.
\end{Corollary}

\dem
Let $J=(a,b)$ connected component of $\UU \setminus \Omega(f)$. Suppose it is not a wandering interval. Then, Lemma~\ref{homtervals} says there will be an $n$ for which $f^n(J)\ni c$. Lemma~\ref{vizinhanca} and Proposition~\ref{Proposition008345678} say there are plenty of points with $c$ in their $\alpha$-limits  inside this set $f^n(J)$. We know $f^{-1}(\alpha_{f}(x))\subset\alpha_{f}(x)$ and then Corollary~\ref{poiupoiu} insures us these points are in $\Omega(f)$, but they are inside $J$, that should not contain any point of $\Omega(f)$.
\cqd

\begin{Definition}[Strong Topological Transitivity]\label{StTopTrans} Let $\XX$ be a compact metrical space. Given a continuous map $g:A\subset\XX\to\XX$, we say it is strongly (topologically) transitive if for any open set $V\subset\XX$ with $V\cap A\ne\emptyset$, we have $\bigcup_{j \ge 0}g^j(V)=A$.

Let us make precise the notation used in this definition: given $V\subset\XX$, let $g^{-1}(V)=\{x\in A\,;\,g(x)\in V\}$. We define inductively $g^{-n}(V)$, for $n\ge2$, by $g^{-n}(V)=g^{-(n-1)}(g^{-1}(V))$. We define for $n\ge1$, $g^{n}(V)=\{g^{n}(v)\,;\,v\in V\cap g^{-n}(A)\}$. 
\end{Definition}

\begin{Proposition} \label{transitivo} If $f \notin \cl_{Per}\cup\cl_{Sol} \cup \cl_{Che}$ and $f$ doesn't have any weak repeller, then    
$f|_{\Omega\cap\UU}$ is strongly transitive.
In particular, $$f|_{\Omega\cap\UU}\text{ is transitive.}$$
\end{Proposition}

\dem
We know that $f^{-1}(\alpha_{f}(x))\subset\alpha_{f}(x)$. 

We will show that $\bigcup_{j\ge0}f^j(V\cap\Omega(t)) = \Omega(f)\cap \UU, \forall V  \subset \UU$, $V$ open and $V \cap \Omega(f) \ne \emptyset$.

It follows from the Corollary \ref{poiupoiu} that 
\begin{equation}
\label{inclusaoestrela}
f^{-1}(\Omega(f)\cap\UU)\cap\UU\subset\Omega(f)\cap\UU.
\end{equation}

Let $V\subset\UU$, $V$ any open set with $V \cap \Omega(f) \ne \emptyset$. Given $x \in \Omega(f)\cap\UU$ we have that $\alpha_{f}(x)\cap V\ne \emptyset$ and then $\co^-_f(x)\cap V \ne \emptyset$. Pick $x_t \in f^{-t}(x) \cap V$. Define $x_k=f^{t-k}(x_t)$ for $0 \le k \le t$. 
$$
x_t \stackrel{f}{\to} x_{t-1} \stackrel{f}{\to} \dots \stackrel{f}{\to} x_0=f^t(x_t) 
$$
As $\UU$ is a trapping region, we have that $x_k$ in $\UU, \forall 0 \le k \le t$. 

We claim that indeed $x_t \in \Omega(f)\cap\UU$. 

We have that $x_0 \in \Omega(f)\cap\UU$. Suppose it also worths for $k-1$, that is, $x_{k-1} \in \Omega(f)\cap\UU$. We have that $x_k \in \UU$. Then $x_k \in f^{-1}(x_{k-1})\cap \UU$ and by (\ref{inclusaoestrela}) we have that $x_k \in \Omega(f)\cap \UU$.
It follows by induction that $x_t \in \Omega(f)\cap\UU$.

\cqd

\begin{Theorem} 
\label{cicloint}
Let $f:[0,1]\setminus\{c\}\to[0,1]$ be a $C^{2}$ non-flat contracting Lorenz map without any weak repeller. 
If $f$ doesn't have a periodic attracting orbit, isn't Cherry-like nor $\infty$-renormalizable, then there is an open trapping region $U\ni c$ given by a finite union of open intervals such that $\Lambda:=\overline{U\cap\Omega(f)}$ satisfies the following statements.

\begin{enumerate}

\item$\omega_{f}(x)=\Lambda$ for a residual set of points of $\Lambda$ (in particular, $\Lambda$ is transitive).  
\item The basin of attraction of $\Lambda$, $\beta(\Lambda):=\{x;\omega_{f}(x)\subset \Lambda\}$, is an open and dense set of full measure.
\item $\exists \lambda >0$ such that $\lim_{n\to\infty} \frac{1}{n}\log |Df^n(x)|=\lambda$ for a dense set of points $x$ in $\Lambda$.
\item \label{opcoes} either $\Lambda$ is a finite union of intervals or it is a Cantor set.
\item if $\Lambda$ is a finite union of intervals then $\omega_{f}(x)=\Lambda$ for a residual set of $x$ in $[0,1]$.
\item \label{cantor} $\Lambda$ is a Cantor set if and only if there is a wandering interval.

\end{enumerate}
\end{Theorem}

\dem Set $\Lambda:=\overline{\Omega(f) \cap\UU}$ with $\UU$ as defined in (\ref{regiaoarmadilha}).
\begin{enumerate}

\item Lemma~\ref{dicotomia} of Appendix insures us it is true, as we have transitivity provided by Proposition~\ref{transitivo}.
\item By Ma\~ne (see Theorem~\ref{ThMane} and \ref{lebzero}), the set of points that never visit $\UU$ has measure zero, then empty interior, so its complement $\cu$ is a full measure set, open and dense. We claim that any point $y$ in this set $\cu$ is also in $\beta(\Lambda)$. For some $k$, $f^k(y)=x \in \UU$, and we have two possible situations for a point $q\in\omega_{f}(x)=\omega_f(y)$. As $\UU$ is a trapping region, $q$ can be an interior point of $\UU$, and then it automatically belongs to $\Lambda=\overline{\Omega \cup \UU}$. If not an interior point, $q\in\partial\UU$. In this case, as $q\in\omega_{f}(x)$, there are infinitely many $f^{n_j}(x)$ accumulating in $q$. Then, there can be no wandering interval with border $q$ (as images of $x$ keep coming close to $q$). By Corollary~\ref{compon}, as $q$ can't be in the border of a wandering interval, it is not in the border of a connected component of $\UU\setminus\Omega(f)$, then it is accumulated by points of this set, that is, $q\in \Lambda=\overline{\Omega(f) \cap\UU}$.
\item Proposition (\ref{Proposition008345678}) says that repeller points $p \in Per(f)$ accumulate in $c$. As they are in $\Omega(f)$, the ones that are in $\UU$ are also in $\Lambda$, and it follows from Corollary(\ref{poiupoiu}) that $\co^-_f(p)$ is dense in $\Lambda$. 
Given any point $x\in\co^-_f(p)$, as it is eventually periodic, say $f^j(x)=p$, we have that
as $$\lim_{n\to\infty} \frac{1}{n}\log |D(f^{n-j}\circ f^j)(x)|$$
$$=\lim_{n\to\infty} \frac{1}{n}\log\big(|Df^{n-j}(f^{j}(x))||Df^j(x)|\big)$$
$$=\lim_{n\to\infty} \frac{1}{n}\log\big(|Df^{n-j}(p)|\big)+\lim_{n\to\infty} \frac{1}{n}\log|Df^j(x)|$$
$$=\lim_{n\to\infty} \frac{1}{n}\log\big(|Df^{n-j}(p)|\big)$$
$$=\lim_{n\to\infty} \frac{n-j}{n(n-j)}\log\big(|Df^{n-j}(p)|\big)$$
$$=\lim_{n\to\infty} \frac{1}{(n-j)}\log\big(|Df^{n-j}(p)|\big)=\exp_f(p)=:\lambda.$$

\item \label{item4} As $\Lambda$ is transitive, $\exists x \in \Lambda=\omega_{f}(x)$, then, by Lemma~\ref{aeroporto} of Appendix, it is a perfect set.
We have two possibilities. $\interior(\Lambda)= \emptyset$ or not.
As $\Lambda$ is a subset of $\RR$, if it has empty interior, it is totally disconnected. Consequently, it will be a Cantor set (as we already proved it is compact and perfect).
Suppose then $\interior(\Lambda) \ne \emptyset$. Let $I$ be an open interval, $I \subset \Lambda$ and it can't be a wandering interval, as it is a subset of $\Lambda \subset \Omega(f)$. Then, by Lemma~\ref{homtervals}, $\exists j$ such that $f^{j}(I)\ni c$, and so, $c \in \interior\Lambda$. This forbids the existence of wandering intervals. Indeed, if there is a wandering interval $J$, it has to accumulate in the critical point (by Lemma~\ref{LemmaStP}), but this would imply that $f^{n}(J)\cap\Omega(f)\ne\emptyset$ for $n$ sufficiently big. An absurd. So, as we cannot have wandering intervals, Corollary~\ref{compon}, $\UU\setminus\Omega(f)$ has to be an empty set. As $\UU$ is an orbit of intervals, it proves the claim of the Theorem.
\item Let $\Lambda'=\{x\in\UU;\omega_{f}(x)=\Lambda\}$.
Observe that $x \in \bigcup_{j\ge0}f^{-j}(\Lambda')$ implies that $\omega_{f}(x)=\Lambda$.
As $\Lambda'$ is residual, there exist $A_n$, $n\in\NN$ open and dense sets in $\UU$ such that $\Lambda'=\bigcap_{n\in \NN}A_n$.
On the other hand, for every $n\in\NN$ we have that $\bigcup_{j\ge0}f^{-j}(A_n)$ is an open dense set in $[0,1]$.
Then, $\bigcap_{n\in\NN}\big(\bigcup_{j\ge0}f^{-j}(A_n)\big)$ is residual in $[0,1]$.
So we have that $\bigcup_{j\ge0}f^{-j}\big(\Lambda'\big)=\bigcup_{j\ge0}f^{-j}\big(\bigcap_{n\in\NN}A_n\big)=\bigcap_{n\in\NN}\big(\bigcup_{j\ge0}f^{-j}(A_n)\big)$ is residual.

\item It follows straightforwardly from the former construction: $\Lambda$ being a Cantor set implies that $\UU\setminus\Omega(f)$ has connected components, that Lemma~\ref{homtervals} says it is a wandering interval. The converse, for as $\Lambda$ is compact and perfect, if we suppose $\interior(\Lambda)\ne\emptyset$, following the same reasoning of (\ref{item4}), there would be an interval $I$ such that $f^j(I)\ni c$ for some $j$, contradicting the existence of wandering interval.

\end{enumerate}

\cqd

\begin{Corollary} Let $f:[0,1]\setminus\{c\}\to[0,1]$ be a $C^{2}$ non-flat contracting Lorenz map without any weak repeller with $\log |f'|$ bounded then if $f$ is not $\infty$-renormalizable and has no attracting periodic orbit nor is a Cherry-like map, then $\exists$ intervals $I_1, ..., I_s$ with $f(I_1 \cup ... \cup I_s) \subset I_1 \cup ... \cup I_s$ such that $\omega_{f}(x)=I_1\cup...\cup I_s$ for a residual set of $x\in[0,1]$.

\end{Corollary} 

\dem
It was proved by Barry and Mestel \cite{BM} that under these conditions $f$ has no wandering interval. It follows from the theorem above that it is not a Cantor set, but a finite orbit of intervals.
\cqd

\begin{Lemma} 
\label{perdenso} Let $f:[0,1]\setminus\{c\}\to[0,1]$ be a $C^{2}$ non-flat contracting Lorenz map without weak repellers. 
Suppose that $f \notin \cl_{Per}\cup\cl_{Sol} \cup \cl_{Che}$. Then $\overline {Per(f)\cap \Lambda}=\Lambda$, with $\Lambda$ as obtained in Theorem \ref{cicloint}.

\end{Lemma}
\dem
Suppose that $\Lambda\setminus\overline {Per(f)}\ne\emptyset$. Let $I$ be connected component of $\UU\setminus\overline {Per(f)}$ such that $I\cap\Lambda\ne\emptyset$. As $\Lambda$ is perfect and compact we have that $I\cap\Lambda$ is uncountable. Moreover, as $\{x\in\Lambda;\omega_f(x)=\Lambda\}$ is residual in $\Lambda$, we have that $\{x\in\Lambda;\omega_f(x)=\Lambda\}\cap I$ is uncountable.
Then, the set of points that return infinitely many times to $I$ (that is, $\bigcap_{j\ge0}f^{-j}(I))$ is uncountable.
Let $I^*=\{x\in I;\co^+_f(f(x))\cap I\ne\emptyset\}$ be the set of points that return to $I$ and $F:I^*\to I$ the first return map. 
Observe that the set of points that return infinitely many times to $I$ is given by
$$
\{x; \#(\co^+_f(x)\cap I)=\infty\}=\bigcap_{j\ge0}F^{-j}(I))
$$
This way
\begin{equation}
\label{estdavi}
\bigcap_{j\ge0}F^{-j}(I) \text{ is uncountable.} 
\end{equation}

\begin{Claim}\label{afir1}
If $J$ is connected component of $I^*$ then $F(J)=I$.
\end{Claim}
{\em Proof of the claim.}
Let $I=(i_0,i_1)$. If $F(J)\ne I$ then let $(t_0,t_1)=F(J)$ and in this case $t_0\ne i_0$ or $t_1\ne i_1$. Suppose $t_0\ne i_0$ (the other case is analogous). Let $n=R(J)$.
As $t_0\ne i_0$, there is $0\le s< n$ such that $f^s(t_0)=c$. Then we have that
$$
\#\big(Per(f)\cap f^s(J)\big)=\#\big(Per(f)\cap (c,f^s(t_1))\big)=\infty,
$$
as the periodic points accumulate in both sides of the critical point (Proposition~\ref{Proposition008345678}).
Then $\#\big(Per(f)\cap I\big)\ge\#\big(Per(f)\cap f^n(J)\big)=\infty$, contradicting the fact that $I$ is connected component of $\UU\setminus\overline {Per(f)}$.
$\square $(end of the proof of the claim)  

\begin{Claim}\label{afir2}
$I^*$ has more than one connected component.
\end{Claim}
{\em Proof of the claim.}
Suppose it isn't so, then $I^*$ is an interval and we will write it as $(u,v)$ and $F=f^n|{(u,v)}$ for some $n\in\NN$. This implies then that $\bigcap_{j\ge0}F^{-j}(I)=Fix(f^n|{(u,v)})$. 
But this is an absurd, as by equation (\ref{estdavi}) this set would be uncountable and so the set of periodic points of $f$ would also be uncountable.

$\square $(end of the proof of the claim)

%%%%%%%%%%%%%%%%%%%%%%%%%%%%%%%%%%%%%%%%%%%%
\begin{figure}
\begin{center}\label{PerDenso.png}
\includegraphics[scale=.27]{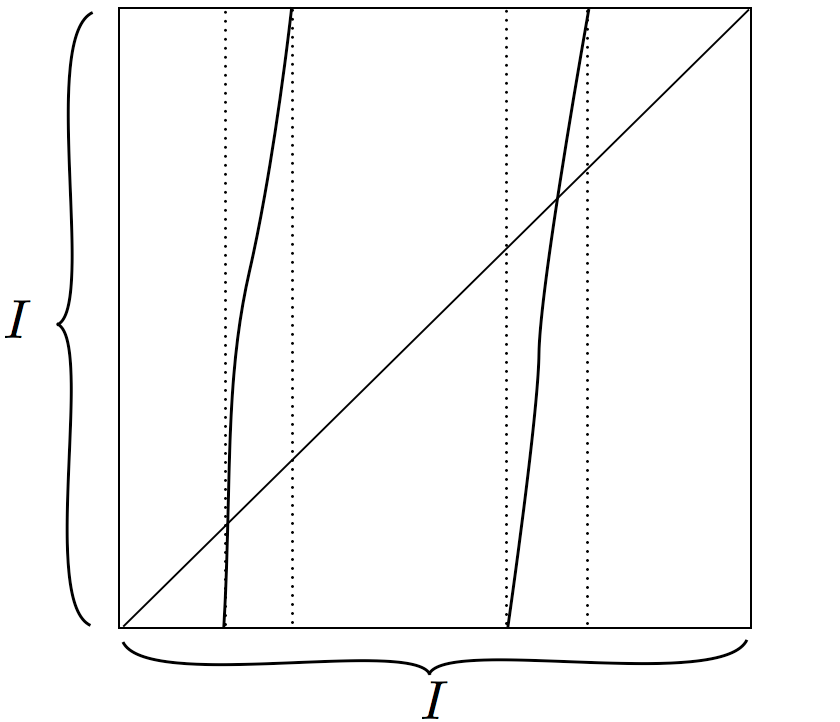}\\
Figure~\ref{PerDenso.png}
\end{center}
\end{figure}
%%%%%%%%%%%%%%%%%%%%%%%%%%%%%%%%%%%%%%%%%%%%%%

As $F$ has at least two branches covering the full image $I$, we have it has infinitely many periodic points and then $f$ also has infinitely many periodic points in $I$, absurd.

\cqd

\subsection{Proof of Theorems \ref{baciastopologicas}, \ref{teoalfalim} and \ref{atratortopologico}}\label{ProofABCD}

Now, we will prove the main theorems: Theorem \ref{baciastopologicas}, \ref{teoalfalim} and \ref{atratortopologico}. %\ref{atrattopologicoXmetrico}.

\dem[Proof of Theorem~\ref{baciastopologicas}]
We are supposing $f$ has no attracting periodic orbit. The first situation to consider is 
\begin{enumerate}

\item $\exists \varepsilon>0$ such that $B_\varepsilon(c)\cap Per(f)=\emptyset$. 
As there are no attracting periodic orbits, Lemma~\ref{noitenoite} can be rewritten to say that $\omega_{f}(x)\ni c \,\forall x \in B_\varepsilon(c)$. That is, according to Lemma~\ref{rotirra} $f$ there exists $\rho\notin\QQ$ such that $rot_{f}(x)=\rho$ $\forall\,x\,B_\varepsilon(c)$. From Lemma~\ref{SecAtCherry-like} in the Appendix there is a compact minimal set $\Lambda$ s.t. $\omega_{f}(x)=\Lambda$ $\forall\,x\in B_\varepsilon(c)$. As $\{x\,;\,\co_{f}^{+}(x)\cap B_\varepsilon(c)=\ne\emptyset\}$ is an open and dense set (with Lebesgue measure one), it is not difficult to conclude that $\Lambda$ is a Cherry attractor (as defined in Section~\ref{MainResults}).

Observe that $\Lambda = \omega(c)=\overline{\co_f^+(c)}$, and its basin contains $B_\varepsilon(c)$.

As $\bigcup_{j\in\NN}^{-j}B_\varepsilon(c)$ is open and also follows by Ma\~ne that it is dense, $\Lambda$ attracts a residual subset of the interval. 

It may occur that the semi-conjugacy is not a bijection, meaning the Cherry map has a gap, that is, there is a wandering interval. 

For the remaining cases we have then that $\forall \varepsilon>0 \,\exists  p; p\in B_\varepsilon(c)\cap Per(f)$.

\item Among these, the first situation to consider is the one of $\Lambda$ being a solenoidal attractor.

As we have defined, it is a set $\Lambda \subset \bigcap_{n=0}^\infty K_n$, where $K_n=\bigcup_{j=0}^{\period(p_{n})}f^{j}((p_n,c))\cup \bigcup_{j=0}^{\period(q_n)}f^{j}((c,q_n))$, $J_n=(p_{n},q_{n})$, $n\in\NN$, and $J_1 \supset J_2 \supset \cdots $ is the chain of renormalization intervals.

	It follows from the construction that $\Lambda \ni c$, indeed $\co^+_f(c)=\omega(c)=\Lambda$. Moreover, it follows from Proposition~\ref{Prop765091} that all points of the open dense set  $V_n=[0,1]\setminus\Lambda_{J_{n}}$ visit $J_n$, then there is a residual set $\bigcap_{n=0^\infty}V_{n}$ of points that eventually fall in any renormalization interval, that is, belong to the basin of $\Lambda$, as stated. By Lemma \ref{Lemma88609}.

\item Now we come to the situation that $f$ has no periodic attractor, neither Cherry attractor nor Solenoidal attractor. It follows from Theorem~\ref{cicloint} that $\exists \Lambda$ compact, $f(\Lambda)=\Lambda$, transitive set   such that $\omega_{f}(x)=\Lambda$ for a residual set of points of $\Lambda$, whose basin of attraction $\beta(\Lambda):=\{x;\omega_{f}(x)\subset \Lambda\}$, is an open and dense set of full measure. Also, $\exists \lambda >0$ such that $\lim_{n\to\infty} \frac{1}{n}\log |Df^n(x)|=\lambda$ for a dense set of points $x$ in $\Lambda$.

Theorem~\ref{cicloint} also gives two possibilities for this setting:
\begin{enumerate}

\item either $\Lambda$ is a finite union of intervals and $\omega_{f}(x)=\Lambda$ for a residual set of $x$ in $[0,1]$
\item or it is a Cantor set and there is a wandering interval.
\end{enumerate}

In both cases all we have to prove to complete the theorem is that these two are chaotic attractors, and for this, it only remains to prove that periodic orbits are dense in it ($\overline {Per(f)\cap \Lambda}=\Lambda$) and that its topological entropy $h_{top}(f|_\Lambda)$ is positive. The condition on the periodic points follows from Lemma~\ref{perdenso}. 
The fact that the topological entropy is positive can be obtained by taking arbitrarily small nice intervals whose borders are non-periodic (e.g., pre-periodic points), and observing that the returns to this interval provide at least two full branches, that will give in the dynamics shifts that have positive entropy.

\end{enumerate}

\cqd

\dem[Proof of Theorem~\ref{teoalfalim}]
The existence of a single topological attractor is given by Theorem~\ref{cicloint}. If $\Lambda$ is a Cherry attractor and does not have a wandering interval, then there is an interval $[a,b]$, such that (identifying $a$ and $b$) the first return map to $F:[a,b]\to[a,b]$ is conjugated to an irrational rotation. In particular $\alpha_{f}(x)\supset\alpha_{F}(x)=[a,b]=\omega_{F}(x)\subset\omega_{f}(x)$ $\forall\,x\in[a,b]$. Furthermore, $\Lambda=``\co_{f}^{+}([a,b])''=[a,b]\cup \bigcup_{j=0}^{\ell-1}f^{j}([f(a),f(c_{-})])\cup \bigcup_{j=0}^{r-1}f^{j}([f(c_{+}),f(b)])$, where $\ell$ and $r$ is given by $f^{\ell}((a,c))\subset(a,b)\supset f^{r}((c,b))$. So, $$\alpha_{f}(x)\supset\Lambda\subset\omega_{f}(x)\,\,\forall\,x\in\Lambda.$$ In particular, $$\alpha_{f}(x)\supset(a,b)\subset\omega_{f}(x)\,\,\forall\,x\in\Lambda.$$

Thus, as $[0,1]\setminus\Lambda_{c}$ is open and dense, where $\Lambda_{c}=\{x\in[0,1]\,;\,\omega_{f}(x)\ni c\}$ is the  $c$-phobic  set, we get $\alpha_{f}(x)=[0,1]$ $\forall\,x\in\Lambda$.

If $\Lambda$ is a Solenoid, then $\Lambda\subset\bigcap_{n=0}^\infty K_n$, where $$K_n=\bigg(\bigcup_{j=0}f^{\period(a_{n})}([a_{n},c))\bigg)\cap\bigg(\bigcup_{j=0}^{\period(b_{n})}f^{j}((c,b_{n}])\bigg)$$ and $\{J_n=(a_{n},b_{n})\}_{n}$ is an infinite nested chain of renormalization intervals. In this case, $\alpha_{f}(x)\supset \Lambda_{J_{n}}$ $\forall\,x\in K_{J_{n}}$ $\forall\,n$. If $f$ does not have wandering interval, it is easy to show that $\bigcup_{n\ge0}\Lambda_{J_{n}}$ is dense in $[0,1]$. Thus $\alpha_{f}(x)=[0,1]$ $\forall\,x\in\Lambda$, because $\bigcap_{n=0}^\infty K_n$.

Finally, if $\Lambda$ is not a Cherry or a Solenoid attractor, the proof follows from Corollary~\ref{sobrealfa} and items (\ref{opcoes}) and  (\ref{cantor}) of Theorem~\ref{baciastopologicas}. Indeed, as we are assuming that $f$ does not have wandering intervals, it follows from items (\ref{opcoes}) and (\ref{cantor}) of Theorem~\ref{baciastopologicas} that $\Lambda$ is a cycle of intervals. By Corollary~\ref{sobrealfa} and the fact that $\Lambda=\overline{\UU\cap\Omega(f)}$, we get $\alpha_{f}(x)\supset\Lambda$. As $\Lambda$ contains a open neighborhood of $c$, it follows that the set of $x\in[0,1]$ such that $\co_{f}^{+}(x)\cap\Lambda\ne\emptyset$ has Lebesgue measure one (and so, it is dense). Thus $\bigcup_{j\ge0}f^{-j}(\Lambda)$ is dense and so $\alpha_{f}(x)$ is dense $\forall\,x\in\UU$. As the $\alpha$-limit is closed, $\alpha_{f}(x)=[0,1]$ for all $x\in\UU$. If $\Lambda=\overline{\UU\cap\Omega(f)}=\UU\cap\Omega(f)\subset\UU$, the proof is done. On the other hand, if $\Lambda\not\subset\UU$ then $\Lambda\setminus\UU\subset(\co_{f}^{+}(c_{-})\cap\co_{f}^{+}(c_{+}))$. But, as it was defined in the beginning of Section~\ref{MainResults}, $c\in f^{-1}(f(c_{-}))$ and also $c\in f^{-1}(f(c_{+}))$. Thus $\alpha_{f}(\Lambda\setminus\UU)\supset\alpha_{f}(c)=[0,1]$ (because $c\in\UU$).
\cqd

\begin{Lemma}[Denseness of wandering interval, when they exist]\label{DenWanInt}
Let $f:[0,1]\setminus\{c\}\to[0,1]$ be a $C^{2}$ be a non-flat contracting Lorenz map without periodic attractors or weak repellers.  If $f$ has a wandering interval $I$ then then $\cw$ is an open and dense set, where $\cw$ is union of all open wandering intervals of $f$.
\end{Lemma}
\begin{proof}
If $\cw$ is not dense then $[0,1]\setminus\cw$ contains some open interval $I$. Of course $I$ is not a wandering interval. As $f$ does not have periodic attractors or weak repellers, we can apply Lemma~\ref{homtervals} and conclude that there is $n\in\NN$ that $f^{n}|_{I}$ is a homeomorphism and that $f^{n}(I)\ni c$. As $\cw$ is invariant ($f^{-1}(\cw)=\cw$), $[0,1]\setminus\cw$ is also invariant. Thus $c\in\interior([0,1]\setminus\cw)$, that is, there is no wandering interval in a neighborhood of $c$. An absurd, by Lemma~\ref{LemmaStP}.

\end{proof}

\dem[Proof of Theorem~\ref{atratortopologico}]
Items (1),(2) and (3)(a) repeat what is said in Theorem~\ref{baciastopologicas}. 
In the other cases, we have the existence of wandering intervals, so let's consider $V$ the union of all wandering intervals. This set is open and dense in $[0,1]$. It follows from St. Pierre's Theorem (first Theorem of page 10 in~\cite{StP}) that $\exists !$ attractor $\Lambda_{sP}$ such that $\omega_{f}(x)=\Lambda_{sP}$ for Lebesgue almost $x \in [0,1]$. Then, if $I$ is a wandering interval, we have that $\omega_{f}(x)=\Lambda_{sP}$ for Lebesgue almost $x \in I$. On the other hand, as $I$ is a wandering interval, $\omega_{f}(x)$ is the same for any $x\in I$. This way we conclude that $\omega_{f}(x)=\Lambda_{sP}, \forall x \in V$. Take $\Lambda=\Lambda_{sP}$.

In this case that we have wandering intervals two situations can occur: $\Lambda=\overline{\UU\cap\Omega(f)}$ where $\UU$ is defined in (\ref{regiaoarmadilha}) or $\Lambda \subsetneqq \overline{\UU\cap\Omega(f)}$.
We have seen in the proof of Theorem A that $\overline{\UU\cap\Omega(f)}$ is chaotic and then the attractor is chaotic in the case $\Lambda=\overline{\UU\cap\Omega(f)}$. On the other hand, if $\Lambda\subsetneqq\overline{\UU\cap\Omega(f)}$ then, as $\overline{\UU\cap\Omega(f)}$ is transitive and $\Lambda$ also is by St. Pierre's first Theorem of page 10 in~\cite{StP} again, as it says that for almost every $x$,  $\omega_{f}(x)=\Lambda$, it follows from the definition of wild attractor that $\Lambda$ is wild.

\cqd

We may emphasize the possibilities of occurrence of wild attractors and how this relates to the existence or not of wandering intervals. We saw that whenever we have the existence of wandering intervals, it necessarily happens that the topological and metrical attractors are the same $\Lambda$ and $\Lambda\subsetneqq\overline{\UU\cap\Omega(f)}$, so that $\Lambda$ is wild. 

On the other hand, if we don't have wandering intervals, then $\Lambda = \overline{\UU\cap\Omega(f)}$ one may verify that there can't be a bigger transitive set containing both of them.

\newpage

\section{Particular features of each kind of Attractor}

We give in Table~\ref{TABELA.pdf} a panoramic view of the topological attractors for contracting Lorenz maps. There, the main features of each kind of possible attractor are described, according to their structure of periodic points, topological entropy (see Lemma~\ref{LemmaZeroEntropy} in the Appendix), Lyapunov exponents, transitivity, $\alpha$ and $\omega$-limit sets of the points of the attractor, etc.

To include the periodic attractor and simplify technicalities like inessential periodic attractors, we assumed that the map is $C^{3}$ with negative Schwarzian derivative.
For more informations about essential and inessential periodic attractors , see Section~4 in Chapter II of \cite{MvS}

In Table~\ref{TABELA.pdf}, $\Lambda$ are the possible ``topological global attractors'' for a $C^{3}$ contracting Lorenz map $f:[0,1]\setminus\{c\}\to[0,1]$ with negative Schwarzian derivative. This means that $\Lambda$ attracts a residual set of points of $[0,1]$. If $f$ does not have periodic attractors, the existence and classification of $\Lambda$ is given by Theorem~\ref{atratortopologico}. Otherwise, we know by Theorem~\ref{stpierreperiodic} that either $\Lambda$ is a single periodic attractor or $\Lambda=\Lambda_{1}\cup\Lambda_{2}$ is a union of two periodic attractors $\Lambda_{1}$ and $\Lambda_{2}$. Recall that a periodic attractor can be of two types: an attracting periodic orbit of or a super-attractor (see Definition~\ref{PeriodicAttractor} and the comments just below it).

%%%%%%%%%%%%%%%%%%%%%%%%%%%%%%%%%%%%%%%%%%%%
\begin{figure}
\begin{center}\label{TABELA.pdf}
\includegraphics[scale=.59]{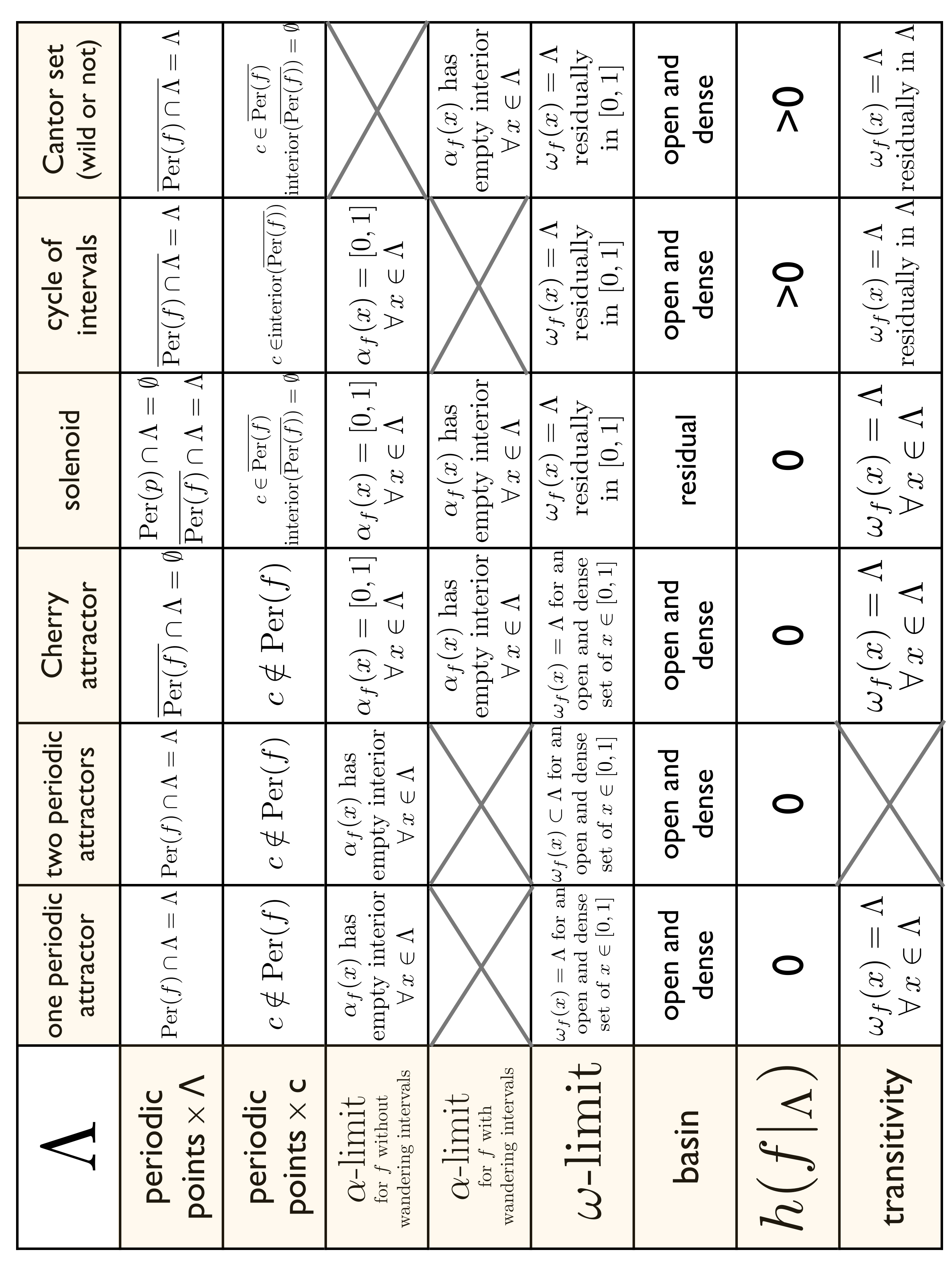}\\
Figure~\ref{TABELA.pdf}
\end{center}
\end{figure}
%%%%%%%%%%%%%%%%%%%%%%%%%%%%%%%%%%%%%%%%%%%%%%

\newpage

\section{Spectral Decomposition of the Non-Wandering Set}

Recall that if $J=(a,b)$ is a renormalization interval, the  renormalization cycle associated to it is the of a nice interval $J=(j_0,j_1)$ was defined (Definition~\ref{rencyc}) as being the union of iterates of $J$, that is, $$U_J=\bigg(\bigcup_{i=0}^{\period(a)}f^{i}((a,c))\bigg)\cup\bigg(\bigcup_{i=0}^{\period(b)}f^{i}((c,b))\bigg).$$

The nice trapping region of $J$ (Definition~\ref{nicetrap}) is the union of gaps of $\Lambda_J$ such that each gap contains one interval of the renormalization cycle, where $\Lambda_{J}=\{x\,\;\,\co_{J}^{+}(x)\cap J=\emptyset\}$ is the $J$-phobic set: the set whose points avoid $J$ (see Definition~\ref{Def765091} and Proposition~\ref{Prop765091}). 

Observe that a nice trapping region $K_J$ is a positively invariant set as iterating any point of it, it travels through gaps of $\Lambda_J$ that eventually fall into the nice interval $J$ that generated $K_J$, as we constructed in Section \ref{phobic}. Indeed, $f^{j}((a,c))\cap\Lambda_{J}=\emptyset=f^{i}((c,b))\cap\Lambda_{J}$ for $0\le j<\period(a)$ and $0\le i<\period(b)$, as $J$ is a renormalization interval. Thus, if we denote, for a given $x\notin\Lambda_{J}$, the gap of $\Lambda_{J}$ containing $x$ by $G[x]$ then $$f^{\period(a)-1}\big(G[f(a,c)]\big)=G[(a,c)]=J=G[(c,b)]=f^{\period(b)-1}\big(G[f((c,b))]\big)$$ and $$K_J=\bigg(\bigcup_{i=0}^{\period(a)}G[f^{i}((a,c))]\bigg)\cup\bigg(\bigcup_{i=0}^{\period(b)}G[f^{i}((c,b))]\bigg)=$$
$$=\bigg(\bigcup_{i=0}^{\period(a)-1}f^{j}\big((G[f((a,c))]\big)\bigg)\cup\bigg(\bigcup_{i=0}^{\period(b)-1}f^{i}\big(G[f((c,b))]\big)\bigg)$$

Now we will state the main result of this section: the Spectral Decomposition for Lorenz Maps. Below we present a more detailed version of Theorem~\ref{spdec}.

\begin{theorem}[Spectral Decomposition for Lorenz Maps]\label{spdecnotacao}
Let $f:[0,1]\setminus\{c\}\to[0,1]$ be a non-flat $C^{3}$ contracting Lorenz map with negative Schwarzian derivative. Then, there is $n_f \in \NN \cup \{\infty\}$, the number of strata of the decomposition of $f$, and a sequence $\{K_{n}\}_{n}$ of trapping regions such that:
\begin{enumerate}
\item $K_{0}=(0,1)\supsetneqq K_{1}\supsetneqq K_{2}\supsetneqq\cdots$;
\item $K_{n}$ is a nice trapping region for some nice interval $I_{n}$ $\forall\,0<n<n_{f}$;
\item $\Omega(f)$ can be decomposed into closed forward invariant subsets $\Omega_j$ such that
$$
\Omega(f)=\bigcup_{0\le j \le n_f}\Omega_j
$$
where, for $j<n_f$,
$$
\Omega_j:= \Omega(f)\cap (K_j\setminus K_{j+1})
$$

and $$\Omega_{n_f}:= \begin{cases} 

\overline{ \Omega(f)\cap (K_{n_f})} & \text{ if } n_f < \infty\\
\bigcap_{j\ge 0}K_j & \text{ if } n_f = \infty

\end{cases}$$
\item
\begin{itemize}
\item[(i)]
$\Omega_0  = \begin{cases}

[0,1] & \text{ if $f$ is transitive (in this case, $n_f=0$, $f(c_{+})=0$ and}\\
 & \text{\hspace{0.5cm}$f(c_{-})=1$)}\\
\{0\} & \text{ if $f(c_{+})>0$ and $f(c_{-})=1$ (and in this case, $n_f=1$)}\\
\{1\} & \text{ if $f(c_{+})=0$ and $f(c_{-})<1$ (and in this case, $n_f=1$)}\\
\{0,1\} & \text{ otherwise, i.e., if $f(c_{+})>0$ and $f(c_{-})<1$}
\end{cases}$
\item[(ii)] $\Omega_j$ is a transitive set for all $0 <j < n_f$.

\item[(iii)] $\Omega_{n_f}$ is either a transitive set or the union of a pair of attracting periodic orbits.

\end{itemize}
\item
For each $0<j<n_{f}$ there is a decomposition of $\Omega_j$ into closed sets $$\Omega_j= X_{0,j}\cup X_{1,j}\cup\cdots\cup X_{\ell_j,j}$$ such that $\# ( X_{a,j}\cap X_{b,j} )\le1$ when $a\ne b$, the first return map to $X_{0,j}$ is topologically exact (in particular, topologically mixing) and,
for each  $i\in\{1,\cdots,\ell_{j}\}$, there is  $1\le s_{i,j}\le \ell_{j}$ such that $f^{s_{i,j}}(X_{i,j})\subset X_{0,j}$.

\end{enumerate}

\end{theorem}

We define a {\em stratum on level $0 \le j \le n_f$} of the {\em filtration} $\big(K_i\big)_{i\le {n_f}}$ as being $W_0:=\{0,1\}$ and, for $j>0$, the set $W_j=K_{j-1}\setminus  {K_{j}}$.

We have previously studied the structure of the attractors that can occur for contracting Lorenz maps, what gives us a fundamental information on the structure of, at least, the last stratum as defined above, $\Omega_{n_f}$. But it gives no information on the non-wandering sets of the other levels of the stratification (although the former study gives us the structure of the non-wandering set in the last renormalization level in the cases that we don't have a periodic attractor, but it doesn't say anything about the non-wandering set if there is a periodic attractor, and we will have to analyze this situation separately, later).

This construction is analogous to the one developed by Smale to hyperbolic maps, with two important differences. The first one is the number of strata, that in that case was in a finite number and in this context can be infinity, although countable. The second one is the fact that the non-wandering sets of the last stratum may not be uniformly hyperbolic (although in many cases it can have a non-uniformly expanding behavior). Nevertheless, it is important to emphasize that in all the other strata the non-wandering set is transitive and expanding, as in the Axiom A setting.

\subsection{Adapted results to the negative Schwarzian derivative context}

\begin{Definition}Let $f:[0,1]\setminus\{c\}\to[0,1]$ be a contracting Lorenz map. We say that a renormalization interval $I=(a,b)$ is regular if $f^{\ell}(x)>x$ $\forall x \in{(a,c)}$,  $f^{r}(x)<x$ $\forall x \in{(c,b)}$ and $\lim_{x \uparrow c}f^{\ell}(x)>c > \lim_{x\downarrow c}f^{r}(x)$, where $\ell=\period(a)$ and $r=\period(b)$ (see Figure~\ref{RenorRegXIrreg1.png}).
\end{Definition}

%%%%%%%%%%%%%%%%%%%%%%%%%%%%%%%%%%%%%%%%%%%%
\begin{figure}[H]
\begin{center}\label{RenorRegXIrreg1.png}
\includegraphics[scale=.25]{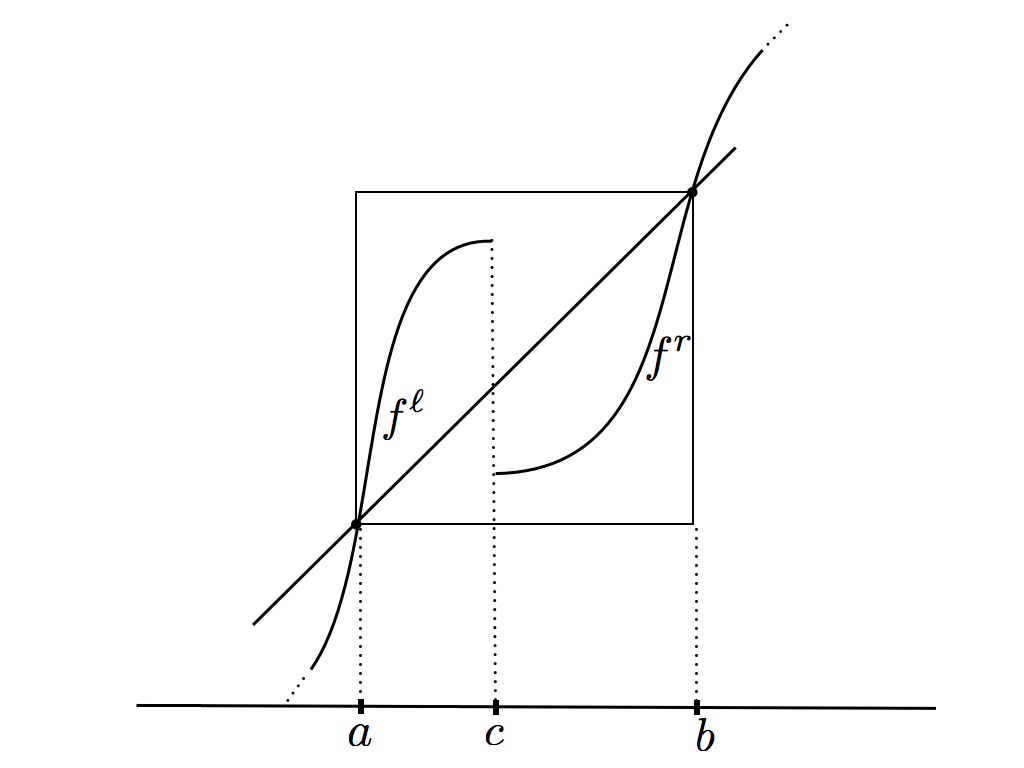}\\
Figure~\ref{RenorRegXIrreg1.png}: Regular renormalization interval
\end{center}
\end{figure}
%%%%%%%%%%%%%%%%%%%%%%%%%%%%%%%%%%%%%%%%%%%%%%

\begin{Lemma}\label{Lemma657h6}
Let $f:[0,1]\setminus\{c\}\to[0,1]$ be a contracting Lorenz map and $I=(a,b)$ a regular renormalization interval.  Given any $\varepsilon>0$ there are $n,m\ge1$ such that $f^{n\period(a)}|_{(a,a+\varepsilon)}$ and $f^{m\period(b)}|_{(b-\varepsilon,b)}$ are homeomorphisms and $$f^{n\period(a)}((a,a+\varepsilon))\supset(a,c)\text{ and }f^{m\period(b)}((b-\varepsilon,b))\supset(c,b).$$
\end{Lemma}
\dem
As $(a,b)$ is a renormalization interval, $f^{\period(a)}|_{[a,c)}$ is a homeomorphism. As $I$ is regular , $f^{n\period(a)}([a,c))\supset[a,c)$. Thus, choosing $n=\min\{j\ge1$ $;$ $(f^{\period(a)}|_{[a,c)})^{-j}(c)\in(a,a+\varepsilon)\}$ we get that $f^{n\period(a)}|_{(a,a+\varepsilon)}$ is a homeomorphism and $f^{n\period(a)}((a,a+\varepsilon))\supset(a,c)$. The proof for the other side ($b$ side) is the same.
\cqd

\begin{Remark}\label{Remark54556}
Let $f:[0,1]\setminus\{c\}\to[0,1]$ be a $C^{3}$ contracting Lorenz map with negative Schwarzian derivative.  A renormalization interval $J=(a,b)$ is regular if and only if $$f^{\period(a)}((a,c))\ni c\in f^{\period(b)}((c,b)).$$
\end{Remark}

The remark below follows straightforwardly from the definition of a regular renormalization interval and the Minimum Principle (Proposition~\ref{MinP}).

\begin{Remark}\label{Remark5466}
Let $f:[0,1]\setminus\{c\}\to[0,1]$ be a $C^{3}$ contracting Lorenz map with negative Schwarzian derivative. If  
a renormalization interval $J$ is not regular then $f$ has a periodic attractor $\Lambda$ such that $(p,c)$ or $(c,p)$ is contained in the basin of $\Lambda$ for some $p\in\partial J$ (see Figure~\ref{RenorRegXIrreg2.png}).
\end{Remark}

%%%%%%%%%%%%%%%%%%%%%%%%%%%%%%%%%%%%%%%%%%%%
\begin{figure}
\begin{center}\label{RenorRegXIrreg2.png}
\includegraphics[scale=.43]{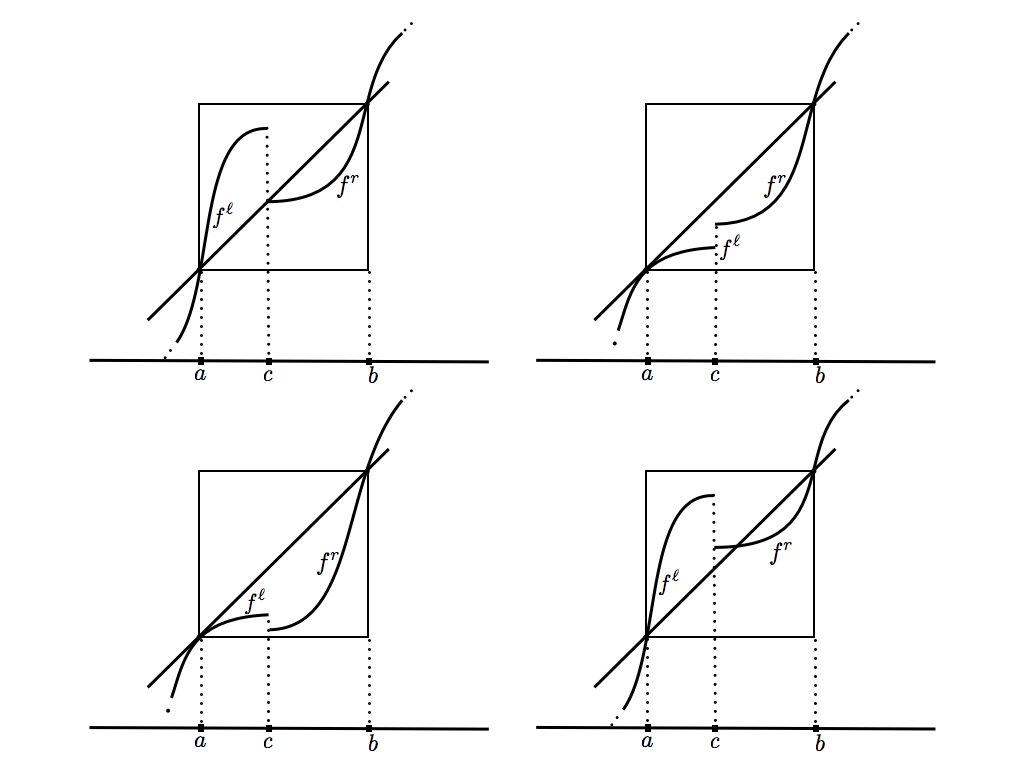}\\
Figure~\ref{RenorRegXIrreg2.png}(a): Examples of non-regular renormalization intervals
\end{center}
\end{figure}
%%%%%%%%%%%%%%%%%%%%%%%%%%%%%%%%%%%%%%%%%%%%%%

%%%%%%%%%%%%%%%%%%%%%%%%%%%%%%%%%%%%%%%%%%%%
\begin{figure}
\begin{center}\label{RenorRegXIrreg3.png}
\includegraphics[scale=.43]{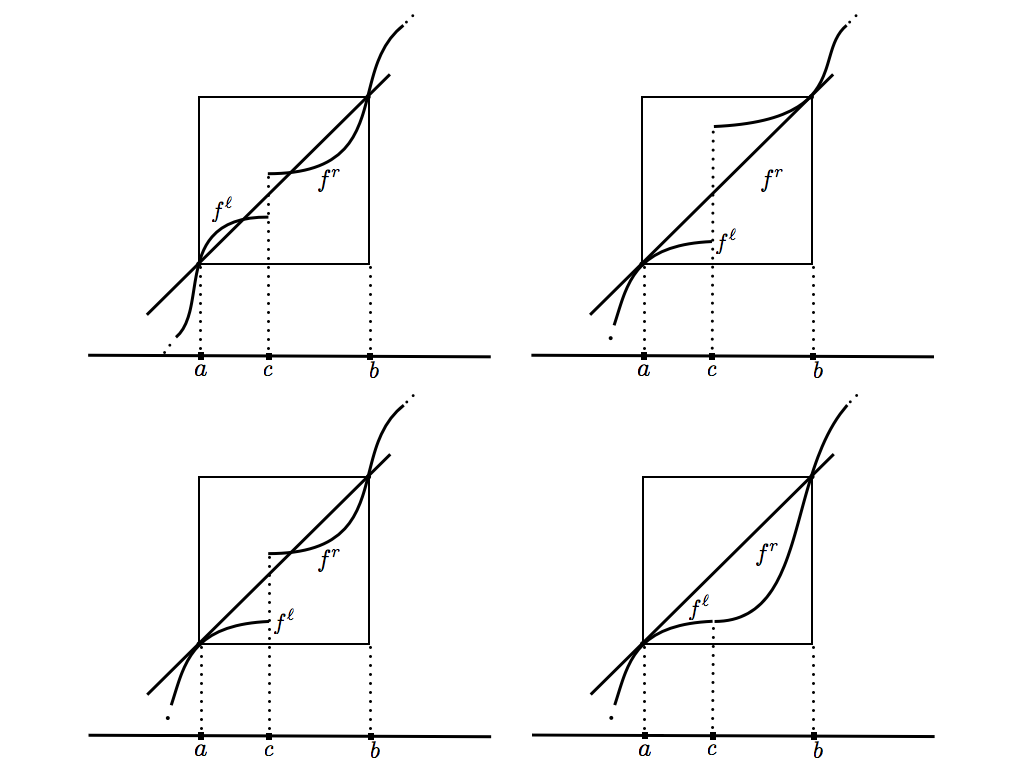}\\
Figure~\ref{RenorRegXIrreg3.png}(b): Examples of non-regular renormalization intervals
\end{center}
\end{figure}
%%%%%%%%%%%%%%%%%%%%%%%%%%%%%%%%%%%%%%%%%%%%%%

%SE DER TEMPO, ESCREVER DEPOIS E UM DESENHO TB

\begin{Lemma}\label{LemmaJ1J2}
Let $f:[0,1]\setminus\{c\}\to[0,1]$ be a $C^{3}$ contracting Lorenz map with negative Schwarzian derivative. Let $J_{0}$ and $J_{1}$ be renormalization intervals of $f$. If $J_{0}$ is regular but $J_{1}$ is not then $J_{0}\not\subset J_{1}$.
\end{Lemma}
\dem
Suppose that $J_{1}$ is not regular with $J_{0}\subset J_{1}$ and let $\Lambda$ be the periodic attractor such that $(p,c)$ or $(c,p)$ is contained in the basin of $\Lambda$ for some $p\in\partial J_1$. Write $J_{0}=(a,b)$. As $J_{0}$ is regular, $p\notin\partial J_{0}$. Thus $p<a<c$ or $c<b<p$. In any case we get an absurd as $a,b$ are periodic points.
\cqd

\begin{Lemma}[Regular renormalization intervals are not linked]\label{renormalinksSCH} Let $f:[0,1]\setminus\{c\}\to[0,1]$ be a $C^{3}$ contracting Lorenz map with negative Schwarzian derivative. Let $J_{0}$ and $J_{1}$ be renormalization intervals of $f$. If $J_{0}$ is regular then it can never be linked to other renormalization interval (regular or not). Moreover, if $J_{0}$ and $J_{1}$ renormalization intervals, with $J_{0}$ regular and $J_{0}\ne J_{1}$ then  either $\overline{J_{0}}\subset J_{1}$ or $\overline{J_{1}}\subset J_{0}$. In particular, $\partial J_{0}\cap\partial J_{1}=\emptyset$.
\end{Lemma}
\dem

The proof follows closely the proof of Lemma~\ref{renormalinks}.

Write $J_0=(a_0,b_0)$ and $J_1=(a_1,b_1)$. First note that $J_{0}$ and $J_{1}$ can not be linked. Indeed, if they were linked, we would either have $a_{0}<a_{1}<c<b_{0}<b_{1}$ or $a_{1}<a_{0}<c<b_{1}<b_{0}$. Suppose that $a_{0}<a_{1}<c<b_{0}<b_{1}$. In this case, $a_{1}\in J_{0}$ and by Lemma~\ref{Lemma545g55}.
$\emptyset\ne\co_{f}^{+}(a_{1})\cap(c,b_{0})\subset\co_{f}^{+}(a_{1})\cap(a_{1},b_{1})=\co_{f}^{+}(a_{1})\cap J_{1}$ contradicting the fact that $J_{1}$ is a nice interval. Now, if $a_{1}<a_{0}<c<b_{1}<b_{0}$, we have $b_{1}\in J_{0}$. Again, by Lemma~\ref{Lemma545g55}, we get $\emptyset\ne\co_{f}^{+}(b_{1})\cap(a_{0},c)\subset\co_{f}^{+}(b_{1})\cap J_{1}$ and a contradiction with $J_{1}$ being a nice interval.

As $J_{0}\cap J_{1}\ne\emptyset$ (because both contains the critical point) and as $J_{0}$ and $J_{1}$ are not linked, in follows that either $J_{0}\supset J_{1}$ or $J_{0}\subset J_{1}$.

Suppose that $J_{0}\supset J_{1}$. In this case, as $J_{0}\ne J_{1}$ we have three possibilities: either $a_{0}< a_{1}<c<b_{1}= b_{0}$ or $a_{0}= a_{1}<c<b_{1}< b_{0}$ or $J_{0}\supset\overline{J_{1}}$. If $a_{0}< a_{1}<c<b_{1}= b_{0}$, we can use again Lemma~\ref{Lemma545g55} to get  $\co_{f}^{+}(a_{1})\cap J_{1}\ne\emptyset$. On the other hand, if $a_{0}= a_{1}<c<b_{1}< b_{0}$, the same  Lemma~\ref{Lemma545g55} implies that $\co_{f}^{+}(b_{1})\cap J_{1}\ne\emptyset$. In both cases we get a contradiction to the fact that $J_{1}$ is a nice interval. Thus, the remaining possibility is the true one.

On the other hand, if $J_{1}\supset J_{0}$, it follows from Lemma~\ref{LemmaJ1J2} that $J_{1}$ is regular. Thus, the proof is exactly the same of the case $J_{0}\supset J_{1}$.

\cqd

\begin{Lemma}[The maximal non-regular renormalization interval]\label{LemmaS6546SS} Let $f:[0,1]\setminus\{c\}\to[0,1]$ be a $C^{3}$ contracting Lorenz map with negative Schwarzian derivative. If $f$ has a non-regular renormalization interval then there is a non-regular interval $J_{max}$ such that $J_{max}\supset I$ for every non-regular renormalization interval $I$.
\end{Lemma}

\dem
Suppose that $I=(\alpha,\beta)$ and $J=(a,b)$ are non-regular renormalization intervals. If they are not linked then $I\cup J$ is a non-regular renormalization interval containing $I$ and $J$. Now, suppose that $I$ and $J$ are linked. We may assume that $\alpha<a<c<\beta<b$.

Because $Sf<0$, if $f^{\period(\alpha)}((\alpha,c))\ni c$ then $\co_{f}^{+}(a)\cap(c,\beta)\ne\emptyset$. Thus $\co_{f}^{+}(a)\cap(c,b)\ne\emptyset$. But this is impossible, as $(a,b)$ is a nice interval.
So, we necessarily have $f^{\period(\alpha)}((\alpha,c))\subset(\alpha,c)\subset(\alpha,b)$.
By the same reasoning, we get $f^{\period(b)}((c,b))\subset(c,b)\subset(\alpha,b)$. Thus $I\cup J=(\alpha,b)$ is a renormalization interval. Furthermore, $I\cup J$ is non-regular (Remark~\ref{Remark54556}).

Thus $J_{max}=\bigcup_{J\in\cu}J$ is a non-regular renormalization interval, where $\cu$ is the collection of all non-regular renormalization intervals. Of course that $J_{max}\supset J$ $\forall\,J$ being a non-regular renormalization interval.

\cqd

\begin{Lemma}[$\Omega(f)$ inside a non-regular renormalization interval]\label{LemmaSSSS} Let $f:[0,1]\setminus\{c\}\to[0,1]$ be a $C^{3}$ contracting Lorenz map with negative Schwarzian derivative. If $J$ is a non-regular renormalization interval then there a finite set $\Lambda$ such that $J\subset\beta(\Lambda)$, where $\Lambda$ is a periodic attractor or a union of two periodic attractors. In particular, $\Omega(f)\cap J=\Omega(f)\cap\Lambda$.
\end{Lemma}
\dem
Write $J=(a,b)$, $\ell=\period(a)$ and $r=\period(b)$.
By Remark~\ref{Remark54556}, $f^{\ell}((a,c))\subset(a,c)$ or $f^{r}((c,b))\subset(c,b)$. Suppose the first case (the other one is analogous). As $Sf<0$ there is a some $p\in[a,c]$ such that $\lim_{n\to\infty}(f^{\ell})^{n}(x)=p$ $\forall\,x\in(a,c)$. Let $\Lambda_{p}$ be the periodic attractor of $f$ containing $p$. Note that $(a,c)\subset\beta(\Lambda_{p})$. If $f^{r}((c,b))\ni c$ the $\co_{f}^{+}(x)\cap(a,c)\ne\emptyset$ $\forall\,x\in(c,b)$. Thus, $(a,c)\subset\beta(\Lambda_{p})$. On the other hand, if $f^{r}((c,b))\subset(c,b)$, there is some $q\in[c,b]$ such that $\lim_{n\to\infty}(f^{r})^{n}(x)=q$ $\forall\,x\in(c,b)$. Setting $\Lambda_{q}$ as the periodic attractor of $f$ containing $q$, we get $(c,b)\subset\Lambda_{q}$. So, $(a,b)\subset\Lambda:=\Lambda_{p}\cup\Lambda_{q}$.

\cqd

\begin{Lemma}\label{LemmaofTr}
Let $f:[0,1]\setminus\{c\}\to[0,1]$ be a $C^{3}$ contracting Lorenz map with negative Schwarzian derivative. Let $U\ni c$ be an open trapping region and let $U_{0}=(u_{0},u_{1})$ be the connected component of $U$ containing $c$. Suppose that $U_{0}$ is not a renormalization interval and that there exists $a\in[u_{0},c)$ and $b\in(c,u_{1}]$ such that $\co_{f}^{+}(a)\cap[u_{0},u_{1}]\ne\emptyset\ne[u_{0},u_{1}]\cap\co^{+}_{f}(b)$. If $J_{0}$ is the smallest renormalization interval not contained in $U$ then $\co_{f}^{+}(x)\cap U\ne\emptyset$ $\forall\,x\in J_{0}$.
\end{Lemma}
\dem

Write $\ell=\min\{j>\,;\,f^{j}(a)\in[u_{0},u_{1}]\}$ and $r=\min\{j>\,;\,f^{j}(b)\in[u_{0},u_{1}]\}$.

Let $\cw=\{x\in[0,1]\,;\,\co^{+}_{f}\cap\cu\ne\emptyset\}$. Note that $\cw$ is open and invariant, that is, $f^{-1}(\cw)=\cw$. Let $W=(w_{0},w_{1})$ be the connected component of $\cw$ containing $c$. Thus $W\supset U_{0}$.

Observe that if $p\in\partial J_{0}\subset\cw$ then $\co_{f}^{+}(p)\cap\subset J_{0}$, but this contradicts $J_{1}$ being a nice interval. Thus $\partial J_{0}\cap W=\emptyset$ and so, $J_{0}\supset W\supset U_{0}$.

 As $f^{\ell}(a)\in f^{\ell}([w_{0},c))\cap [w_{0},w_{1}]$  and $f(\cw)\subset\cw$, we get $f^{\ell}((w_{0},c))\subset (w_{0},w_{1})$. Furthermore, as $f^{-1}(\cw)=\cw$ we get $f([0,1]\setminus\cw)\subset [0,1]\setminus\cw$. Thus, $f^{\ell}((w_{0},c))=(w_{0},f^{\ell}(c_{-}))\subset W$. By the same reasoning, $f^{r}((c,w_{0})=(f^{r}(c_{+}),w_{1})\subset W$. That is, $W$ is a renormalization interval. As $U_{0}$ is not a renormalization interval, we get $J_{0}\supset W\supsetneqq U_{0}$.

Observe that if $p\in\partial J_{0}\subset\cw$ then $\co_{f}^{+}(p)\cap\subset J_{0}$, but this contradicts $J_{1}$ being a nice interval. Thus $\partial J_{0}\cap W=\emptyset$ and so, $J_{0}\supset W\supsetneqq J_{1}$. This implies, as $J_{0}$ and $J_{1}$ are consecutive, that $J_{0}=W\subset\cw$. From the minimality of $J_{0}$, we get $J_{0}=W$. Thus, $J_{0}\subset\cw$.

\cqd

\begin{Lemma}\label{Lemma9988}Let $f:[0,1]\setminus\{c\}\to[0,1]$ be a $C^{3}$ contracting Lorenz map with negative Schwarzian derivative. 
Let $J_{0}\supset J_{1}=(a,b)$ be two consecutive renormalization intervals. If $\exists\,\varepsilon>0$ such that $\omega_{f}(x)=\co_{f}^{+}(a)$ $\forall\,x\in(a-\varepsilon,a]$ or $\omega_{f}(x)=\co_{f}^{+}(b)$ $\forall\,x\in[b,b+\varepsilon)$ then  $\co_{f}^{+}(a)=\co_{f}^{+}(b)$, $\co_{f}^{+}(a)$ is an attracting periodic orbit and $J_{0}\subset\beta(\co_{f}^{+}(a))$.
\end{Lemma}

\dem
Suppose that $\omega_{f}(x)=\co_{f}^{+}(a)$ $\forall\,x\in(a-\varepsilon,a]$ (the case $\omega_{f}(x)=\co_{f}^{+}(b)$ $\forall\,x\in[b,b+\varepsilon)$ is analogous). Set $\ell=\period(a)$ and $r=\period(b)$. In this case, there is some $0<\delta<\varepsilon$ such that $\bigcup_{j=0}^{\ell-1}f^{j}((a-\delta,a])$ is a positive invariant set, i.e., $$f\bigg(\bigcup_{j=0}^{\ell-1}f^{j}((a-\delta,a])\bigg)\subset \bigcup_{j=0}^{\ell-1}f^{j}((a-\delta,a]).$$

Thus $$U=\bigcup_{j=0}^{\ell-1}f^{j}((a-\delta,a])\,\cup \,U_{J_{1}}=$$
$$=\bigcup_{j=0}^{\ell-1}f^{j}((a-\delta,a])\cup \bigcup_{j=0}^{\ell-1}f^{j}((a,c)) \cup \bigcup_{j=0}^{r-1}f^{j}((c,b))=$$
$$\bigcup_{j=0}^{\ell-1}f^{j}((a-\delta,c))\cup \bigcup_{j=0}^{r-1}f^{j}((c,b))$$ is an open trapping region.

The connected component of $\cu$ containing $c$ is $\cu_{0}=(a-\varepsilon,b)$. As $\cu_{0}$ is not a nice interval, it is not a renormalization interval. Thus, as $f^{\ell}(a)=a\in [a-\varepsilon,c)$ and $f^{\ell}(a)=a\in (c,b]$, it follows from Lemma~\ref{LemmaofTr} that $\co_{f}^{+}(x)\cap\cu\ne\emptyset$ $\forall\,x\in J_{0}$.

So, if $x\in J_{0}$ then $\co_{f}^{+}(x)\cap(a,b)\ne\emptyset$ or $\co_{f}^{+}(x)\cap(a-\delta,a]\ne\emptyset$. In the last case we have $\omega_{f}(x)=\co_{f}^{+}(a)$. Thus, to finish the proof we need only to show that $\co_{f}^{+}(a)=\co_{f}^{+}(b)$ and that $J_{1}\subset \beta(\co_{f}^{+}(a))$.

As $Sf<0$ and $x<f^{\ell}(x)$ for $a-\delta<x<a$, it is easy to get that $c_{-}$ is in the local basin of $\co_{f}^{+}(a)$. Thus, if $(a,a')$ is the connected component of $[0,1]\setminus\co_{f}^{+}(a)$ containing $c$ then $f^{\ell}|_{(c,a']}((c,a'])\subset\beta(\co_{f}^{+}(a))$. 
So, $b\notin(c,a')$ (because $b$ is periodic). Moreover, as $(a,b)$ is a nice interval, $a'\notin(a,b)$. Thus, $b=a'$. So, $\co_{f}^{+}(a)=\co_{f}^{+}(b)$. 

%%%%%%%%%%%%%%%%%%%%%%%%%%%%%%%%%%%%%%%%%%%%
\begin{figure}
\begin{center}\label{RRper.png}
\includegraphics[scale=.42]{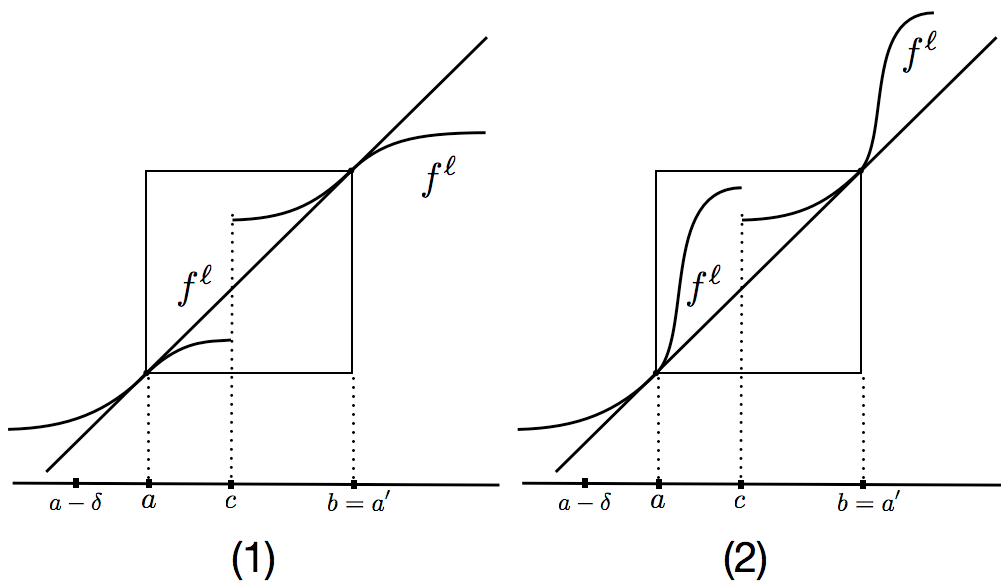}\\
Figure~\ref{RRper.png}
\end{center}
\end{figure}
%%%%%%%%%%%%%%%%%%%%%%%%%%%%%%%%%%%%%%%%%%%%%%

Now we have two possibilities: $a$ is a saddle-noddle (if $x<f^{\ell}(x)$ for $a<x<c$, see Figure~\ref{RRper.png}(2)) or  $a$ attracts by both side (for instance, if $Df^{\ell}(a)<1$), see Figure~\ref{RRper.png}(1). In any case we get $[a,b]\subset\beta(\co_{f}^{+}(a)$ (if $a$ is sable-noddle, for every $a<x<c$ there is some $s_{x}\ge1$ such that $f^{s_{x}\ell}(x)\in[c,a]$ and so, $x\in\beta(\co_{f}^{+}(a))$).

\cqd

\begin{Corollary}\label{Corollary9988}Let $f:[0,1]\setminus\{c\}\to[0,1]$ be a $C^{3}$ contracting Lorenz map with negative Schwarzian derivative. 
Let $J_{0}\supset J_{1}=(a,b)$ be two consecutive renormalization intervals with $J_{0}$ regular. If there is an interval $I\subset J_{0}\setminus J_{1}$ and $\ell\ge1$ such that $f^{\ell}|_{I}$ is a homeomorphism, $f^{j}(I)\cap J_{1}=\emptyset$ $\forall\,0\le j\le\ell$ and $f^{\ell}(I)\subset I$ then $\Lambda=\co_{f}^{+}(a)=\co_{f}^{+}(b)$ is a attracting periodic orbit containing $J_{0}$ in its basin of attraction ($J_{0}\subset\beta(\Lambda)=\{x\,;\,\omega_{f}(x)=\Lambda\}$).
\end{Corollary}
\dem
As  $f^{\ell}(I)\subset I$ and $f^{\ell}|_{I}$ is a homeomorphism, there is a attracting periodic point $p\in\overline{I}$. Suppose for instance that $a\in\co_{f}^{+}(p)$. Say, $a=f^{s}(p)$. In this case, as $(a-\varepsilon,a)=f^{s}(I)$ for some $\varepsilon>0$, we can apply Lemma~\ref{Lemma9988} and finish the proof. The same is true is $b\in\co_{f}^{+}(p)$.

Suppose now that $\partial J_{1}\cap \co_{f}^{+}(p)=\emptyset$. Let $\Lambda=\co_{f}^{+}(p)$ and $(p_{0},p_{1})$ be the connected component of $[0,1]\setminus\Lambda$ containing $c$. By Singer's Theorem $(p_{0},c)$ or $(c,p_{1})\subset\beta(\Lambda)$. As $\co_{f}^{+}(p)\cap\overline{J_{1}}=\emptyset$, we get $p_{0}<a<c<b<p_{1}$. So, $a\in\beta(\Lambda)$ or $b\in\beta(\Lambda)$ and then $\co_{f}^{+}(a)=\omega_{f}(a)=\Lambda$ or $\co_{f}^{+}(a)=\omega_{f}(a)=\Lambda$. An absurd, as we are assuming that $\partial J_{1}\cap\Lambda=\emptyset$.
\cqd

\begin{Corollary}\label{VarPrinII}[The ``adapted Variational Principle''] Let $f:[0,1]\setminus\{c\}\to[0,1]$ be a $C^{3}$ contracting Lorenz map with negative Schwarzian derivative. 
Let $J_{0}\supset J_{1}=(a,b)$ be two consecutive renormalization intervals with $J_{0}$ regular and not contained in the basin of a periodic attractor. Then there exists a unique periodic orbit minimizing the period of all periodic orbit intersecting $J_{0}\setminus J_{1}$.
\end{Corollary}
\dem
Using Corollary~\ref{Corollary9988}, one can easily adapted the proof of the ``Variational Principle'' (Proposition \ref{LemmaVarPric}) to this case.
\cqd

\begin{Lemma}\label{semnomeSCH}  Let $f:[0,1]\setminus\{c\}\to[0,1]$ be a $C^{3}$ contracting Lorenz map with negative Schwarzian derivative.  
Let $J_{0}\supset J_{1}=(a,b)$ be two consecutive renormalization intervals with $J_{0}$ regular and not contained in the basin of a periodic attractor. If $L_J\in\cn_{per}$, $L_J=(p_J,p_J')$ is the connected component of $[0,1]\setminus\co^+_f(p_J)$ containing $c$, where $\period(p_J)=\min\{\period(x)\,;\,x\in Per(f)\cap J_{0}\setminus J_{1}\}$, then
\begin{enumerate}
\item $\co^+_f(x)\cap L_J\ne\emptyset$ for all $x\in J\setminus\co^-_f(p_J)$;
\item $J'\subset L_J$ for all renormalization interval $J'\subset J$.
\end{enumerate}
\end{Lemma}
\dem Changing Corollary~\ref{Corollary545g55} by Lemma~\ref{Lemma545g55} and Lemma~\ref{renormalinks} by Lemma~\ref{renormalinksSCH}, the proof is exactly the same as the proof of Lemma~\ref{semnome}.
\cqd

\subsection{The degenerate renormalization interval}
\label{SecDegRenInt}

Suppose that  $f:[0,1]\setminus\{c\}\to[0,1]$ is a $C^3$ contracting Lorenz map with negative Schwarzian derivative with a periodic attractor and without non-regular renormalization intervals (see, for instance, Example~\ref{ExDegenerado}). In this case, there is a unique periodic attractor $\Lambda$ and only one of the critical values ($\{f(c_{-}),f(c_{+})\}$ is attracted by $\Lambda$. That is, either $c_{-}\in\beta(\Lambda)\not\ni c_{-}$ or $c_{-}\not\in\beta(\Lambda)\ni c_{-}$. Let $I$ be the smallest regular renormalization interval or, if $f$ is not renormalizable, set $I=(0,1)$. As $Sf<0$, there is a periodic point $a\in I$ such that $(a,c)$ or $(c,a)\in\subset\beta(\Lambda)$. We may assume that $(a,c)\in\subset\beta(\Lambda)$. Let $\ell=\period(a)$ and let $T$ be the maximal interval containing $a$ such that $f^{\ell}|_{T}$ is a homeomorphism.  As $Sf<0$, we get in this case that if $a$ is the unique fixed point of $f^{\ell}|_{T}$ then $a$ is a saddle-node and $a\in\Lambda$. If $\Lambda$ is a super-attractor ($c\in\Lambda$) then $a$ is the unique fixed point of $f^{\ell}|_{T}$ and $Df^{\ell}(a)>1$. If $\Lambda$ is not a super-attractor then $f^{\ell}|_{T}$ has exactly two fixed points $a_{0}<a_{1}<c$ (of course $a\in\{a_{0},a_{1}\}$). Moreover, $Df^{\ell}(a_{0})>1>Df^{a_{1}}$, $a_{1}\in\Lambda$ and $f^{\ell}((a_{0},c))\subset(a_{0},c)\subset\beta(\Lambda)$.

So, there is $\alpha<c$ ($\alpha\in I$) such that $\co_{f}^{+}(\alpha)\cap(\alpha,c)=\emptyset=(\alpha,c)\cap\co_{f}^{+}(c_{+})$ and such that $f^{\period(\alpha)}(\alpha,c))\subset(\alpha,c)$.

\begin{Definition}[Degenerate Renormalization Intervals]
We say that an interval $I=(a,c)$ is a degenerate renormalization interval if there is $n$ such that $f^{n}(I)\subset I$ and $\co_{f}^{+}(a)\cap I=\emptyset= I\cap\co_{f}^{+}(c_{+})$. Analogously, an interval $I=(c,a)$ is called a degenerate renormalization interval if there is $n$ such that $f^{n}(I)\subset I$ and $\co_{f}^{+}(a)\cap I=\emptyset=I\cap\co_{f}^{+}(c_{-})$.
\end{Definition}

\begin{Example}[Figure~\ref{ExemploDegeneradoA.png}]\label{ExDegenerado}
Let $f:[0,1]\setminus\{0.5\}\to[0,1]$ given by
$$
f(x)=\begin{cases}
3.4 x (1 - x)&\text{ if }x<0.5\\
1 - 4 x (1 - x)&\text{ if }x>0.5
\end{cases}
.$$

%%%%%%%%%%%%%%%%%%%%%%%%%%%%%%%%%%%%%%%%%%%%
\begin{figure}
\begin{center}\label{ExemploDegeneradoA.png}
\includegraphics[scale=.23]{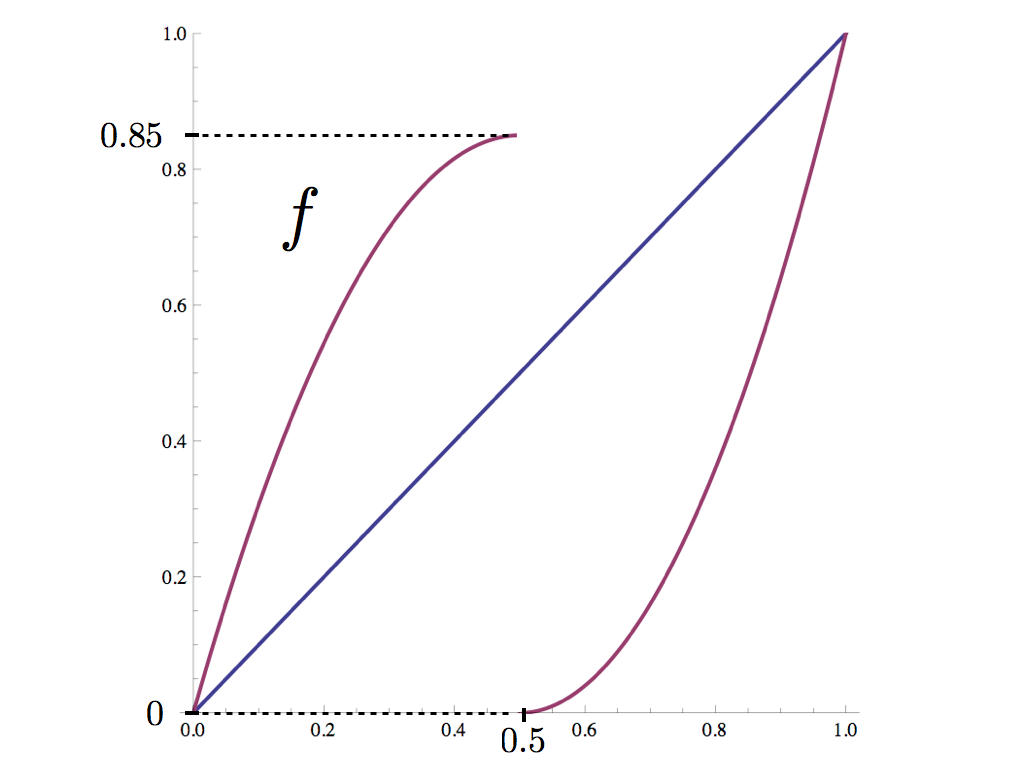}\\
Figure~\ref{ExemploDegeneradoA.png}
\end{center}
\end{figure}
%%%%%%%%%%%%%%%%%%%%%%%%%%%%%%%%%%%%%%%%%%%%%%

In this example $c=0.5$, $f(c_{-})=0.85$ and $f(c_{+})=0$. As $f(c_{+})=0$, $f$ does not admit a renormalization. Indeed, if $I=(a,b)\subsetneqq(0,1)$ is a renormalization interval then we would have $0<a<c<b<1$ and $f^{\period(a)}((a,c))\subset(a,b)\supset f^{\period(b)}((c,b))$. But as $f(c_{+})=0$, $f^{\period(b)}((c,b))=(0,b)\not\subset(a,b)$. On the other hand $\Lambda=\{p,q\}$, where $p\approx0.48880830755049054\cdots$ and $q\approx0.849574136468393\cdots$  is an attracting periodic orbit (of period $2$) for $f$ (Figure~\ref{ExemploDegeneradoB.png}).

%%%%%%%%%%%%%%%%%%%%%%%%%%%%%%%%%%%%%%%%%%%%
\begin{figure}
\begin{center}\label{ExemploDegeneradoB.png}
\includegraphics[scale=.37]{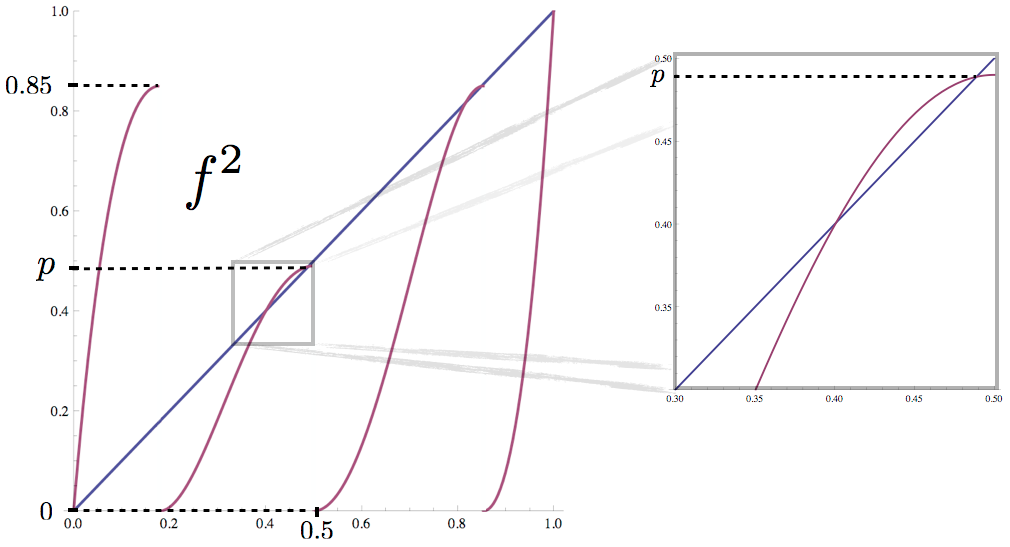}\\
Figure~\ref{ExemploDegeneradoB.png}
\end{center}
\end{figure}
%%%%%%%%%%%%%%%%%%%%%%%%%%%%%%%%%%%%%%%%%%%%%%

\end{Example}

The proof of the following lemma is immediate.
\begin{Lemma}\label{LemmaUmOuOutro}
If $f:[0,1]\setminus\{c\}\to[0,1]$ is a $C^{3}$ contracting Lorenz map with negative Schwarzian derivative then either $f$ has a non-regular renormalization interval or $f$ has a degenerate renormalization interval.
\end{Lemma}

\begin{Lemma}[The maximal degenerate renormalization interval]\label{LemmaS65Deg} Let $f:[0,1]\setminus\{c\}\to[0,1]$ be a $C^{3}$ contracting Lorenz map with negative Schwarzian derivative. If $f$ has a degenerate renormalization interval then there is a degenerate renormalization interval $J_{max}$ such that $J_{max}\supset I$ for every degenerate renormalization interval $I$.
\end{Lemma}
\dem The same of the Proof of Lemma~\ref{LemmaS6546SS}. \cqd

\begin{Lemma}[$\Omega(f)$ inside a degenerate renormalization interval]\label{LemmaSSDeg} Let $f:[0,1]\setminus\{c\}\to[0,1]$ be a $C^{3}$ contracting Lorenz map with negative Schwarzian derivative. If $J$ is a degenerate renormalization interval then $J\subset\beta(\Lambda)$, where $\Lambda$ is a periodic attractor. In particular, $\Omega(f)\cap J=\Omega(f)\cap\Lambda$. Furthermore, $\Lambda$ is the unique periodic attractor of $f$.
\end{Lemma}
\dem Straightforward.\cqd

\subsection{Proof of the Spectral Decomposition}\label{SecSepDECOM}

Let $f:[0,1]\setminus\{c\}\to[0,1]$ be a $C^3$ contracting Lorenz map with negative Schwarzian derivative.

Let $\cR_{f}$ be the following sequence of nested renormalization intervals of $f$. If all renormalization intervals of $f$ are regular, then let $\cR_{f}$ be the nested sequence $(0,1)\supsetneqq I_{1}\supsetneqq I_{2}\supsetneqq I_{3}\cdots$ for all renormalization intervals of $f$ (by Lemma~\ref{renormalinksSCH} the sequence of regular renormalization interval is a nested sequence). If there is an non-regular (or a degenerate) renormalization interval, let $J_{max}$ be the maximal non-regular (or a degenerate) renormalization interval (see Lemma~\ref{LemmaS6546SS} (or Lemma~\ref{LemmaS65Deg})). Note that we can not have, for the same map, a non-regular and a degenerate renormalization interval (Lemma~\ref{LemmaUmOuOutro}).  The presence of a non-regular or a degenerate renormalization interval implies in the existence of a periodic attractor and the finitude of the number of renormalization intervals. In this case, let $J_{1}\supsetneqq I_{2}\supsetneqq\cdots \supsetneqq J_{t}$ be the nested sequence of all regular renormalization intervals. Note that $J_{t}\supset\overline{J_{max}}$ (Lemma~\ref{LemmaJ1J2} and \ref{renormalinksSCH}). Thus, set $I_{j}=J_{j}$ for $1\le j\le t$, $I_{t+1}=J_{max}$ and define $\cR_{f}=\{I_{n}\}_{1\le n\le t+2}$.

\begin{Definition}[The nested sequence of renormalization intervals]\label{NestRI}
Given a $C^3$ contracting Lorenz map $f$ with negative Schwarzian derivative, define {\em the nested sequence of renormalization intervals} of $f$ as the sequence $\cR_{f}$ constructed above.
\end{Definition}

\begin{Notation}
Given $I_{n}\supset I_{n+1}$ two consecutive intervals of the nested sequence of renormalization intervals $\cR_{f}$, define $K_{n}=K_{I_{n}}$  and $$\Omega_{n}(f)=\Omega(f)\cap\big(K_{n}\setminus K_{n+1}\big).$$
\end{Notation}

\begin{definition}[$J_x$, the $x$-phobic critical gap]
Let $f:[0,1]\setminus\{c\}\to[0,1]$ be a contracting Lorenz map. Write $v_{1}=\sup f([0,c))$ and $v_{0}=\inf f((c,1])$ For any $v_{0} < x <v_{1}$ define $J_{x}$ as the connected component of $[0,1]\setminus\alpha_{f}(x)$ containing $c$.% If there is  a connected component $J$ of $[0,1]\setminus\alpha_{f}(x)$  such that $c \in \overline J$ then define $J_{x}$ as $J$. Otherwise, set $J_x=\emptyset$. 
\end{definition}

Recall that Lemma \ref{jotaxrenormaliza} states that if $J_{x}\ne\emptyset$ then $J_{x}$ is a renormalization interval and $\partial J_{x}\subset\alpha_{f}(x)$.

\begin{Lemma}\label{LemmaPrimeiroEspectro}Let $f:[0,1]\setminus\{c\}\to[0,1]$ be a $C^3$ contracting Lorenz map with negative Schwarzian derivative. Let $I_n\supset I_{n+1}$, be two consecutive intervals of the nested sequence of renormalization intervals $\cR_{f}$, with $I_{n}$ being regular.  Given $x \in I_{n} \setminus K_{I_{n+1}}$ then $J_x$ is either $I_{n}$ or $I_{n+1}$. Also, it is always true that $\alpha_{f}(x) \supset \partial I_{n}$, and if $\alpha_f(x) \cap \partial I_{n+1} \ne \emptyset$ indeed $\alpha_f(x) \supset \partial I_{n+1}$.

%
%{\color{red} REVER:
%If $I\ne (0,1)$ is the smallest renormalization interval of $f$ and has no periodic attractor, then $\forall x \in I$, $\alpha_{f}(x) \supset \partial I$.}%color

\end{Lemma}

\dem Let $x \in I_{n} \setminus K_{I_{n+1}}$, so, as $x \not\in K_{I_{n+1}}$ no pre-image of it can be inside this set, for otherwise it would never leave it  (as it is a positively invariant set), so $c\notin\alpha_{f}(x)$ (as $c\in K_{I_{n+1}}$). Thus, $J_x\ne\emptyset$, and then it is a renormalization interval by the former observation (using lemma \ref{jotaxrenormaliza}). As $\alpha_{f}(x)\cap I_{n+1}=\emptyset$ (because $I_{n+1}\subset K_{I_{n+1}}$), we get $J_{x}\supset I_{n+1}$. 

This way, $J_x$ has to be $I_{n+1}$ or perhaps other renormalization interval containing it.

Write $I_n=(a_n,b_n)$.
We will suppose that $x \in (a_n,c)\setminus K_{I_{n+1}}$ (the case $x\in(c,b_{n})\setminus K_{I_{n+1}}$ is analogous). Given $\forall \varepsilon >0$ let $T_{\varepsilon}=(a_{n},a_{n}+\varepsilon)$, it follows from Lemma \ref{Lemma657h6} that $f^{\theta(T_{\varepsilon})}(T_{\varepsilon})\supset (a,c)\ni x$. And this way we have \begin{equation}\label{Eq897295t5}
a_n \in \alpha_{f}(x),
\end{equation}
for pre-images of $x$ can be found as close to $a_n$ as we want. Then $J_x$ cannot be $I_j$ for any $I_j \supsetneqq I_{n}$ as, by definition, $J_x \cap \alpha_{f}(x) \ne \emptyset$ and any other renormalization interval containing $J_x$ would have to contain it properly (they couldn't even share a point in the border, by Lemma \ref{renormalinksSCH}), and in this case, $J_x$ can be $I_{n}$, but no other bigger renormalization interval. So, these are the only two possibilities,
\begin{equation}\label{Eq09876567}
J_{x}=\text{ either }\,I_{n}\text{ or }I_{n+1}.
\end{equation}

If $J_x=I_{n}$, by definition of $J_x$, we get $\alpha_{f}(x) \supset \partial I_{n}$. By (\ref{Eq09876567}), we have $\alpha_{f}(x)\cap\partial I_{n+1}\ne\emptyset$ $\iff$ $J_{x}=I_{n+1}$. Thus, assume that $J_x=I_{n+1}$. In this case, again by the definition of $J_{x}$, we get $\alpha_{f}(x) \supset \partial I_{n+1}$ (in particular, $\alpha_{f}(x)\cap\partial I_{n+1}\ne\emptyset$ $\implies$ $\alpha_{f}(x)\supset\partial I_{n+1}$). Now, we will show that we also have $\alpha_{f}(x)\supset\partial I_{n}$. By (\ref{Eq897295t5}), we already know that $a_{n}\in\alpha_{f}(x)$. On the other hand, as $b_{n+1}\in\partial I_{n+1}\subset\alpha_{f}(x)$ and $\alpha_{f}(x)\cap I_{n+1}=\emptyset$, there is some $x'\in\co_{f}^{-}(x)\cap(b_{n+1},b_{n})$. So, an analogous reasoning as we did for $x\in(a_n,a_{n+1})$ above can be done to $x'\in(b_{n+1},b_{n})$ and we also get $\alpha_{f}(x)\supset\alpha_{f}(x') \ni b_{n}$. 
%We can also observe that by the definition of $J_x$, if $\alpha_f(x) \cap \partial I_{n+1} \ne \emptyset$ indeed $\alpha_f(x) \supset \partial I_{n+1}$.

\cqd

\begin{Lemma}
\label{2162}
Suppose that $f:[0,1]\setminus\{c\}\to[0,1]$ is a $C^3$ contracting Lorenz map with negative Schwarzian derivative. Let $I_n\supset I_{n+1}$ be two consecutive intervals of the nested sequence of renormalization intervals $\cR_{f}$, with $I_{n}$ being regular and not contained in the basin of a periodic attractor, then 
$$
\alpha_f(x) \cap \partial I_{n+1} \ne \emptyset \Rightarrow \alpha_f(x) \supset \Omega_{n}(f)
$$
$\forall x \in I_n \setminus K_{n+1}$.
\end{Lemma}
\dem
Let $x\in I_n \setminus K_{n+1}$ such that $\alpha_f(x) \cap \partial I_{n+1} \ne \emptyset $ and given $y \in \Omega_{n}(f)=\Omega(f)\cap (K_n\setminus K_{n+1})$ consider any neighborhood $V$ of $y$. As $y$ is non-wandering, there is $z \in V$, and $j \in \NN$ such that $f^j(z) \in V$ (we may assume $z \not \in \co^-_f (c) \cup \co^-_f(Per(f))$). 

We claim that there is $k \in \NN$ such that $f^{jk}((z,f^j(z))) \cap \partial K_{n+1} \ne \emptyset $. 
Indeed, if $f^{js}((z,f^j(z)))$ $\cap$  $K_{n+1}=\emptyset$ 
for any $s$, then $f^t|_{[z,f^j(z)]}$ is a well defined homeomorphism, for any $t$. Observe that  $\bigcup_{s \ge 0}f^{js}(z,f^j(z)]$ is an interval, say, $(z,q)$. Moreover $f^j((z,q)) \subset (z,q)$ and $(z,q)\cap I_{n+1} = \emptyset$.
But this can not be the case, because it 
would imply, by Corollary~\ref{Corollary9988}, that $J_{0}$ is in the basin of a periodic attractor.

We can observe that $f^{jk}((z,f^j(z))) \cap  K_{n+1} \ne \emptyset$ $\implies$ $f^{s}((z,f^j(z))) \cap  I_{n+1} \ne \emptyset$ for some $s\in\NN$.
Consider the sequence $\{ f^{js}(z)\}_{j\in \NN}$. By what we concluded above, there is a minimum $s$ such that $f^{js}(z) \not\in I_{n+1}$ and $f^{j(s+1)}(z) \in I_{n+1}$, say $s=k$, then  $f^{jk}((z,f^j(z)))$ is such that one border is within $I_{n+1}$ and the other is outside it. Then we prove what was claimed.
 
Let $k$ be as above. 

As $x$ is taken in $I_n \setminus K_{n+1}$, the former lemma says that if $\alpha_f(x) \cap \partial I_{n+1} \ne \emptyset $, then $\alpha_f(x) \supset \partial I_{n+1}$, and we have that $\co^-_f (x) \cap f^{jk}(z,f^j(z)) \ne \emptyset$ and then $\co^-_f (x) \cap V \supset \co^-_f (x) \cap (z,f^j(z)) \ne \emptyset$.

As $V$ is any neighborhood, we can conclude $y \in \alpha_f(x)$.
 
\cqd

\begin{Notation}{}
If $I_{n}$ is regular, in the same way that we defined in Lemma \ref{semnomeSCH}, let $\co_{n}$ be the periodic orbit with minimum period that intersects $I_{n}$, let $L_n=(p_n,p_n')$ be the connected component of $[0,1]\setminus\co_{n}$ containing $c$. 
\end{Notation}

Note that $L_{n}\subset I_{n}$ is a nice interval and, if $I_{n+1}$ exists, then $I_{n}\supset L_{n}\supset I_{n+1}$.

\begin{Lemma}\label{bordos}Suppose that $f:[0,1]\setminus\{c\}\to[0,1]$ is a $C^3$ contracting Lorenz map with negative Schwarzian derivative. If $I_{n}$ is a regular renormalization interval then $\alpha(p_{n})=\alpha(p_{n}') \supset \partial I_{n+1}$. 

\end{Lemma}

\dem
Observe that $J_{p_{n}}=J_{p_{n}'}$ and $J_{p_{n}}\ne I_n$. Indeed, as $p_{n}$ is periodic, $p_{n}\in \alpha_{f}(p_{n})$. This way, the connected component of the complement of $\alpha_{f}(x)$ containing $c$, $J_{p_{n}}$, can't be $I_n$, then it has to be $I_{n+1}$. And then $\alpha_{f}(p_{n})=\alpha(p_{n}')\supset \partial J_{p_{n}}=\partial I_{n+1}$.

\cqd

\begin{Lemma}
Let $f:[0,1]\setminus\{c\}\to[0,1]$ be a $C^3$ contracting Lorenz map with negative Schwarzian derivative. Suppose that $I_n\supset I_{n+1}$ are two consecutive intervals of the nested sequence of renormalization intervals $\cR_{f}$ and that  $I_{n}$ is regular. If $L_{n}=I_{n+1}$ then $\Omega_{n}(f)=\co_{f}^{+}(p_{n})$. In particular, $\Omega_{n}(f)$ is transitive.
\end{Lemma}
\dem
If $L_{n}= I_{n+1}$, that is $I_{n+1}=(p_{n},p_{n}')$, then it follows from Theorem \ref{imperfeito} that $I_{n}=\xi (p_{n},p_{n}')$ (note that $\co_{f}^{+}(p_{n})=\co_{f}^{+}(p_{n}')$). Using Corollary~\ref{Corollaryimperfeito}, we get $\Omega_{n}(f)=\co_{f}^{+}(p_{n})$.
\cqd

\begin{definition}{} Given a regular renormalization interval $I_{n}$, define $$P_n:=\big( f^{period(p_{n}')}(c_+),f^{period(p_n)}(c_-) \big) \setminus J_{p_{n+1}}.$$
\end{definition}

%faz um papel semelhante ao E do caso da ultima renormalizacao, no atratores topologicos

\begin{Lemma}
\label{anelarmadilha}
Let $f:[0,1]\setminus\{c\}\to[0,1]$ be a $C^3$ contracting Lorenz map with negative Schwarzian derivative. If $I_n\supset I_{n+1}$ are two consecutive intervals of the nested sequence of renormalization intervals $\cR_{f}$ with $I_{n}$ regular then $\forall x \in P_n$ 
, $\alpha_f(x) \cap \partial I_{n+1} \ne \emptyset$.
\end{Lemma}

\dem
As $\alpha_{f}(x)\cap I_{n+1}=\emptyset$ $\forall x \in P_n$, we must have $J_x \supset I_{n+1}$.
Then $J_x$ necessarily will be $I_n$ or $I_{n+1}$.
For any point inside $P_n$ we can have at least one pre-image inside $P_n$ (indeed, inside $L_{n}\subset P_{n}$). Then, given $x \in P_n$, we can construct an infinite pre-orbit of $x$, all inside the compact $P_n$, then $\alpha_{f}(x) \cap P_n \ne \emptyset$. So, $J_x$ cannot be  $I_{n}$, it has to be $I_{n+1}$ and $\alpha_f(x) \cap \partial I_{n+1} \ne \emptyset$ as stated.

\cqd

Using Lemma~\ref{LemmaPrimeiroEspectro}, we get from Lemma~\ref{anelarmadilha} the following result.

\begin{Corollary}
\label{reciproca}
Suppose that $f:[0,1]\setminus\{c\}\to[0,1]$ is a $C^3$ contracting Lorenz map with negative Schwarzian derivative. If $I_n\supset I_{n+1}$ are two consecutive intervals of the nested sequence of renormalization intervals $\cR_{f}$ with $I_{n}$ regular then  then  
$\alpha_f(x) \supset \Omega_{n}(f)$ $ \forall x \in P_n$.
\end{Corollary}

For a Lorenz map $f$ with negative Schwarzian derivative with two consecutive renormalization intervals $I_n\supset I_{n+1}$, let us define $$\EE_n = \{x \in I_n ; \alpha_{f}(x) \ni \partial I_{n+1}  \}.$$
If follows from Lemma~\ref{LemmaPrimeiroEspectro}~and~\ref{anelarmadilha} that $\EE_n$ contains $P_n \cup I_{n+1}$. Let $(\widehat{a}_{n},\widehat{b}_{n})$ be the maximal open interval of such that $$P_{n}\cup J_{n+1}\subset(\widehat{a}_{n},\widehat{b}_{n})\subset I_{n}.$$

\begin{Lemma}Let $f:[0,1]\setminus\{c\}\to[0,1]$ be a $C^3$ contracting Lorenz map with negative Schwarzian derivative. If $I_n\supset I_{n+1}$ are two consecutive intervals of the nested sequence of renormalization intervals $\cR_{f}$ with $I_{n}$ regular then
$\exists \ell_{n}$ and $r_{n}>0$ such that $f^{\ell_{n}}((\widehat{a}_{n},c))\subset(\widehat{a}_{n},\widehat{b}_{n})\supset f^{r_{n}}((c,\widehat{b}_{n})).$ 
\end{Lemma}

\dem
We will show that $\exists \ell_{n}$ such that $f^{\ell_{n}}((\widehat{a}_{n},c))\subset(\widehat{a}_{n},\widehat{b}_{n})$ (the proof of the other inclusion is analogous).  As $Per(f)\cap(\widehat{a}_{n},c)\ne\emptyset\ne(c,\widehat{b}_{n})\cap Per(f)$, there is a minimum k such that $f^k((\widehat{a}_{n},c))\cap (\widehat{a}_{n},\widehat{b}_{n}) \ne \emptyset$.
Suppose that $f^k((\widehat{a}_{n},c))\not\subset (\widehat{a}_{n},\widehat{b}_{n})$. For example, $\widehat{b}_{n} \in  f^k((\widehat{a}_{n},c))$. Then $\exists y \in f^k((\widehat{a}_{n},c))$ such that $y \not \in \EE_n$. But in this case there is a pre-image $w$ of $y$ in $(\widehat{a}_{n},\widehat{b}_{n})$, that is a subset of $\EE_n$, so, as $\alpha_f(w) \ni \partial I_{n+1}$, $\alpha_f(y) \ni \partial I_{n+1}$, absurd. 
\cqd

For a Lorenz map $f$ with negative Schwarzian derivative with two consecutive renormalization intervals $I_n\supset I_{n+1}$ and $\ell_{n}$, with $J_{0}$ regular and $r_{n}$ as given by the former lemma, we define 
\begin{equation}\label{regiaoarmadilha2}
\UU_n = (\widehat{a}_{n},\widehat{b}_{n}) \cup \big(\bigcup_{j=1}^{\ell_{n}-1}f^j((\widehat{a}_{n},c))\big) \cup \big(\bigcup_{j=1}^{r_{n}-1}f^j((c,\widehat{b}))\big)
\end{equation}
and we have that $\UU_n$ is a trapping region, that is, $f( \UU_n \setminus \{c\}) \subset \UU_n$

%Let $\Lambda_n=\overline{\Omega \cup \UU_n}$.

\begin{Corollary}
\label{sobrealfa2}
Suppose that $f:[0,1]\setminus\{c\}\to[0,1]$ is a $C^3$ contracting Lorenz map with negative Schwarzian derivative. Let  $I_n\supset I_{n+1}$ be two consecutive intervals of the nested sequence of renormalization intervals $\cR_{f}$. If $I_{n}$ is regular and it is not contained in the basin of a periodic attractor. Then 
$$\alpha_{f}(x) \supset \Omega_{n}(f)\hspace{.3cm}\forall x \in \UU_n.$$
\end{Corollary}
\dem
As Lemma \ref{2162} says $\alpha_f(x) \supset \Omega_{n}(f)$ to any $x$ such that $\alpha_f(x) \ni \partial I_{n+1}$, this holds for any $x$ in $\UU_n$, as this is contained in $\EE_n$.
\cqd

\begin{Lemma}\label{LemmaUU78} Suppose that $f:[0,1]\setminus\{c\}\to[0,1]$ is a $C^3$ contracting Lorenz map with negative Schwarzian derivative. Let $I_n\supset I_{n+1}$ be consecutive intervals of the nested sequence of renormalization intervals $\cR_{f}$, with $I_{n}$  being a regular renormalization interval not contained in the basin of a periodic attractor. Given $x \in I_n\setminus K_{n+1}$, if $\alpha_f(x) \supset \partial I_{n+1}$ then $\alpha_f(x)\cap\UU_n \subset \Omega_{n}(f)\cap\UU_n$. 
\end{Lemma}

\dem
Consider $x$ such that $\alpha_f(x) \supset \partial I_{n+1}$. Given $y \in \alpha_f(x)\cap\UU_{n}$ consider any neighborhood $V$ of $y$. We may assume $V \subset \UU_n$. Also, as $y \in \alpha_f(x)$ and $x \in I_n\setminus K_{n+1}$, then $y$ must also be in $\UU_n\setminus K_{n+1}$. Note that there is some $0\le j\le\max\{\ell_{n},r_{n}\}$ such that $f^{j}(y)\in I_n\setminus K_{n+1}$. As $f^{j}(y)\in\Omega_{n}(f)\implies y\in\Omega_{n}(f)$, we may assume (changing $y$ by $f^{j}(y)$ if necessary) that $y\in I_n\setminus K_{n+1}$.

\begin{Claim}\label{Claimxixixixix1}
$y \in \overline{(V\setminus\{y\})\cap\co^-_f(x)}$.
\end{Claim}
\dem
On the contrary, $\exists \epsilon >0$ such that $B_\epsilon(y)\cap\co^-_f(x)=\{y\}$. In this case, we have that $\exists n_1<n_2<...<n_j \to \infty$ such that $f^{n_j}(y)=x$. 
Then $$x=f^{n_2}(y)=f^{n_2-n_1}(f^{n_1}(y))=f^{n_2-n_1}(x).$$

Observe that if $f^s(B_\epsilon(y)) \cap \partial I_{n+1} = \emptyset$ $ \forall s$ then writing $(\alpha, \beta)=f^{n_1}(B_\epsilon(y))$ we have
$$x \in (\alpha,\beta) \text{ and } f^{k(n_2-n_1)}((\alpha,\beta))\cap \partial I_{n+1} = \emptyset\; \forall k.$$
Taking $(x,\gamma)=\bigcup_{k\ge1}f^{k(n_2-n_1)}((x,\beta))=\bigcup_{k\ge1}(x,f^{k(n_2-n_1)}(\beta))$
we have $f^{n_2-n_1}|_{(x,\gamma)}$ homeomorphism and $f^{n_2-n_1}((x,\gamma))\subset(x,\gamma).$
But this would imply, by Corollary~\ref{Corollary9988}, that $J_{n}$ is containing in the basin of a periodic attractor, contradicting our hypothesis. Then, we necessarily have that $\exists s$ such that $f^s(B_\epsilon(y))\cap \partial I_{n+1} \ne \emptyset$. 
As $\alpha_{f}(x)\cap \partial I_{n+1} \ne \emptyset$ we would have that $\#\co_f^-(x)\cap B_\epsilon(y)=\infty$. Again an absurd, proving the claim.
\cqd
Because of the Claim~\ref{Claimxixixixix1} above, we may assume that $y\in \overline{(y,1)\cap V\cap \co^-_f(x)}$ (the case when we have that $y\in \overline{(0,y)\cap V\cap \co^-_f(x)}$ is analogous).
We may take $x_1<x_2\in(y,1)\cap V \cap\co_f^-(x)$ such that $f^{n_2}(x_2)=x=f^{n_1}(x_1)$ with $n_1<n_2$.

\begin{Claim}\label{Claimxixixixix2}
$\exists\,s\ge0$ such that $c\in f^s([x_1,x_2))$.
\end{Claim}
\dem 
% TOTALMENTE ANALOGA AO CASO DA ULTIMA RENORMALIZACAO

If $c\notin f^s([x_1x_2))$ $\forall\,s\ge0$ then $c\notin f^{k(n_2-n_1)}([f^{n_2-n_1}(x),x))$ $=$ $f^{k(n_2-n_1)+n_2}([x_1,x_2))$ $\forall k \in \NN.$
As $f$ preserves orientation, $f^{k(n_2-n_1)}|_{[f^{n_2-n_1}(x),x)}$ is a homeomorphism $\forall x$, so we have $$f^{(k+1)(n_2-n_1)}(x)<f^{k(n_2-n_1)}(x)\,\,\forall k.$$
Then, $\bigcup_{k\ge0}f^{k(n_2-n_1)}([f^{n_2-n_1}(x),x))$ is an interval $(\gamma,x)$. Besides that, $f^{n_2-n_1}|_{(\gamma,x)}$ is a homeomorphism and $f^{n_2-n_1}((\gamma,x))\subset(\gamma,x)$. 
But this is an absurd, because it would imply, by Corollary~\ref{Corollary9988}, that $J_{0}$ is containing in the basin of a periodic attractor.

\cqd

By Claim~\ref{Claimxixixixix2} there is a smaller integer $s\ge0$ such that $c\in f^s([x_1,x_2))$. Notice that $x_1,x_2\notin I_{n+1}$, as $x_1,x_2\in\co_f^-(x)$ and $x\in I_n\setminus K_{n+1}$. Thus, $\overline{ I_{n+1}}$ $\subset$ $\interior f^s([x_1,x_2))$.
As $x_1 \in \UU_n$, we have that $\co^-_f(x_1)$ accumulates in both points of $\partial I_{n+1}$. Then, $\co_f^-(x_1)\cap f^s([x_1,x_2))\ne \emptyset$.
This implies that $\exists x_1' \in \co^-_f(x_1)\cap[x_1,x_2)\subset V$, say $x_1' \in f^{-t}(x_1)\cap V$. Then,
$$
f^t(V)\cap V \ne \emptyset
$$

As $V$ is a neighborhood of $y \in \UU_n$ that was arbitrarily taken, we may conclude that $y \in \Omega(f)$. Because $y\in\UU_{n}\setminus K_{n+1}\subset K_{n}\setminus K_{n+1}$, we get $y\in\Omega_{n}(f)$.

\cqd

\begin{Corollary}
\label{poiupoiu2} Let $f:[0,1]\setminus\{c\}\to[0,1]$ is a $C^3$ contracting Lorenz map with negative Schwarzian derivative and $I_n\supset I_{n+1}$ be two consecutive intervals of the nested sequence of renormalization intervals $\cR_{f}$. If $I_{n}$ is regular and it is not contained in the basin of a periodic attractor then
$$\alpha_{f}(x)\cap \UU_n = \Omega_{n}(f)\;\; \forall x \in I_n \setminus K_{n+1}.$$
%In particular, given any $x \in \Omega_{n}(f)$,\, $\overline{\co^-_f(x)\cap\Omega_{n}(f)}=\Omega_{n}(f)$. 
\end{Corollary}
\dem
It follows from Lemma~\ref{LemmaUU78} and Lemma~\ref{sobrealfa2} that $$\alpha_{f}(x)\cap \UU_n = \Omega_{n}(f)\cap \UU_n\;\; \forall x \in I_n \setminus K_{n+1}.$$
As $\UU_{n}$ is an open trapping set with $\UU_{n}\supset\overline{I_{n+1}}$, we can apply Lemma~\ref{LemmaofTr} and conclude that $\co_{f}^{+}(x)\cap\UU_{n}$ for all $x\in I_{n}$. This implies that $\Omega_{n}(f)\subset\UU_{n}$. That is $\Omega_{n}(f)\cap\UU_{n}=\Omega_{n}(f)$.

\cqd

\begin{Proposition} \label{transitivo2} Let $f:[0,1]\setminus\{c\}\to[0,1]$ is a $C^3$ contracting Lorenz map with negative Schwarzian derivative. If $I_n\supset I_{n+1}$ are two consecutive intervals of the nested sequence of renormalization intervals $\cR_{f}$ with $J_{n}$ being regular and not contained in the basin of a periodic attractor, then $f|_{\Omega_{n}(f)}$ is strongly transitive.
In particular, $$f|_{\Omega_{n}(f)}\text{ is transitive.}$$
\end{Proposition}

\dem
We will show that $\bigcup_{j\ge0}f^j(V\cap\Omega_{n}(f)) = \Omega_{n}(f), \forall V  \subset \UU_n \setminus I_{n+1}$, $V$ open and $V \cap \Omega_{n}(f) \ne \emptyset$.

As we know that $f^{-1}(\alpha_{f}(x))\subset\alpha_{f}(x)$ for any $x \in [0,1]$, it follows from the Corollary~\ref{poiupoiu2} that 
\begin{equation}
\label{inclusaoestrela2}
f^{-1}(\Omega_{n}(f)\cap \UU_n)\cap \UU_n \subset\Omega_{n}(f)\cap \UU_n.
\end{equation}

Let $V \subset \UU_n \setminus I_{n+1}$, $V$ any open set with $V \cap \Omega_{n}(f) \ne \emptyset$. Given $x_0 \in \Omega_{n}(f)\cap \UU_n$ we have that $\alpha_{f}(x)\cap V\ne \emptyset$ and then $\co^-_f(x)\cap V \ne \emptyset$. Pick $x_t \in f^{-t}(x) \cap V$. Define $x_k=f^{t-k}(x_t)$ for $0 \le k \le t$. 
$$
x_t \stackrel{f}{\to} x_{t-1} \stackrel{f}{\to} \dots \stackrel{f}{\to} x_0=f^t(x_t) 
$$
As $\UU_n$ is a trapping region, we have that $x_k$ in $\UU_n, \forall 0 \le k \le t$. 

We claim that indeed $x_t \in \Omega_{n}(f)\cap \UU_n$. We will prove this claim by induction. We have that $x_0 \in \Omega_{n}(f)\cap \UU_n$. Suppose it also worths for $k-1$, that is, $x_{k-1} \in \Omega_{n}(f)\cap \UU_n$. We have that $x_k \in \UU_n$. Then $x_k \in f^{-1}(x_{k-1})\cap \UU_n$ and by (\ref{inclusaoestrela2}) we have that $x_k \in \Omega_{n}(f)\cap \UU_n$.
It follows by induction that $x_t \in \Omega_{n}(f)\cap \UU_n$.

\cqd

\begin{Remark}
In the Section~\ref{SecAnSpDeNonReMap} of the Appendix, we describe the possible decomposition to a non renormalizable $f$.
\end{Remark}

\dem[Proof of the Spectral Decomposition Theorem, Theorem \ref{spdecnotacao}]

We begin by setting $I_0:=(0,1)$. 
For $n\ge1$, $I_n$ is the $n$-th interval of the nested sequence of renormalization intervals $\cR_{f}$ (see the beginning of Section~\ref{SecSepDECOM} and Definition~\ref{NestRI}). If $\cR_{f}$ is a finite sequence, let $n_{f}-1$ be the biggest $n\ge1$ such that $I_{n}$ is not contained in the basin of periodic attractors. Otherwise, set $n_{f}=\infty$. Note that $\#\cR_{f}-1\le n_{f}\le\#\cR_{f}$.

Then, we found a sequence of consecutive renormalization intervals, $\{I_j\}_{j\ge 0}$, and to each $n\ge1$ we may associate $K_n:=K_{I_{n}}$, the nice trapping region associated to $I_n$. Set $K_0=(0,1)$.

Obviously if $n_{j}$ is finite, we have that
$$(0,1)=(K_0\setminus  K_1)\cup (K_1\setminus  K_2)\cup\cdots\cup(K_{n_f-1}\setminus  K_{n_f})\cup K_{n_{f}}.$$
If not, $f$ is a $\infty$-renormalizable map and
$$(0,1)=\bigg(\bigcup_{0\le n<n_{f}}K_n\setminus  K_{n+1}\bigg)\cup\bigg(\bigcap_{0\le n<n_{f}}K_{n}\bigg),$$
where $\Lambda:=\bigcap_{0\le n<n_{f}}K_{n}$ is the Solenoid Attractor of $f$.

Defining, $\Omega_{0}(f)=\{0,1\}$, 
$\Omega_{n}(f)= \overline{\Omega(f)\cap (K_n\setminus K_{n+1})}$ for $1\le n<n_{f}$ and
$$\Omega_{n_{j}}(f)=
\begin{cases}
\overline{\Omega(f)\cap K_{n_{f}}}&\text{ if }n_{f}<\infty\\
\bigcap_{0\le n<n_{f}}K_{n}&\text{ if }n_{f}=\infty
\end{cases},$$
we get
\begin{equation}\label{Eg5DS}
\Omega(f)=\bigcup_{0\le n \le n_f}\Omega_{n}(f).
\end{equation}

As $K_{n}$ is a trapping region, $\co_{f}^{+}(x)\subset K_{n}$ $\forall\,x\in K_{n}$ $\forall\,n$.

If $f$ does not have periodic attractors then every renormalization interval is regular (Remark~\ref{Remark5466}). Thus, it follows from Proposition~\ref{transitivo2} that $\Omega_{n}(f)$ is transitive $\forall n$, $\,1\le n<n_{f}$. Furthermore, as $f$ does not have periodic attractor, 
we can use Theorem~\ref{cicloint} to conclude that $\Omega_{n_{f}}$ is also transitive.
Indeed, the trapping region $U$ given by the Theorem has to be inside $K_{n_{f}}$ and by Lemma~\ref{LemmaofTr}, $\co_{f}^{+}(x)\cap U$ $\forall\,x\in I_{n_{f}}$. This implies that $\Omega(f)\cap K_{n_{f}}\subset \Omega(f)\cap U$, which is transitive.
Furthermore,$\forall\,1\le n\le n_{f}-1$, $\Omega_{n}(f)$ is a compact forward invariant uniformly expanding set (the expansion follows, for instance, from Ma\~ne's Theorem (Theorem~\ref{ThMane})).
Finally, if $\Omega_{n}(f)$ is not a periodic orbit, then $\Omega_{f}$ is a Cantor set (Lemma~\ref{aeroporto} in the Appendix provides $\Omega_{f}$ to be perfect and Lemma~\ref{lebzero} to have Lebesgue zero measure and so, to be totally disconnected).

If $f$ has a non-regular (or a degenerate) renormalization interval, then there is a finite attractor $\Lambda$ such that $I_{n_{f}}\subset\beta(\Lambda)$, where $\Lambda$ can be one periodic attractor or a union of two periodic attractors. Indeed, either $I_{n_{f}}$ is $J_{max}$, where $J_{max}$ is the maximal non-regular (or degenerate) renormalization interval, or $I_{n_{f}}$ is a regular renormalization interval contained in the basin of a periodic attractor (as in Lemma~\ref{Lemma9988}, see Figure~\ref{RRper.png}). As $J_{max}$ is always contained in the basin of a periodic attractor (Lemma~\ref{LemmaSSSS} and Lemma~\ref{LemmaSSDeg}), we get that $I_{n_{f}}\subset\beta(\Lambda)$. If $I_{n_{f}}=J_{max}$, the $\Omega_{n_{f}}(f)$ can be given by Lemma~\ref{LemmaSSSS} and \ref{LemmaSSDeg}. 
Moreover, $\forall\,1\le n\le n_{f}-1$, we have that $I_{n}$ is a regular renormalization interval not contained in the basin of a periodic attractor $\forall\,1\le n\le n_{f}-1$. Thus, $\Omega_{n}(f)$ is a transitive compact forward invariant uniformly expanding set (Proposition~\ref{transitivo2}, Lemma~\ref{aeroporto} and Lemma~\ref{lebzero}).

Finally, let us prove the last item of the theorem.

Let $0<j<n_{f}$. Thus, $I_n=(a_n,b_n)\supset I_{n+1}=(a_{n+1},b_{n+1})$ are consecutive intervals of the nested sequence of renormalization intervals $\cR_{f}$ with $I_{n}$ being regular.
Let $\co_{n}$ the periodic orbit with minimum period that intersects $I_{n}$, and $L_n=(p_n,p_n')$ be the connected component of $[0,1]\setminus\co_{n}$ containing $c$.

As $I_n$ is a renormalization interval (and $n<n_f$), there is only a finite number of gaps $J\in C_{I_n}$ of $\Lambda_{I_n}$ such that $\co_f^+(I_n)\cap J\ne\emptyset$.
Furthermore, by Theorem~\ref{imperfeito} and Corollary~\ref{Corollaryimperfeito}, the number of gaps of $\Lambda_{L_n}$ that intersect $\co_f^+(I_n)$ is also finite.
That is, $\#\XX_n<\infty$, where $\XX_n=\{I\in C_{L_n}\,;\,\co_f^+(I_n)\cap I\ne\emptyset\}$ is the set of gaps of
$\Lambda_{L_n}$ that intersect the positive orbit of the renormalization interval $I_n$.
Of course that $(p_n,p_n')\in\XX_n$. Set $X_{0,n}:=(p_n,p_n')$ and write $\XX_n=\{ X_{0,n}, X_{1,n},\cdots,X_{\ell_n,n}\}$, where $\ell_n=\#\XX_n-1$.

Of course that $\Omega_n= X_{0,n}\cup X_{1,n}\cup\cdots\cup X_{\ell_n,n}$,  $\# ( X_{a,n}\cap X_{b,n} )\le1$, when $a\ne b$, and that for each $i\in\{1,\cdots,\ell_{n}\}$ there is  $1\le s_{i,n}\le \ell_{n}$ such that $f^{s_{i,n}}(X_{i,n})\subset X_{0,n}$. So we only have to prove that the first return map to $X_{0,n}$ is topologically exact. 

If $L_n$ is a renormalization interval, $L_n=I_{n+1}$ and $\Omega_{n+1}=\co^+(p_n)$ (Corollary~\ref{Corollaryimperfeito}). In this case there is nothing to prove. Thus, suppose then that $L_n$ is not a renormalization interval, i.e., $L_n \ne I_{n+1}$. As $L_n$ is not a renormalization interval, the first return map $F_n$ to $(p_n,p_n')$ is defined in a domain that has more than two connected components. Then, at least one of the periodic points, let's say $p_n$ (the other case is analogous), returns in such a way that the branch of its return  covers the whole interval $(p_{n},p_{n}')$, as in Figure~\ref{FigFig.png}.

%%%%%%%%%%%%%%%%%%%%%%%%%%%%%%%%%%%%%%%%%%%%
\begin{figure}
\begin{center}\label{FigFig.png}
\includegraphics[scale=.27]{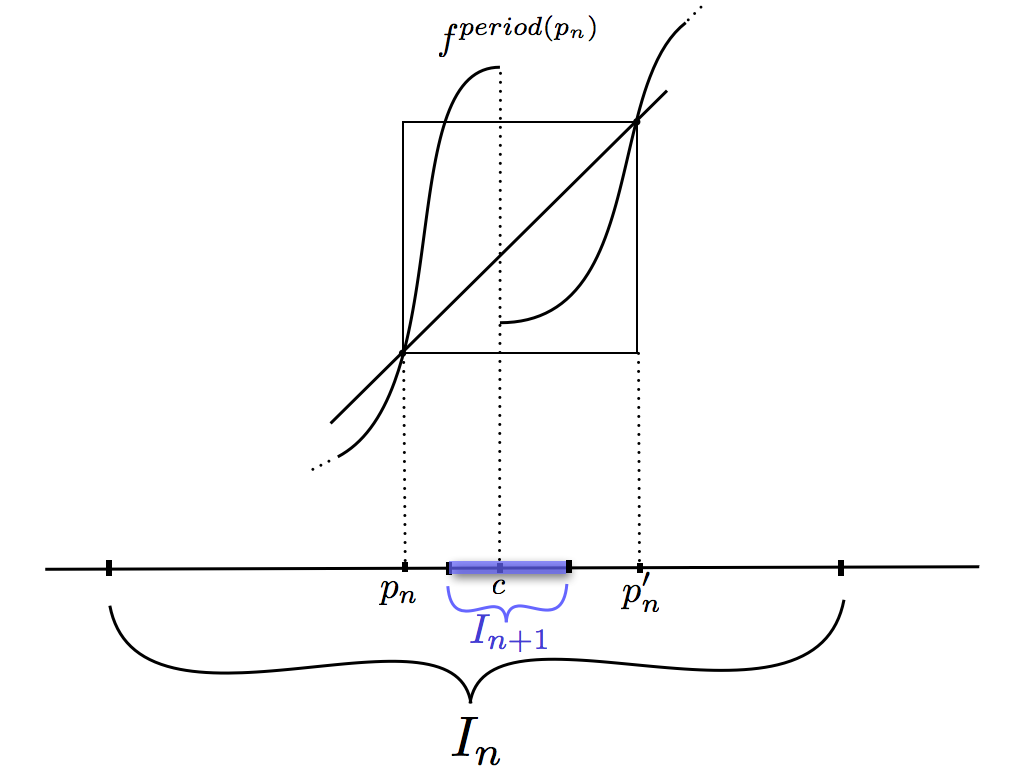}\\
Figure~\ref{FigFig.png}
\end{center}
\end{figure}
%%%%%%%%%%%%%%%%%%%%%%%%%%%%%%%%%%%%%%%%%%%%%%

\begin{Remark}\label{moduloretorno}
Let $F:(a,b)^{*}\to (a,b)$ be the first return map to $(a,b)$ by $f$. If $f^n(x)\in(a,b)$ for $x\in(a,b)$ then $\exists m$, $1 \le m \le n$, such that $F ^m (x)=f^n(x)$.
\end{Remark}

\begin{Claim}\label{cheganobordo}
Given $y \in \Omega_n$ and $\epsilon>0$, $\exists j$ such that $f^j((y-\epsilon,y+\epsilon))\cap\partial I_{n+1}\ne\emptyset$.
\end{Claim}

\dem[Proof of the claim]
As the set of points that don't visit $I_{n+1}$ has zero Lebesgue measure, $\exists j$, $j$ minimum such that $f^j((y-\epsilon,y+\epsilon))\cap I_{n+1} \ne \emptyset$. But $y \in \Omega_n$, then $\forall j$, $ f^j(y) \not\in \interior(I_{n+1})$ for otherwise it would be trapped into $K_{n+1}$ with any sufficiently small neighborhood of it. So, there is a point that is sent to $\partial I_{n+1}$. \cqd

\begin{Claim}\label{cobretudo}
$\exists$ sequence $\co^-(p_{n})\ni y_k\nearrow a_{n+1}$, $n_k\to \infty$ and $ \epsilon_k >0$ such that $y_k<y_k+\varepsilon_k<a_{n+1}$, $f^{n_k}|_{(y_k,y_k+\epsilon_k)}$ homeomorphism and $f^{n_k}(y_k,y_k+\epsilon_k)=(p_{n},p_{n}')$.
\end{Claim}

\dem[Proof of the claim]

Observe that we have already obtained in Lemma \ref{bordos} that as $I_{n}$ is a regular renormalization interval then $\alpha(p_{n})=\alpha(p_{n}') \supset \partial I_{n+1}$. Then, there is a sequence of pre-images of $p_n$, $y_n\in \co^-(p_{n})$, converging to $a_{n+1}$. Taking these $\delta_k>0$ small enough to have  $f^{j_k}|_{(y_k,y_k+\delta_k)}$ being a diffeomorphism with $j$ minimum such that $f^{j_k}(y_k)=p_{n}$. Let $s_k=\min\{j>0\,;\,(f^{\period{p_n}}|_{(p_n,c)}^{-1})^s(p_n')\in(p_n,f^{j_k}(y_k+\delta_k))\}$ and $y_k+\varepsilon_k\in(y_k,y_k+\delta_k)$ s.t. $f^{j_k}(y_k+\varepsilon_k)=(f^{\period(p_n)}|_{(p_n,c)}^{-1})^{s_k}(p_n')$. Thus, $f^{s_k+j_k}|_{(y_k,y_k+\varepsilon_k)}$ is a homeomorphism and $$f^{s_k+j_k}((y_k,y_k+\varepsilon_k))=(p_n,p_n').$$\cqd

\begin{claim}
Given $y \in \Omega_n \cap (p_n,p_n')$ and $\epsilon>0$, $\exists I \subset B_{\epsilon}(y)$ and $k\ge0$ such that $f^k|_I$ is a homeomorphism and $f^k(I)=(p_n,p_n')$.
\end{claim}

\dem[Proof of the claim]
By claims \ref{cheganobordo} and \ref{cobretudo} we can obtain this I of the statement. \cqd

By Remark \ref{moduloretorno} we have that given $y \in \Omega_n \cap (p_n,p_n')$ and $\epsilon > 0$, $\exists I \subset B_\epsilon(y)$ and $m$ such that $F_n^m(I)=(p_n,p_n')$

As $F_n((p_n,p_n'))=(p_n,p_n')$ it follows that given $y \in \Omega_n \cap (p_n,p_n')$ and $V$ a neighborhood of $y$, we will have that $\exists m$ such that $F_n^s(V)=(p_n,p_n')$, $\forall s \ge m$. That is, given an open set $V$ with $V\cap\Omega_n\ne\emptyset$, we have that $F_n^s(V \cap\Omega_n)=(p_n,p_n')\cap\Omega_n$, $\forall s \ge m$, for some $m \ge 0$.

\cqd

\newpage

\section{Appendix}

\begin{Lemma}\label{dicotomia} (Dichotomy of transitivity) 
If $f:U \to \XX$ is a continuous map  defined in an open and dense subset $U$ of  compact metric space $\XX$ then either $\nexists x \in U$ such that $\omega_{f}(x)=\XX$ or $\omega (x)=\XX$ for a residual set of $x \in \XX$.
\end{Lemma}

\dem
Suppose that $\co^{+}_{f}(p)$ is dense in $\XX$ for some $p\in \bigcap_{j\ge0}f^{-j}(U)$. Write $p_{\ell}=f^{\ell}(p)$. For each $\ell\in\NN$ there is some $k_{n,\ell}$ such that $\{p_{\ell},$ $\cdots,$ $f^{k_{n,\ell}}(p_{\ell})\}$ is $(1/2n)$-dense. As $f$ is continuous and $U$ open, there is some $r_{n,\ell}>0$ such that $f^{j}(B_{r_{n,\ell}}(p_{\ell}))\subset B_{1/2n}(f^{j}(p_{\ell}))$ $\forall\,0\le j\le k_{n,\ell}$. Thus, $\{y,\cdots,f^{k_{n,\ell}}(y)\}$ is $(1/n)$-dense $\forall\,y\in B_{r_{n,\ell}}(p_{\ell})$. Let $$\XX_{n}=\{x\in\XX\,;\,\co^{+}_{f}(x)\text{ is }(1/n)-\text{dense}\}.$$ Therefore $\bigcup_{\ell\in\NN}B_{r_{n,\ell}}(p_{\ell})\subset\XX_{n}$ is a open and dense set. Furthermore, $$\bigcap_{n\in\NN}\bigcup_{\ell\in\NN}B_{r_{n,\ell}}(p_{\ell})$$ is a residual set contained in $\bigcap_{n\in\NN}\XX_{n}=\{x\in\XX$ $;$ $\omega_{f}(x)=\XX\}$.
\cqd

\begin{Lemma}
\label{aeroporto}
Let $\XX$ be a compact metric space and $f:U\to\XX$ be a continuous map defined in a subset $U$. If $x\in\bigcap_{n\ge0}f^{-n}(U)$ and $x\in\omega_f(x)$ then either $\co_f^+(x)$ is a periodic orbit (in this case $\omega_f(x)=\co_f^+(x)$) or $\omega_{f}(x)$ is a perfect set.
\end{Lemma}

\dem
Suppose $\exists p \in \omega_{f}(x)$ an isolated point, say $B_{\varepsilon}(p)\cap\omega_{f}(x)=\{p\}$, with $\varepsilon>0$. As $x\in\omega_{f}(x)$ and $f$ is continuous on $\co^{+}_{f}(x)$, we have $\co^{+}_{f}(x)\subset\omega_{f}(x)$. Thus, $\co^{+}_{f}(x)\cap(B_{\varepsilon}(p)\setminus\{p\})=\emptyset$.
As $p \in \omega_{f}(x) \Rightarrow \exists$ sequence $n_{j}\nearrow \infty$ such that $f^{n_{j}}(x)\to p$.  Taking $j$ big enough we have $f^{n_{j}}(x)\in B_{\varepsilon}(p)$ then $f^{n_{j}}(x)=p$, $\forall j$ big and then $f^{n_{j+1}-n_{j}}(f^{n_{j}}(x))=p=f^{n_{j}}(x)$, that is, $f^{n_{j}}(x)$ is periodic. As $x \in \omega_{f}(x)=\omega(f^{n_{j}}(x))=\co^{+}(f^{n_{j}}(x))$, we have that $x$ is periodic.
\cqd

\begin{Corollary}
\label{CORaeroporto}
Let $f:[0,1]\setminus\{c\}\to[0,1]$ be a contracting Lorenz map. If $c_-\in\omega_f(c_-)$ then either $f$ has a super-attractor containing $c_-$ or $\omega_f(c_-)$ is a perfect set. Analogously,   If $c_+\in\omega_f(c_+)$ then either $f$ has a super-attractor containing $c_+$ or $\omega_f(c_+)$ is a perfect set.
\end{Corollary}

\dem
Suppose that $f$ does not have a super-attractor containing $c_-$. Thus, $v_1:=f(c_-)\notin\co_f^-(c)$. In this case, $\co_f^+(c_-)=\{c\}\cup\co_f^+(v_1)$ (recall the definition of $\co_f^+(c_-)$ in the beginning of Section~\ref{MainResults}). Note that $v_1\in\omega_f(v_1)$, because $c_-\in\omega_f(c_-)$. As $v_1$ can not be a periodic orbit and as $v_1\in\bigcup_{n\ge0}f^{-n}([0,1]\setminus\{c\})$, it follows from Lemma~\ref{aeroporto} that $\omega_f(v_1)$ is a perfect set. As $\omega_f(c_-)=\omega_f(v_1)$  (because $c\in\overline{\omega_f(v_1)\cap(0,c)}$), we finish the proof.\cqd

\begin{Corollary}
\label{CORaeroporto2}
Let $f:[0,1]\setminus\{c\}\to[0,1]$ be a contracting Lorenz map without periodic attractors. Suppose $\omega_f(c_-)\ni c\in\omega_{f}(c_{+})$. If $\overline{\co_{f}^{+}(p)\cap(0,c)}\ni c\in\overline{(c,1)\cap\co_{f}^{+}(p)}$, $p\in(0,1)\setminus\{c\}$, then $\omega_{f}(p)$ is a perfect set and $\overline{\omega_{f}(p)\cap(0,c)}\ni c\in\overline{(c,1)\cap\omega_{f}(p)}$.\end{Corollary}
\begin{proof}
It follows from Corollary~\ref{CORaeroporto} that $\omega_f(c_-)$ and $\omega_f(c_+)$ are perfect sets. Furthermore, $\overline{\omega_f(c_-)\cap(0,c)}\ni c\in\overline{(c,1)\cap\omega_f(c_+)}$.
If $\overline{\co_{f}^{+}(p)\cap(0,c)}\ni c\in\overline{(c,1)\cap\co_{f}^{+}(p)}$
then $\omega_f(p)\supset\omega_f(c_-)\cup\omega_f(c_+)$ and so,
$\overline{\omega_f(p)\cap(0,c)}\ni c\in\overline{(c,1)\cap\omega_f(p)}$.

Now suppose that $\omega_{f}(p)$ is not perfect. Thus, there is $q\in\omega_{f}(p)$ and $\delta>0$ such that $B_{\delta}(q)\cap\omega_{f}(p)=\{q\}$. Let $J=(a,b)$ be the connect component of $[0,1]\setminus\big(\omega_{f}(p)\setminus\{q\}\big)$ containing $q$. Note that $a,b\subset\big(\omega_{f}(x)\cup\{0,1\})$.

\end{proof}

\subsubsection{Topological entropy of the Solenoid Attractor}

To compute the topological entropy of the Solenoid attractor we will need to compactify the $[0,1]\setminus\{c\}$ in such way that the contracting Lorenz map $f$ can be extended as a continuous map. For this, we may assume that $f$ does not have a super-attractor (we want to study the $\infty$-renormalizable case). Thus, $c\notin\co_{f}^{+}(c)$. Let $\cc:=\co_{f}^{-}(c)\times\{-1,1\}$ and $$\XX_{f}:=\bigg(\big([0,1]\setminus\co_{f}^{-}(c)\big)\times\{0\}\bigg)\cup\bigg(\co_{f}^{-}(c)\times\{-1,1\}\bigg).$$ Let $\pi:\XX_{f}\to[0,1]$ be the projection in the first coordinate , i.e., $\pi((x,j))=x$.

Given $(x,j)\in\XX_{f}$, define $B^{f}_{r}((x,j))$, the ``ball'' of radius $r>0$ with center in $x$, by 
$$B^{f}_{r}((x,j))=\begin{cases}
\pi^{-1}((x-r,x))\cup\{(x,-1)\}&\text{ if }j=-1\\
\pi^{-1}((x-r,x+r))&\text{ if }j=0\\
\pi^{-1}((x,x+r))\cup\{(x,1)\}&\text{ if }j=1
\end{cases}
$$

Let $\widehat{f}:\XX_{f}\to\XX_{f}$ be the lift of $f$ to $\XX_{f}$ given by 
$$\widehat{f}((x,j))=\begin{cases}
(f(x),j)&\text{ if }x\ne c\\
(f(c_{-}),0)&\text{ if $x=c$ and $j=-1$}\\
(f(c_{+}),0)&\text{ if $x=c$ and $j=1$}
\end{cases}
$$

Of course $\XX_{f}=\bigcup_{r>0,x\in\XX_{f}}B^{f}_{r}(x)$. Moreover, if $y\in B^{f}_{r}(x)\cap B^{f}_{r'}(x')$ then $B^{f}_{\delta}(y)\subset B^{f}_{r}(x)\cap B^{f}_{r'}(x')$, for any sufficiently small $\delta>0$. Thus, there exists an unique topology on $\XX_{f}$ generated by the balls $\{B^{f}_{r}((x,j))\,;\,r>0$ and $(x,j)\in\XX_{f}\}$. With this topology $\XX_{f}$ is a $T_{1}$-space and it has a countable dense subset and it is compact (note that every sequence $(p_{n},j_{n})\in\XX_{f}$ has a convergent subsequence). So, $\XX_{f}$ is metrizable. For instance, if we write $\{q_{1},q_{2},q_{3},\cdots\}=\co_{f}^{-}(c)\times\{-1,1\}\cup\pi^{-1}(\QQ)$ and $\{r_{1},r_{2},r_{3},\cdots\}=(0,1)\cap\QQ$ then we can define the distance $$d((x,i),(y,j))=\sum_{n,m\in\NN}\frac{1}{2^{n+m}}\bigg|\chi_{{B^{f}_{r_{n}}(q_{m})}}((x,i))-\chi_{{B^{f}_{r_{n}}(q_{m})}}((y,j))\bigg|,$$ which is compatible with the topology. It is easy to check that $\pi$ and $\widehat{f}$ are continuous in this  topology. Furthermore, $\pi$ is a semi-conjugation between $\widehat{f}$ and $f$, that is, the diagram below commutes.
{\large
$$\begin{CD}
\XX_{f}\setminus\{(c,\pm 1)\}     @>\widehat{f}>>  \XX_{f}\setminus\{(c,\pm 1)\}\\
@VV\pi V        @VV\pi V\\
[0,1]\setminus\{c\}     @>f>>  [0,1]\setminus\{c\}
\end{CD}$$
}

\begin{Lemma}Let $f:\XX\to\XX$ and $g:\YY\to\YY$ be continuos maps defined in compact metrical spaces. Let $\Psi\in C^0(\Omega(f),\YY)$ be such that  $\#\Psi^{-1}(x)\le M$ $\forall\,x\in\YY$, for some $M\ge1$. If $g\circ\Psi=\Psi\circ f$ then $h(f)\le h(g)$.
\end{Lemma}
\begin{proof} Note that $\Lambda=\Omega(f)$ and $\Psi(\Lambda)$ are compact metrical spaces with theirs induced metrics, as $g(\Psi(\Lambda))=\Psi(\Lambda)$, we can consider the following commutative diagram.
$$\begin{CD}
\Lambda     @>f|_{\Lambda}>>  \Lambda\\
@VV\Psi V        @VV\Psi V\\
\Psi(\Lambda)     @>g|_{\Psi(\Lambda)}>> \Psi(\Lambda)
\end{CD}$$
Of course $g|_{\Psi(\Lambda)}$ is a factor of $f|_{\Lambda}$. As $\#\Psi^{-1}(x)\le M$ $\forall\,x$, 
the fiber entropy is smaller or equal to $\log M$, i.e., $H(\cu|\Psi^{-1}(y))=\log\min\big\{\#\cu'\,;\,\cu'\subset\cu\text{ and }\Psi^{-1}(y)\subset\bigcup_{P\in\cu'}P\big\}\le\log M$, where $\cu$ is cover of $\Lambda$ by open sets (see, for instance,  pages 167, 170 (Section 6.3) and 179 (Section 6.7) of \cite{TD} for more details). As a consequence, the factor conditional entropy is zero, that is, $$h\big(f|_{\Lambda}\big|g|_{\Psi(\Lambda)}\big)=\sup_{\cu}\inf_n\sup_{y\in\Psi(\Lambda)}\frac{1}{n}H\bigg(\bigvee_{j=0}^{n-1}f|_{\Lambda}^{-j}\cu\bigg|\Psi^{-1}(y)\bigg)=0$$ (see Lemma 6.8.2, page 182 of \cite{TD}). Thus, it follows from Ledrappier~\cite{Le1979} that $h(f)=h(f|_{\Lambda})=h(g|_{\Psi(\Lambda)})\le h(g)$ (see also , see also Corollary~6.4.15 in pp. 172 of \cite{TD}).

\end{proof}

\begin{Corollary}\label{Corollry09g87b4}
$h(\widehat{f}\,)\le\log 2$ for all contracting Lorenz map $f$.
\end{Corollary}
\begin{proof}
Let $i:\Omega(\widehat{f}\,)\to\sum_{2}:=\{0,1\}^{\NN}$ be the itinerary map. That is, $$i((x,j))(n)=\begin{cases}
0&\text{ if }(f^{n}(x)<c)\text{ or }(f^{n}(x)=c\text{ and }j=-1)\\
1&\text{ if }(f^{n}(x)>c)\text{ or }(f^{n}(x)=c\text{ and }j=1)
\end{cases}$$
 It is easy to check that $\#i^{-1}(y)\le 2$ $\forall\,y\in\sum_{2}$. As $\widehat{f}\circ i=i\circ \sigma$, where $\sigma:\sum_{2}\to\sum_{2}$ is the shift, it follows for the Lemma above the $h(\widehat{f}\,)\le h(\sigma)=\log 2$.
\end{proof}

\begin{Lemma}[The Solenoid attractor has zero entropy]\label{LemmaZeroEntropy}
Let $f:[0,1]\setminus\{c\}\to[0,1]$ be a $C^{2}$ non-flat contracting Lorenz map without inessential periodic attractors or weak repellers. If $f$ is $\infty$-renormalizable then $h(f|_{\Lambda})=0$, where $\Lambda$ is the solenoid attractor of $f$.
\end{Lemma}
\begin{proof}
Let suppose that $f$ is $\infty$-renormalizable and $I_{1}\supset I_{2}\supset I_{3}\supset\cdots$ be the nested sequence of renormalization intervals of $f$.
Observe that $\widehat{I}_{n}:=\pi^{-1}I_{n}$ is a renormalization interval to $\widehat{f}$, i.e., the first return map to $\widehat{I}_{n}$ by $\widehat{f}$ is conjugated with some $\widehat{g}$ where $g$ is a contracting Lorenz map.

Let $f_{n}$ be the first return map to $\overline{I_{n}}$ by $f$ and $\widehat{f}_{n}$ be the first return map to $\overline{\widehat{I}_{n}}$ by $\widehat{f}$. As $\XX_{f}$ is compact metrical space and $\widehat{f}$ continuous, we can apply the Variational Principle for the entropy to $\widehat{f}$ and also to $\widehat{f}_{n}$ $\forall\,n$. That is, $h(\widehat{f}_{n})=\sup\{h_{\mu}(\widehat{f}_{n})$ $;$ $\mu$ is a $\widehat{f}_{n}$-invariant probability$\}$. So, using Corollary~\ref{Corollry09g87b4} we get 
\begin{equation}\label{Eqhat}
h_{\mu}(\widehat{f}_{n})\le h(\widehat{f}_{n})\le\log2,
\end{equation}
for all $\mu$ $\widehat{f}_{n}$-invariant probability and all $n\ge1$.

Furthermore, if $\mu$ is $\widehat{f}$ an ergodic invariant probability with $\mu(\widehat{I}_{n})>0$ then $\nu:=\frac{1}{\mu(\widehat{I}_{n})}\mu|_{\widehat{I}_{n}}$ is a $f_{n}$-invariant probability and $h_{\mu}(\widehat{f})=\frac{h_{\nu}(\widehat{f}_{n})}{\int R_{n} d\nu}$, where $R_{n}$ is the first return time to $\widehat{I}_{n}$. As $\int R_{n} d\nu$ $\ge$ $\min\{\period(p)$ $;$ $p\in\partial \widehat{I}_{n}\}=:\gamma_{n}$ $\to$ $\infty$ and as $h_{\nu}(\widehat{f}_{n})\le\log2$, we get (applying the  Variational Principle to $\widehat{f}|_{K_{\widehat{I}_{n}}}$) that
$$h\bigg(\widehat{f}|_{\overline{K_{\widehat{I}_{n}}}}\bigg)\le\frac{\log2}{\gamma_{n}},$$ where $K_{\widehat{I}_{n}}$ is the nice trapping region generated by $\widehat{I}_{n}$. Indeed, $K_{\widehat{I}_{n}}=\pi^{-1}(K_{I_{n}})$.
Thus, as $h(f|_{\Lambda})\le h(\widehat{f}|_{\pi^{-1}(\Lambda)})\le h(\widehat{f}|_{\overline{K_{\widehat{I}_{n}}}})\le\frac{\log2}{\gamma_{n}}\,\,\,\forall\,n\ge1,$
we get that the topological entropy of $f$ restricted to the Solenoid attractor is zero. That is, $h(f|_{\Lambda})=0.$
\end{proof}

\subsubsection{Misiurewicz contracting Lorenz maps}
\label{SecMisContLM}

A contracting Lorenz map $f:[0,1]\setminus\{c\}\to[0,1]$ is called {\em Misiurewicz} if the critical values do not belong to the basin of a periodic attractor and ``the critical point is not recurrent''. Precisely: if neither $f(c_{-})$ nor $f(c_{+})$ is contained in the basin of a periodic attractor and there is some $\delta>0$ such that $\big(\co_f^+(c_{-})\cup\co_f^+(c_{+})\big)\cap(c-\delta,c+\delta)=\emptyset$.

By definition, a Misiurewicz contracting Lorenz map can not have a super-attractor, because this implies that at least one of the critical values belongs to the periodic attractor (in particular, to its basin). Furthermore, by Singer Theorem (see Theorem~\ref{ThSinger} and Corollary~\ref{CoSinger}), if $f$ has negative Schwarzian derivative then the existence of a periodic orbit implies that one of the critical values of $f$ belongs to its basin. Thus, if $f$ is a Misiurewicz contracting Lorenz map with negative Schwarzian derivative then $f$ does not admit periodic attractors.

Note that if $I=(a,b)$ is a renormalization interval to $f$ then $f^{\period(a)}(c_{-})\in(a,b)\ni f^{\period(b)}(c_{+})$, that is, $\co_{f}^{+}(c_{-})\cap I\ne\emptyset\ne I\cap \co_{f}^{+}(c_{+})$. Thus, if $f$ is $\infty$-renormalizable then $c$ is accumulated by $\co_{f}^{+}(c_{-})$ and $\co_{f}^{+}(c_{+})$. This implies that $f$ cannot be Misiurewicz. That is, a Misiurewicz map is only finitely many times renormalizable.

Finally, a Misiurewicz contracting Lorenz map cannot be a Cherry-like map. To check this, let $\delta>0$ be as in the definition of Misiurewicz map and consider a sequence $c-\delta<x_1<x_2<x_3<\cdots<x_n\uparrow c$. If $f$ is Cherry-like, we may suppose that $c\in\omega_f(x_n)$ $\forall\,n$. In this case, let $s_n$ be such that  $|f^{s_n}(x_n)-c|<|x_n-c|$ $\forall\,n$. Let $I_n\ni x_n$ be the maximal interval contained in $[x_n,c)$ such that $f^{s_n}|_{I_n}$ is a homeomorphism. It is easy to see that $f^{s_n}(I_n)\supset[x_n,c)$ $\forall\,n$. Thus for each $n\ge1$ there is a periodic point $p_n\in[x_n,c)$. That is, $c\in\overline{Per(f)}$, but this is impossible for a Cherry-like map. So, we get a contradiction.

A map $f:V\subset\RR\to\RR$ is called {\em ergodic} with respect to the Lebesgue measure if either $\leb(U)=0$ or $\leb(V\setminus U)=0$ for every $f$-invariant set $U\subset V$ ($f$-invariant means $f^{-1}(U)=U$). 

\begin{Lemma}[Ergodicity of Misiurewicz contracting Lorenz maps]
\label{LemmaErgMisCLM}
If a $C^3$ contracting Lorenz map $f:[0,1]\setminus\{c\}\to[0,1]$ with $Sf<0$ is Misiurewicz then $f$ is ergodic with respect to Lebesgue measure. In particular, $f$ does not admit wandering interval.
\end{Lemma}
\begin{proof}
Let $\delta>0$ be such that $\big(\co_f^+(c_{-})\cup\co_f^+(c_{+})\big)\cap B_{\delta}(c)=\emptyset$, where $B_r(c)=(c-r,c+r)$ denotes the ball of radius $r$ centered at $c$. Let $B_{\delta/2}(c)\supset J_1\supset J_2\supset J_3\supset\cdots$ be a nested sequence of nice intervals with $\bigcap_n J_n=\{c\}$. Let us verify that such sequence exists, i.e., verify the existence of arbitrarily small nice intervals. As  $\overline{Per(f)\cap[0,c)}\ni c\in\overline{(c,1]\cap Per(f)}$ (Lemma~\ref{AcumulacaoDePer}), given $\varepsilon>0$ one can choose $p,q\in Per(f)$ with $c-\varepsilon<p<c<q<c+\varepsilon$ and consider $J_{\varepsilon}$ as the connected component of $[0,1]\setminus\big(\co_f^+(p)\cup\co_f^+(q)\big)$. Note that $\emptyset\ne J_{\varepsilon}\subset B_{\varepsilon}(c)$ is a nice interval.

For each $n$, let $C_{J_n}$ be the gap of $\Lambda_{J_n}$, that is, the collection of connected components of $[0,1]\setminus\Lambda_{J_n}$ (see Definition~\ref{Def765091} and \ref{Def535353}). Given $I\in C_{J_n}$ let $\widehat{I}$ be the maximal open interval containing $I$ such that $f^{\theta(I)}|_{\widehat{I}}$ is a homeomorphism (see Proposition~\ref{Prop765091} and Section~\ref{SecCyBrEx}). As $\big(\co_f^+(c_{-})\cup\co_f^+(c_{+})\big)\cap B_{\delta}(c)=\emptyset$, we get $f^{\theta(I)}\big(\widehat{I}\,\big)\supset B_{\delta}(c)$ $\forall\,I\in C_{J_n}$ $\forall\,n\ge1$.

By Koebe Lemma, there is a constant $K$  such that
\begin{equation}\label{eq99}
\frac{1}{K}\le\frac{Df^{\theta(I)}(x)}{Df^{\theta(I)}(y)}\le K\,\,\forall\,x,y\in (f^{\theta(I)}|_{\widehat{I}})^{-1}(B_{\delta/2}(c)),\,\,\forall\,I\in C_{J_n}\text{ and }\forall\,n\ge1.
\end{equation}
In particular, because $I\subset(f^{\theta(I)}|_{\widehat{I}})^{-1}(B_{\delta/2}(c))$,
\begin{equation}\label{eq999}
\frac{1}{K}\le\frac{Df^{\theta(I)}(x)}{Df^{\theta(I)}(y)}\le K\,\,\forall\,x,y\in I,\,\,\forall\,I\in C_{J_n}\text{ and }\forall\,n\ge1.
\end{equation}
We emphasize  that $K$ depends only on the relative space between $B_{\delta}(c)$ and $B_{\delta/2}(c)$, it does not depend on $n$ or $I\in C_{J_n}$.

\begin{Claim}\label{ClaimMis}$\lim_{n\to\infty}\sup\{|I|\,;\,I\in C_{J_n}\}=0$
\end{Claim}
\begin{proof}[Proof of the claim]
It follows from (\ref{eq99}) that $\frac{|I|}{|(f^{\theta(I)}|_{\widehat{I}})^{-1}(B_{\delta/2}(c))|}\le K\frac{|J_n|}{B_{\delta}(c)}$ $\forall\,I\in C_{I_n}$. Thus, $$|I|\le K \big|(f^{\theta(I)}|_{\widehat{I}})^{-1}(B_{\delta/2}(c))\big|\frac{|J_n|}{B_{\delta}(c)}\le  K\frac{|J_n|}{B_{\delta}(c)}.$$
So, $\sup\{|I|\,;\,I\in C_{I_n}\}\le K\frac{|J_n|}{B_{\delta}(c)}\to0$, when $n\to\infty$.
\end{proof}
Now, suppose that there exist a $f$-invariant set $U\subset[0,1]\setminus\{c\}$ such that $0<\leb(U)<1$. Let $p\in U$ and $q\in U^{\complement}$ be Lebesgue density points for $U$ and $U^{\complement}$ respectively, where $ U^{\complement}:=([0,1]\setminus\{c\})\setminus U$. As we are working in $\RR$, this means that given any sequences $A_j,A'_j$ of intervals with $A_j\ni p$, $A'_j\ni q$ and $|A_j|,|A'_j|\to0$ we have
\begin{equation}\label{eq4411}\lim_{j\to\infty}\frac{\leb(U\cap A_j)}{\leb(A_j)}=1=\lim_{j\to\infty}\frac{\leb(U^{\complement}\cap A'_j)}{\leb(A'_j)}.
\end{equation}

As Lebesgue almost every point of $U$ and $U^{\complement}$ are density points, we may assume that $p,q\notin\bigcup_{n}\Lambda_{J_n}$ (indeed, $\leb(\bigcup_{n}\Lambda_{J_n})=0$ by Lemma~\ref{Remark98671oxe}). That is, for every $n\in\NN$, there are $I_n(p),I_n(q)\in C_{J_n}$ such that $p\in I_n(p)$ and $q\in I_n(q)$.

By (\ref{eq4411}) and Claim~\ref{ClaimMis}, we can choose $n$ big enough so that $\frac{\leb(U\cap I_n(p))}{\leb(I_n(p))}$ and $\frac{\leb(U^{\complement}\cap I_n(q))}{\leb(I_n(q))}>\frac{2}{3 K}$. Thus, $\frac{\leb(U^{\complement}\cap I_n(p))}{\leb(I_n(p))}$ and $\frac{\leb(U\cap I_n(q))}{\leb(I_n(q))}<\frac{1}{3 K}$. From (\ref{eq999}) it follows that
$$\frac{\leb(U^{\complement}\cap J_n)}{\leb(J_n)}\le K\frac{\leb(U^{\complement}\cap I_n(p))}{\leb(I_n(p))}<\frac{1}{3}>K\frac{\leb(U\cap I_n(p))}{\leb(I_n(q))}\ge\frac{\leb(U\cap J_n)}{\leb(J_n)}.$$
But this is an absurd, as $(U\cup U^{\complement})\cap J_n=J_n$. So $f$ does not admit an invariant set with Lebesgue measure different from $0$ or $1$. That is, $f$ is ergodic with respect to Lebesgue measure.

Note that if $W=(a,b)$ is a wandering interval then $U=\bigcup_{n\in\ZZ}f^n((a,\frac{1}{2}(a+b))$ is a $f$-invariant set. Moreover, as $U\cap((\frac{1}{2}(a+b),b))=\emptyset$ (because $W$ is wandering interval), we have $0<\leb(U)<1$. An absurd, as $f$ is ergodic. Thus, $f$ does not admit wandering intervals.
\end{proof}

\subsubsection{Embedding the dynamics of the symmetric unimodal maps into the dynamics of  the contracting Lorenz maps}\label{SecEmbeding}

An unimodal map $f:[0,1]\to[0,1]$ with critical point $c$ is called symmetric if $f(x)=f(1-x)$. If $f$ is symmetric then $f'(x)=-f'(1-x)$ and so, $f'(1/2)=-f'(1/2)$. Thus $f'(1/2)=0$, i.e., $c=1/2$.

Given a symmetric unimodal map $f:[0,1]\to[0,1]$ we can associate it to the contracting Lorenz map $L_f:[0,1]\setminus\{\frac{1}{2}\}\to[0,1]$ given by
$$L_f(x)=
\begin{cases}
f(x) & \text{ if }x<\frac{1}{2}\\
1-f(x)& \text{ if }x>\frac{1}{2}
\end{cases}.
$$
Note that $$f\circ L_f(x)=
\begin{cases}
f(f(x)) & \text{ if }x<\frac{1}{2}\\
f(1-f(x))& \text{ if }x>\frac{1}{2}
\end{cases}\,\,=\,f(f(x)),$$  as $f(1-f(x))=f((fx))$ by the symmetry of $f$.  That is, the diagram below commutes.
$$\begin{CD}
[0,1]\setminus\{\frac{1}{2}\}     @>L_f>>  [0,1]\setminus\{\frac{1}{2}\}\\
@VV f V        @VV f V\\
[0,1]     @>f>>  [0,1]
\end{CD}$$

Thus, $f\circ L_f^n=f^n\circ f=f\circ f^n$ $\forall\,n\ge1$. As $\big($ $f(x)=f(y)$ $\big)$ $\iff$ $\big($ $x=y$ or $1-y$ $\big)$, it follows from $f(L_f^n(x))=f(f^n(x))$ that $L_f^n(x)=f^n(x)$ or $1-f^n(x)$. In particular,
\begin{equation}\label{EqLoUnD}|(L_f^n)'(x)|=|(f^n)'(x)|\,\,\forall\,x\,\forall\,n.
\end{equation}

Furthermore, as $(L_f)'(x)>0$, we get that $(L_f^n)'(x)=f^n(x)$ for a given $x\in[0,1]\setminus\{1/2\}$ if and only if $(f^n)'(x)>0$. That is,
\begin{equation}\label{EqLoUn}L_f^n(x)=
\begin{cases}
f^n(x) & \text{if }(f^n)'(x)>0\\
1-f^n(x) & \text{if }(f^n)'(x)<0\\
\end{cases}.
\end{equation}
The formulas (\ref{EqLoUnD}) and (\ref{EqLoUn}) above show that most of the features of the orbits of $L_f$ and $f$ are the same. In particular, $\co_f(p)$ and $\co_{L_f}(p)$ have the same Lyapunov exponent and $\co_f(p)$ is finite iff $\co_{L_f}(p)$ is finite. Furthermore, it is not difficult to check that the type of topological and metrical attractors of $f$ and $L_f$ are always the same.

\subsubsection{The attractor for the Cherry-like maps}\label{SecAtCherry-like}

\begin{Lemma}
Let $f:[0,1]\setminus\{c\}\to[0,1]$ be a contracting Lorenz map without super-attractors. If $c\in\omega_f(x)$ $\forall\,x\in(0,1)$ then there exists a compact set $\Lambda\subset(0,1)$ such that $\omega_f(x)=\Lambda$ $\forall\,x\in(0,1)$. In particular, $\Lambda$ is a minimal set.
\end{Lemma}
\begin{proof}
As $f$ does not have super-attractor and $c\in\omega_f(x)$ $\forall\,x\in(0,1)$, we get $$Per(f)=\{0,1\}.$$
Note also that $f([0,c))\ni c\in f((c,1])$, because $c\in\omega_f(x)$ $\forall\,x\in(0,1)$.
Taking in Lemma~\ref{Lemma545g55} $(a,b)=(0,1)$, we conclude that $\co_f^+(x)\cap(0,c)\ne\emptyset\ne(c,1)\cap\co_f^+(x)$ $\forall\,x\in(0,1)$.
So, we get
\begin{equation}\label{EqCL}
\overline{\co_f^+(x)\cap(0,c)}\ni c\in \overline{(c,1)\cap\co_f^+(x)}\,\;\forall\,x\in(0,1).
\end{equation}
As a consequence,
\begin{equation}\label{Eq323130}\omega_f(x)\supset\omega_f(c_-)\cup\omega_f(c_+)\,\;\forall\,x\in(0,1).
\end{equation}
In particular, $$c_-\in\omega_f(c_-)\text{ and }c_+\in\omega_f(c_+).$$

Thus, it follows from Corollary~\ref{CORaeroporto2} that
\begin{equation}\label{EqUU21}
\overline{\omega_f(x)\cap(0,c)}\ni c\in\overline{(c,1)\cap\omega_f(x)}\,\,\;\forall\,x\in(0,1).
\end{equation}

Now we will prove that $\omega_f(p)=\omega_f(q)$ $\forall\,p,q\in(0,1)$. If this is not true then there exist $p,q\in(0,1)$ such that $\omega_f(p)\setminus\omega_f(q)\ne\emptyset$. Let $y\in \omega_f(p)\setminus\omega_f(q)$. Set $[\alpha,\beta]=\big[\min\omega_f(p),\max\omega_f(p)\big]$ (indeed, $[\alpha,\beta]=[f(c_{+}),f(c_{-})]$).  It is easy to see that $f([\alpha,\beta])=[\alpha,\beta]$. As $c\in (\alpha,\beta)$ (by (\ref{EqCL})) and as $c\in\omega_f(x)$ $\forall\,x$, we get $\omega_f(x)\subset[\alpha,\beta]$ $\forall\,x\in(0,1)$. As a consequence, $y\in(\alpha,\beta)$.

Let $J=(a,b)$ be the connected component of $[0,1]\setminus\omega_f(p)$ containing $q$. As $y\in(\alpha,\beta)$, we get $a,b\in\omega_f(p)$. As $y\in \omega_f(q)\cap J$, one can find $0\le n_1<n_2$ such that $f^{n_1}(q),f^{n_2}(q)\in(a,b)$. We may suppose that $f^{n_1}(q)<f^{n_2}(q)$ (the case $f^{n_1}(q)>f^{n_2}(q)$ is analogous).

Let $T:=(t,f^{n_1}(q)]$ be the maximal interval contained in $(a,f^{n_1}(q)]$ such that $f^{n_2-n_1}|_T$ is a homeomorphism and that $f^{n_2-n_1}(T)\subset (a,f^{n_2}(q)]$.

\begin{Claim}
$f^{n_2-n_1}(T)=(a,f^{n_2}(q)]$
\end{Claim}
\begin{proof}[Proof of the claim]
If not, there are two possibles cases: (1) $f^{s}(t)=c$ for some $0\le s<n_2-n_1$ or (2) $t=a$ and $a<f^{n_2-n_1}(a)<f^{n_{2}}(q)$. As $a<f^{n_{2}-n_{1}}(a)<f^{n_{2}}(q)$ will implies that $\omega_{f}(p)\cap J\ne\emptyset$, and this contradicts the fact that $J\subset[0,1]\setminus\omega_{f}(p)$, we have only to analyze the first case.

Thus $f^s(T)\cap\omega_f(p)=(c,f^s(f^{n_2-n1})(q))\cap\omega_f(p)\ne\emptyset$ (because of (\ref{EqUU21})). But this implies that  $J\cap\omega_f(x)\supset f^{n_2-n_1}(T)\cap\omega_f(x)\supset f^{n_2-n_1-s}(f^s(T)\cap\omega_f(x))\ne\emptyset$. An absurd, as $J\subset[0,1]\setminus\omega_f(x)$.\end{proof}

It follows from the claim above that $f^{n_2-n_1}(T)=(a,f^{n_2}(q)]\supset T$. This implies that $f$ has a periodic point in $\overline{T}$ (because $f^{n_2-n_1}|_{T}$ is a homeomorphism). But this is a contradiction with the fact that $Per(f)\cap(0,1)=\emptyset$.
\end{proof}

\subsubsection*{Analyzing the Spectral Decomposition to a non renormalizable map}
\label{SecAnSpDeNonReMap}

Let $f:[0,1]\setminus\{c\}\to[0,1]$ be a non renormalizable $C^3$ contracting Lorenz map with negative Schwarzian derivative.
\begin{enumerate}
\item (Figure~\ref{SemRerNor011.png}) If $0=f(c_+)$ and $f(c_-)=1$ then $f$ is transitive. In this case $n_f=0$ and $\Omega_0=[0,1]$. this is the Misiurewicz case (see Section~\ref{SecMisContLM} of the Appendix). In this case $f$ is ergodic with respect to Lebesgue measure and it does not admit wandering interval (Lemma~\ref{LemmaErgMisCLM} of the Appendix). Thus, as $f$ is neither Cherry-like nor $\infty$-renormalizable and as $f$ does not have periodic attractors (see details in Section~\ref{SecMisContLM} of the Apendix), we can apply Theorem~\ref{atratortopologico} and conclude that $f$ $[0,1]=[f(c_+,f(c_-)]$ is the attractor of $f$. In particular, $[0,1]$ is transitive.
%%%%%%%%%%%%%%%%%%%%%%%%%%%%%%%%%%%%%%%%%%%%
\begin{figure}[H]
\begin{center}\label{SemRerNor011.png}
\includegraphics[scale=.2]{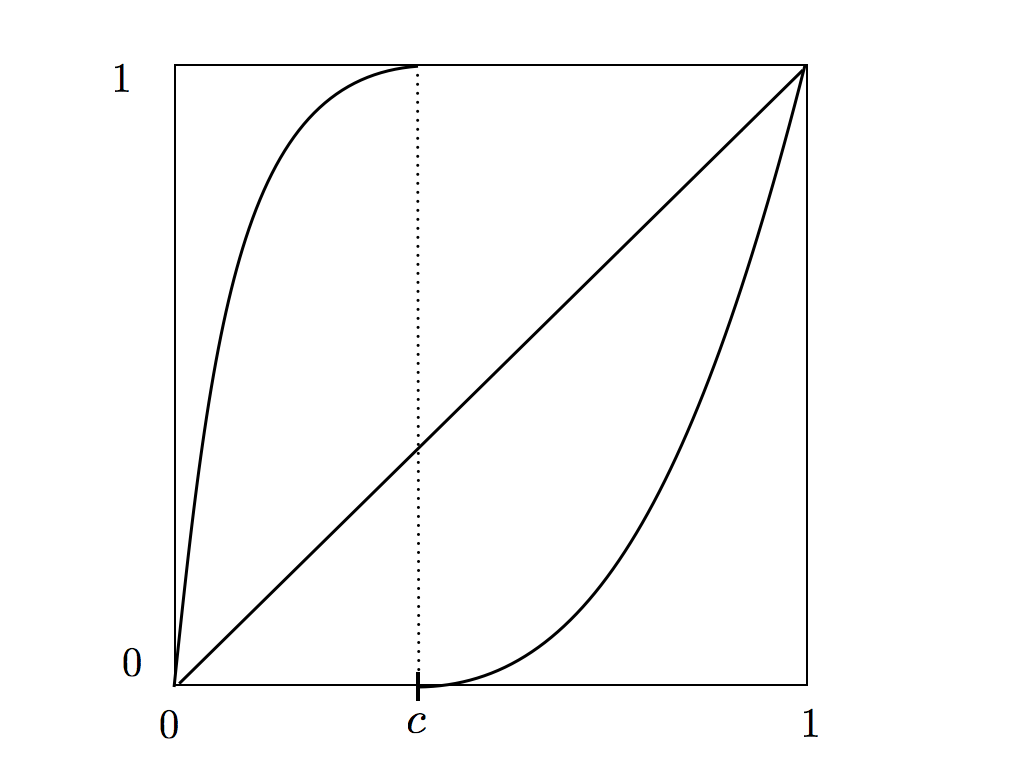}\\
Figure~\ref{SemRerNor011.png}
\end{center}
\end{figure}
%%%%%%%%%%%%%%%%%%%%%%%%%%%%%%%%%%%%%%%%%%%%%%
\item  (Figure~\ref{SemRerNor0111.png}) If $0<f(c_+)\le c\le f(c_-)<1$ then $n_f=1$, $\Omega_0=\{0,1\}$ and $\Omega_1$ is transitive. In this case, $\Omega_1$ is the topological (and metrical) attractor and it can be $\{c\}$ ($\{c\}$ is a super-attractor) or a Cherry attractor or $[f(c_+),f(c_-)]$ (a chaotic attractor) or a chaotic Cantor set (with the existence of wandering intervals).
%%%%%%%%%%%%%%%%%%%%%%%%%%%%%%%%%%%%%%%%%%%%
\begin{figure}[H]
\begin{center}\label{SemRerNor0111.png}
\includegraphics[scale=.35]{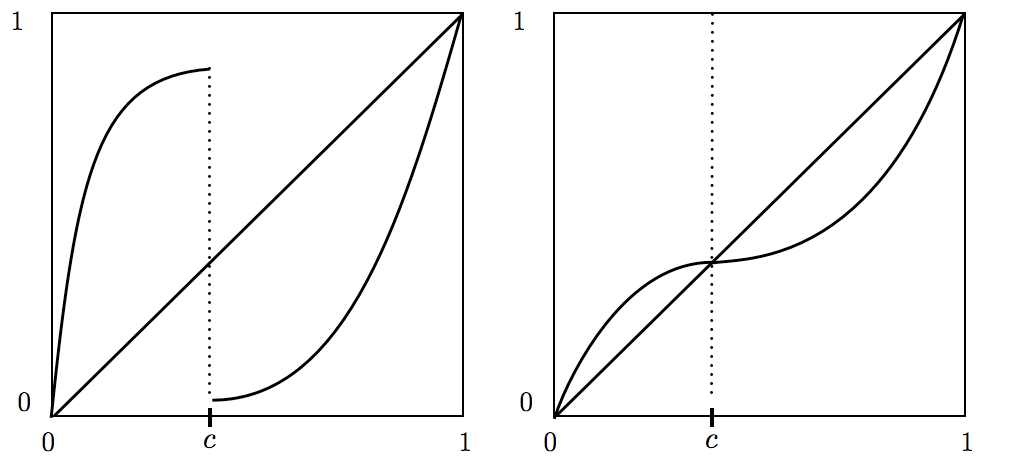}\\
Figure~\ref{SemRerNor0111.png}
\end{center}
\end{figure}
%%%%%%%%%%%%%%%%%%%%%%%%%%%%%%%%%%%%%%%%%%%%%%
\item  (Figure~\ref{SemRerNor015.png}) If $0=f(c_+)$ and $0$ is an attracting fixed point or if $f(c_-)=1$ and $1$ is an attracting fixed point then $n_f=0$ and $\Omega_0=\{0,1\}$.
%%%%%%%%%%%%%%%%%%%%%%%%%%%%%%%%%%%%%%%%%%%%
\begin{figure}[H]
\begin{center}\label{SemRerNor015.png}
\includegraphics[scale=.30]{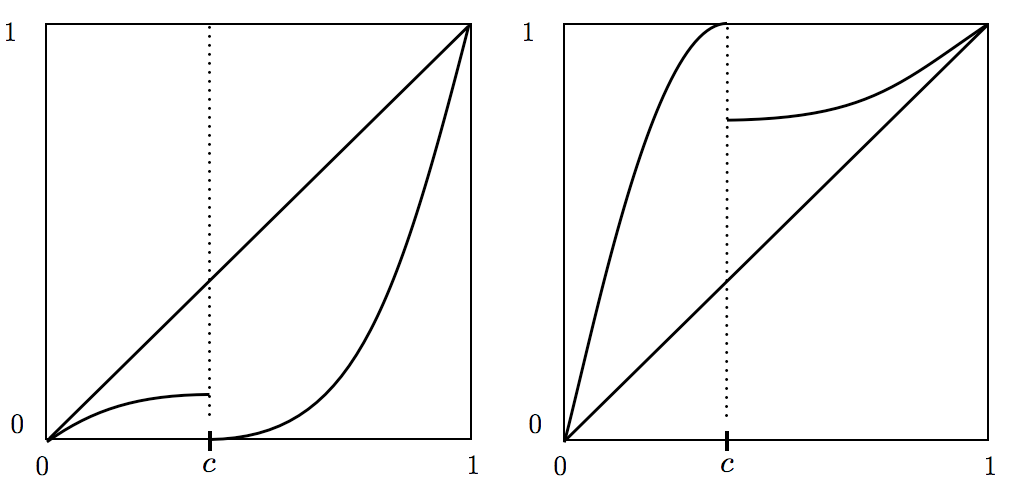}\\
Figure~\ref{SemRerNor015.png}
\end{center}
\end{figure}
%%%%%%%%%%%%%%%%%%%%%%%%%%%%%%%%%%%%%%%%%%%%%%
\item  (Figure~\ref{SemRerNor0133.png}) If $0=f(c_+)<c<f(c_-)<1$ and $f$ does not have a periodic attractor then $n_f=1$, $\Omega_0=\{1\}$ and $\Omega_1$ can be $[0,f(c_-)]$ (chaotic attractor) or a chaotic Cantor set (with the existence of wandering intervals).\\
(Figure~\ref{SemRerNor0133.png}) If $0<f(c_+)<c<f(c_-)=1$ and $f$ does not have a periodic attractor then $n_f=1$, $\Omega_0=\{0\}$ and $\Omega_1$ can be $[f(c_+),1]$ (chaotic attractor) or a chaotic Cantor set (with the existence of wandering intervals).
%%%%%%%%%%%%%%%%%%%%%%%%%%%%%%%%%%%%%%%%%%%%
\begin{figure}[H]
\begin{center}\label{SemRerNor0133.png}
\includegraphics[scale=.30]{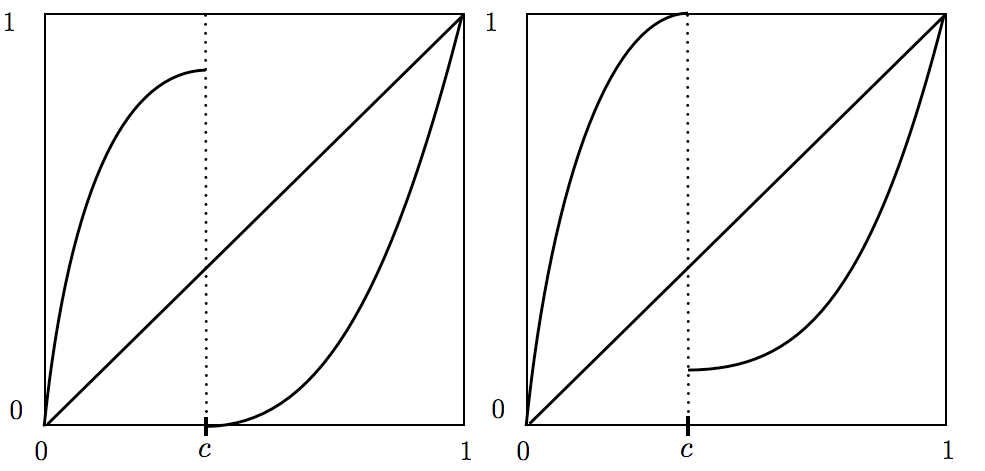}\\
Figure~\ref{SemRerNor0133.png}
\end{center}
\end{figure}
%%%%%%%%%%%%%%%%%%%%%%%%%%%%%%%%%%%%%%%%%%%%%%
\item (Figure~\ref{SemRerNor080.png}) If $0=f(c_+)<c<f(c_-)<1$ and $f$ has an attracting hyperbolic periodic orbit then $n_f=2$, $\Omega_0=\{1\}$ and $\Omega_1$ is a transitive expanding Cantor set and $\Omega_2$ is the attracting hyperbolic periodic orbit.\\
(Analogous to Figure~\ref{SemRerNor080.png}) If $0<f(c_+)<c<f(c_-)=1$ and $f$ has an attracting hyperbolic periodic orbit then $n_f=2$, $\Omega_0=\{0\}$ and $\Omega_1$ is a transitive expanding cantor set and $\Omega_2$ is the attracting hyperbolic periodic orbit.
\begin{Remark}\label{ReM66}
These cases are studied  in Section~\ref{SecDegRenInt}. In these cases, although there are proper renormalization intervals, there is a degenerate renormalization interval $I$ in the form $(p,c)$, $p\in Per(f)$, and it is the local basin of the periodic attractor. $\Omega_{1}=\Omega(f)\cap((0,1)\setminus I)$ $=$ $\{x\in(0,1)\,;\,\co_{f}^{+}(x)\cap I=\emptyset\}$.
\end{Remark}
%%%%%%%%%%%%%%%%%%%%%%%%%%%%%%%%%%%%%%%%%%%%
\begin{figure}[H]
\begin{center}\label{SemRerNor080.png}
\includegraphics[scale=.41]{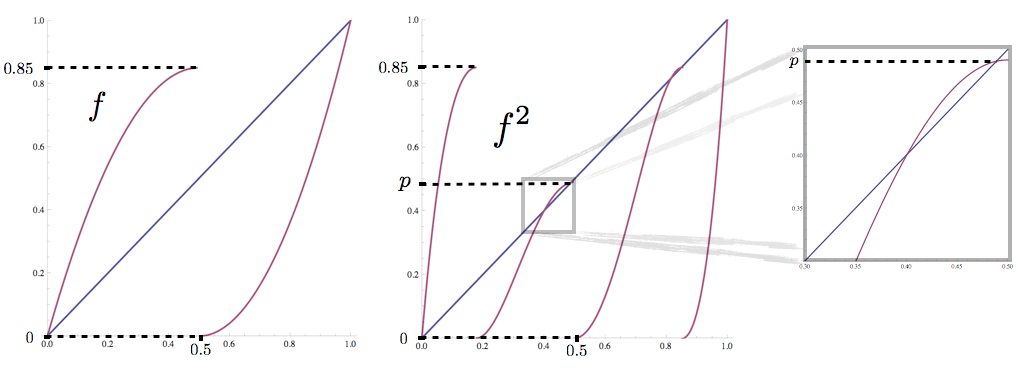}\\
Figure~\ref{SemRerNor080.png}
\end{center}
\end{figure}
%%%%%%%%%%%%%%%%%%%%%%%%%%%%%%%%%%%%%%%%%%%%%%
\item (Figure~\ref{SemRerNor080.png}) If $0=f(c_+)<c<f(c_-)<1$ and $f$ has an non-hyperbolic attracting periodic orbit (a saddle-node) or a super-attractor then $n_f=1$, $\Omega_0=\{1\}$ and $\Omega_1$ is a transitive non-hyperbolic Cantor set.\\
(Analoguous to Figure~\ref{SemRerNor080.png}) If $0<f(c_+)<c<f(c_-)=1$ and $f$ has an non-hyperbolic attracting periodic orbit (a saddle-node) or a super-attractor then $n_f=1$, $\Omega_0=\{0\}$ and $\Omega_1$ is a transitive non-hyperbolic cantor set.

The Remark~\ref{ReM66} is also valid to theses cases.
%%%%%%%%%%%%%%%%%%%%%%%%%%%%%%%%%%%%%%%%%%%%
\begin{figure}[H]
\begin{center}\label{SemRerNor083.png}
\includegraphics[scale=.4]{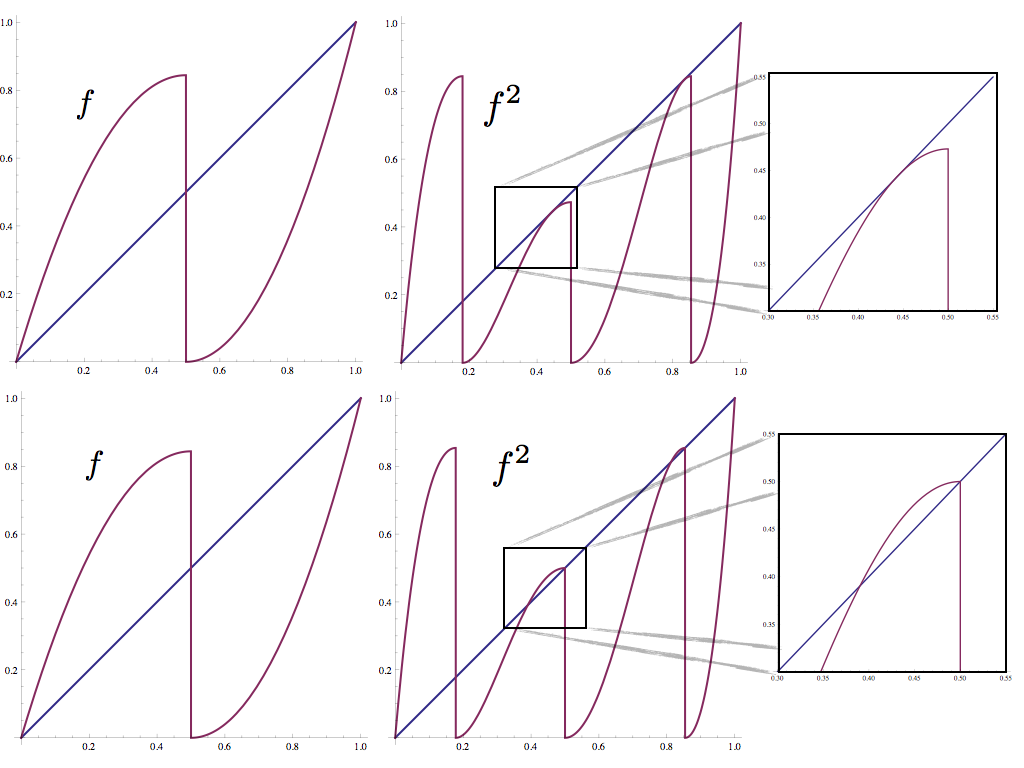}\\
Figure~\ref{SemRerNor083.png}
\end{center}
\end{figure}
%%%%%%%%%%%%%%%%%%%%%%%%%%%%%%%%%%%%%%%%%%%%%%
\end{enumerate}

\end{document}